\documentclass{amsart}

\usepackage{enumerate}
\usepackage{amssymb, amsmath, bbm}
\usepackage{mathrsfs}
\usepackage{amscd}
\usepackage[active]{srcltx}
\usepackage{verbatim}

\usepackage[colorlinks,linkcolor={blue},citecolor={blue},urlcolor={red},]{
hyperref}

\usepackage{newsymbol}

\allowdisplaybreaks

\let\mathcal \undefined
\def\mathcal{\mathscr}

\let\emptyset \undefined
\let\ge       \undefined
\let\le       \undefined
\newsymbol\le          1336  \let\leq\le
\newsymbol\ge          133E  \let\geq\ge
\newsymbol\emptyset    203F
\newsymbol\notle       230A
\newsymbol\notge       230B

\theoremstyle{plain}
\newtheorem{theorem}{Theorem}[section]
\newtheorem{corollary}[theorem]{Corollary}
\newtheorem{lemma}[theorem]{Lemma}
\newtheorem{proposition}[theorem]{Proposition}

\theoremstyle{remark}
\newtheorem{remark}[theorem]{Remark}
\newtheorem{definition}[theorem]{Definition}
\newtheorem{example}[theorem]{Example}

\numberwithin{equation}{section}

\def\N{{\mathbb N}}
\def\Z{{\mathbb Z}}
\def\R{{\mathbb R}}
\def\C{{\mathbb C}}

\newcommand{\E}{{\mathbb E}}
\renewcommand{\P}{{\mathbb P}}
\newcommand{\F}{{\mathcal F}}
\renewcommand{\H}{{\mathscr{H}}}
\newcommand{\G}{\mathscr{G}}

\renewcommand{\a}{\alpha}
\renewcommand{\b}{\beta}
\newcommand{\ga}{\gamma}

\renewcommand{\d}{\delta}
\newcommand{\e}{\varepsilon}

\newcommand{\om}{\omega}
\renewcommand{\O}{\Omega}
\newcommand{\Om}{\Omega}

\newcommand{\sgn}{{\rm sgn}}

\newcommand{\cT}{\mathscr{T}}

\newcommand{\beq}{\begin{equation}}
\newcommand{\eeq}{\end{equation}}
\newcommand{\bal}{\begin{aligned}}
\newcommand{\eal}{\end{aligned}}
\newcommand{\ben}{\begin{enumerate}}
\newcommand{\beni} {\begin{enumerate}[(i)]}
\newcommand{\een}{\end{enumerate}}
\newcommand{\bit}{\begin{itemize}}
\newcommand{\eit}{\end{itemize}}
\newcommand{\beqw}{\begin{equation*}}
\newcommand{\eeqw}{\end{equation*}}
\newcommand{\bthm}{\begin{theorem}}
\newcommand{\ethm}{\end{theorem}}
\newcommand{\bpr}{\begin{proposition}}
\newcommand{\epr}{\end{proposition}}
\newcommand{\ble}{\begin{lemma}}
\newcommand{\ele}{\end{lemma}}
\newcommand{\blem}{\begin{lemma}}
\newcommand{\elem}{\end{lemma}}
\newcommand{\bpf}{\begin{proof}}
\newcommand{\epf}{\end{proof}}
\newcommand{\bex}{\begin{example}}
\newcommand{\eex}{\end{example}}
\newcommand{\bre}{\begin{example}}
\newcommand{\ere}{\end{example}}

\newcommand{\bma}{\begin{bmatrix}}
\newcommand{\ema}{\end{bmatrix}}

\newcommand{\wh}{\widehat}
\renewcommand{\Re}{\hbox{\rm Re}}

\newcommand{\wt}{\widetilde}
\newcommand{\seqx}{(x_n)_{n\ge 1}}
\newcommand{\seqy}{(y_n)_{n\ge 1}}
\newcommand{\seqS}{(S_n)_{n\ge 1}}
\newcommand{\calA}{{\mathscr A}}
\newcommand{\calL}{{\mathscr L}}

\newcommand{\n}{\Vert}
\newcommand{\g}{\gamma}
\newcommand{\one}{{{\mathbbm 1}}}
\newcommand{\embed}{\hookrightarrow}
\newcommand{\s}{^*}
\newcommand{\lb}{\langle}
\newcommand{\rb}{\rangle}

\newcommand{\limn}{\lim_{n\to\infty}}
\newcommand{\limk}{\lim_{k\to\infty}}

\newcommand{\sumn}{\sum_{n\ge 1}}
\newcommand{\sumk}{\sum_{k\ge 1}}
\newcommand{\sumj}{\sum_{j\ge 1}}

\newcommand{\ot}{\otimes}
\newcommand{\ov}{\overline}

\begin{document}

\title{$\g$-Radonifying operators -- a survey}

\author{Jan van Neerven}
\address{Delft Institute of Applied Mathematics\\
Delft University of Technology\\ P.O. Box 5031\\ 2600 GA Delft\\The
Netherlands}
\email{J.M.A.M.vanNeerven@tudelft.nl}

\thanks{Support by VICI subsidy 639.033.604
of the Netherlands Organisation for Scientific Research (NWO) is gratefully
acknowledged}

\keywords{$\g$-Radonifying operators, stochastic integral, isonormal process,
Gaussian random variable, covariance domination, uniform tightness,
$K$-convexity, type and cotype}
  
\subjclass[2000]{Primary: 47B10; Secondary: 28C20, 46B09, 47B10, 60B11, 60H05}

\begin{abstract}
We present a survey of the theory of $\g$-radonifying operators and its
applications to stochastic integration in Banach spaces.
\end{abstract}

\date\today

\maketitle

\tableofcontents

\section{Introduction}

The theory of $\g$-radonifying operators can be traced back to the pioneering
works of {\sc Gel$'$fand} \cite{Gel}, {\sc Segal}, \cite{Seg}, {\sc Gross}
\cite{Gro62, Gro67}, who considered the following problem. A {\em cylindrical
Gaussian distribution} on a real Banach space $F$ is a bounded linear operator
$W: F\s\to L^2(\O)$, where  $F\s$ is the dual of $F$ and $(\O,\F,\P)$ is a
probability space.  
If $T$ is a bounded linear operator from $F$ into another real Banach space $E$,
then
$T$ maps $W$ to a cylindrical Gaussian distribution $T\circ W: E\s \to L^2(\O)$
by
$$ (T\circ W) x\s := W(T\s x\s), \quad x\s\in E\s.$$
The problem is to find criteria on $T$ which ensure that $T\circ W$ is {\em
Radon}. By this we mean that there exists a strongly measurable Gaussian random
variable $X\in L^2(\O;E)$ such that $$(T\circ W)x\s = \lb X,x\s\rb, \quad x\s\in
E\s$$ (the terminology ``Radon'' is explained by Proposition \ref{prop:tight}
and the remarks following it). 
The most interesting instance of this problem occurs when $F = H$ is a real
Hilbert space with inner product $[\cdot,\cdot]$ and $W: H\to L^2(\O)$ is an
{\em isonormal process}, i.e. a cylindrical Gaussian distribution satisfying
$$ \E W(h_1)W(h_2) = [h_1,h_2], \quad h_1,h_2\in H.$$   
Here we identify $H$ with its dual $H\s$ via the Riesz representation theorem.
A bounded operator $T:H\to E$ such that $T\circ W$ is 
Radon is called {\em $\g$-radonifying}. Here the adjective `$\g$-' stands for
`Gaussian'.

{\sc Gross} \cite{Gro62, Gro67} obtained a necessary and sufficient condition
for $\g$-radonification in terms of so-called measurable seminorms on $H$.
His result includes the classical result that a bounded operator from $H$ into a
Hilbert space $E$ is $\g$-radonifying if and only if it is Hilbert-Schmidt. 
These developments marked the birth of the theory of Gaussian distributions on
Banach spaces.  
The state-of the-art around 1975 is presented in the lecture notes by {\sc Kuo}
\cite{Kuo}.  
 
$\g$-Radonifying operators can be thought of as the Gaussian analogues of
$p$-absolutely summing operators. For a systematic exposition of this point of
view we refer to the lecture note by 
{\sc Badrikian} and {\sc Chevet} \cite{BadChe}, the monograph by {\sc Schwartz}
\cite{Schw-Radon} and the {\sc Maurey-Schwartz} seminar notes published between
1972 and 1976. More recent monographs include {\sc Bogachev} \cite{Bog},
{\sc Mushtari} \cite{Mus96}, and {\sc Vakhania, Tarieladze, Chob\-anyan}
\cite{VTC}.

In was soon realised that spaces of $\g$-radonifying operators provide a natural
tool for constructing a theory of stochastic integration in Banach spaces. This
idea, which goes back to a paper of {\sc Hoffman-J\o rgensen} and {\sc Pisier}
\cite{HofPis}, was first developed systematically in the Ph.D. thesis of {\sc
Neidhardt} \cite{Nei} in the context of $2$-uniformly smooth Banach spaces. His
results were taken up and further developed in a series of papers by {\sc
Dettweiler} (see \cite{Dett91} and the references given there) and subsequently
by {\sc Brze\'zniak} (see \cite{Brz95,Brz97}) who used the setting of martingale
type $2$ Banach spaces; this class of Banach spaces had been proved equal, up to
a renorming, to the class of $2$-uniformly smooth Banach spaces by {\sc Pisier}
\cite{Pis75}. The more general problem of radonification of cylindrical
semimartingales has been covered by {\sc Badrikian} and {\sc {\"U}st{\"u}nel}
\cite{BadUst}, {\sc Schwartz} \cite{Schw96} and {\sc Jakubowski, Kwapie{\'n},
Raynaud de Fitte, Rosi{\'n}ski} \cite{JKRR}.

If $E$ is a Hilbert space, then a strongly measurable function $f:\R_+\to E$ is
stochastically integrable with respect to Brownian motions $B$ if and only if
$f\in L^2(\R_+;E)$.
It had been known for a long time that functions in $L^2(\R_+;E)$ may fail to be
stochastically integrable with respect to $B$. The first simple counterexamples,
for $E=\ell^p$ with $1\le p<2$, were given by {\sc Yor} \cite{Yor}. {\sc
Rosi\'nski} and {\sc Suchanecki} \cite{RosSuc} (see also {\sc Rosi\'nski}
\cite{Ros84, Ros87}) were able to get around this by constructing a stochastic
integral of Pettis type for functions with valued in an arbitrary Banach space.
This integral was interpreted in the language of $\g$-radonifying operators by
{\sc van Neerven} and {\sc Weis} \cite{NeeWei05a}; some of the ideas in this
paper were already implicit in {\sc Brze\'zniak} and {\sc van Neerven}
\cite{BrzNee00}. The picture that emerged 
is that the space $\g(L^2(\R_+),E)$ of all $\g$-radonifying operators from
$L^2(\R_+)$ into $E$, rather than the Lebesgue-Bochner space $L^2(\R_+;E)$, is
the `correct' space of $E$-valued integrands for the stochastic integral with
respect to a Brownian motion $B$.  Indeed, the classical It\^o isometry extends
to the space $\g(L^2(\R_+),E)$ in the sense that
$$ \E\Big\n\int_0^\infty \phi\,dB\Big\n^2 = \n \wt\phi\n_{\g(L^2(\R_+),E)}^2$$
for all simple functions $\phi :\R_+\to H\ot E$; here $\wt\phi: L^2(\R_+) 
\to E$ is given by integration against $\phi$; on the level of elementary
tensors, the identification $\phi\mapsto \wt\phi$ is given by the identity
mapping
$f\ot x \mapsto f\ot x$. For Hilbert spaces,
this identification sets up an isomorphism 
$$ L^2(\R_+;E) \eqsim \g(L^2(\R_+),E).$$
In the converse direction, if the identity mapping
$f\ot x \mapsto f\ot x$ extends to an isomorphism $ L^2(\R_+;E) \simeq
\g(L^2(\R_+),E)$, then $E$ has both type $2$ and cotype $2$, so $E$ is
isomorphic to a Hilbert space by a classical result of {\sc Kwapie\'n}
\cite{Kwa72}. 

Interpreting $B$ as an isonormal process $W:L^2(\R_+)\to L^2(\O)$ by putting
\begin{align}\label{eq:BM} 
W(f) := \int_0^\infty f\,dB,
\end{align}
this brings us back to the question originally studied by {\sc Gross}. However,
instead of thinking of an operator $T_\phi: L^2(\R_+)\to E$ as `acting' on the
isonormal process $W$, we now think of $W$ as `acting' on $T_\phi$ as an
`integrator'. This suggests an abstract approach to $E$-valued stochastic
integration, where the `integrator' is an arbitrary isonormal processes $W: H\to
L^2(\Omega)$, with $H$ an abstract Hilbert space, and the `integrand' is a
$\g$-radonifying operator from $H$ to $E$. For finite rank operators $T =
\sum_{n=1}^N h\ot x$ the stochastic integral with respect to $W$ is then given
by
$$ W\Big(\sum_{n=1}^N h\ot x\Big) := \sum_{n=1}^N W(h)\ot x.$$
In the special case $H=L^2(\R_+)$ and $W$ given by a standard Brownian motion
through \eqref{eq:BM}, this is easily seen to be consistent with the classical
definition of the stochastic integral. 

This idea will be worked out in detail. This paper contains no new results; the
novelty is rather in the organisation of the material and the abstract point of
view. Neither have we tried to give credits to many results which are more or
less part of the folklore of the subject. This would be difficult, since theory
of $\g$-radonifying operators has changed face many times. Results that are 
presented here as theorems may have been taken as definitions in previous works
and vice versa, and many results have been proved and reproved in apparently
different but essentially equivalent formulations by different authors. Instead,
we hope that the references given in this introduction serves as a guide for the
interested reader who wants to unravel the history of the subject. For the
reasons just mentioned we have decided to present full proofs, hoping that this
will make the subject more accessible.

The emphasis in this paper is on $\g$-radonifying operators rather than on
stochastic integrals. Accordingly we shall only discuss stochastic integrals of
deterministic functions. The approach taken here extends to 
stochastic integrals of stochastic processes if the underlying Banach space is a
so-called UMD space by following the lines of {\sc van Neerven, Veraar, Weis}
\cite{NVW07a}.  
We should mention that various alternative approaches to stochastic integration
in general Banach spaces exist, among them the vector measure approach of {\sc
Brooks} and {\sc Dinculeanu} \cite{BroDin} and {\sc Dinculeanu} \cite{Din}, and
the Dol\'eans measure approach of {\sc Metivier} and {\sc Pellaumail}
\cite{MetPel}. As we see it, the virtue of the approach presented here is that
it is tailor-made for applications to stochastic PDEs; see, e.g., {\sc
Brze\'zniak} \cite{Brz95, Brz97}, {\sc Da Prato} and {\sc Zabczyk}
\cite{DaPZab}, {\sc van Neerven, Veraar, Weis} \cite{NVW09, NVW08} and the
references therein. For an introduction to these applications we refer to the
author's 2007/08 Internet Seminar lecture notes \cite{ISEM}.

Let us finally mention that the applicability of radonifying operators is by no
means limited to vector-valued stochastic integration. Radonifying norms have
been used, under the guise of $l$-norms, in the local theory of Banach space for
many years; see e.g. {\sc Diestel, Jarchow, Tonge} \cite{DJT}, {\sc Kalton} and
{\sc Weis} \cite{KalWei07b}, {\sc Pisier} \cite{Pis89}, {\sc Tomczak-Jaegermann}
\cite {TomJae}. In harmonic analysis, $\g$-radonifying norms are the natural
generalisation of the square functions arising in connection with
Littlewood-Paley theory (see e.g. {\sc Stein} \cite{Ste70})
and were used as such in {\sc Kalton} and {\sc Weis} \cite{KalWei07}, {\sc
Hyt\"onen} \cite{Hyt07}, 
{\sc Hyt\"onen, McIntosh, Portal} \cite{HMP08}, and {\sc Hyt\"onen, van Neerven,
Portal} \cite{HNP}. Further applications have appeared in interpolation theory,
see {\sc Kalton, Kunstmann, Weis} \cite{KKW} and {\sc Su\'arez} and {\sc Weis}
\cite{SuaWei}, control theory, see {\sc Haak} and {\sc Kunstmann}
\cite{HaakKun},
and in image processing, see {\sc Kaiser} and {\sc Weis} \cite{KaiWei}. This
list far from being complete. 

This paper is loosely based on the lectures presented at the 2009 workshop on
Spectral Theory and Harmonic Analysis held at the Australian National University
in Canberra. It is a pleasure to thank the organisers Andrew Hassell and Alan
McIntosh for making this workshop into such a success.

\medskip\noindent
{\bf Notation.} \ 
Throughout these notes, we use the symbols $H$ and $E$ to denote real Hilbert
spaces and real Banach spaces, respectively. The inner product of a Hilbert
space $H$ will be denoted by $[\cdot,\cdot]_H$ or, if no confusion can arise, by
$[\cdot,\cdot]$. We will always identify $H$ with its dual via the Riesz
representation theorem.
The duality pairing between a Banach space $E$ and its dual $E\s$ will be
denoted by $\lb \cdot,\cdot\rb_{E,E\s}$ or simply $\lb \cdot,\cdot\rb.$
The space of all bounded linear operators from a Banach space $E$ into another
Banach space $F$ is denoted by $\calL(E,F)$. The word `operator' always means
`bounded linear operator'. 

\section{Banach space-valued random variables}\label{sec:preliminaries}

Let $(A,\calA,\mu)$ be a $\sigma$-finite measure space and $E$ a Banach space.
A function $f:A\to E$ is called {\em simple} if it is a finite linear
combination of functions of the form $\one_B\ot x$ with $B\in\calA$ of finite
$\mu$-measure and $x\in E$,
and \emph{strongly measurable}
if there exists a sequence of simple functions $f_n:A\to E$
such that $\limn f_n = f$ pointwise almost surely.
By the Pettis measurability theorem, $f$ is strongly measurable if and only if 
$f$ is {\em essentially separably valued} (which means that there exists a null
set $N\in\calA$ and a separable closed subspace $E_0$ of $E$ such that
$f(\xi)\in E_0$ for all $\xi\not\in N$) and {\em weakly measurable} (which means
that $\lb f,x\s\rb$ is measurable for all $x\s\in E\s$).

When $(\O,\F,\P)$ is a probability space, strongly measurable functions $f:\O\to
E$ are called {\em random variables}. Standard probabilistic notions such as
independence and symmetry carry over to the $E$-valued case in an obvious way.
Following tradition in the probability literature, random variables will be
denoted by the letter $X$ rather than by $f$.
The {\em distribution} of an $E$-valued random variable $X$ is the 
Borel probability measure $\mu_X$ on $E$ defined by
$$\mu_X(B) := \P\{X\in B\}, \qquad B\in\mathscr{B}(E).$$
The set $\{X\in B\}:= \{\om\in\O: \ X(\om)\in B\}$ may not belong to $\F$, but
there always exists a set $F\in \F$ such that the symmetric difference $F\Delta
\{X\in B\}$ is contained in a null set in $\F$, and therefore the measure
$\mu_X$ is well-defined.
 
For later use we collect some classical facts concerning $E$-valued random
variables. Proofs, further results, and references to the literature can be
found in 
{\sc Albiac} and {\sc Kalton} \cite{AlbKal}, {\sc Diestel, Jarchow, Tonge}
\cite{DJT}, {\sc Kwapie\'n} and {\sc Woyczy{\'n}ski} \cite{KwaWoy}, {\sc Ledoux}
and {\sc Talagrand} \cite{LedTal}, and {\sc Vakhania, Tarieladze, Chob\-anyan}
\cite{VTC}.

The first result states that $E$-valued random variables 
are {\em tight}: 

\begin{proposition}\label{prop:tight}
If $X$ is a $E$-valued random variable, then for every $\e>0$ there exists a
compact
set $K$ in $E$ such that $\P\{X\not\in K\} < \e.$
\end{proposition}
\begin{proof}
Since $X$ is separably valued outside some null set, we may assume that $E$ is
separable.
Let $\seqx$ be a dense sequence in $E$ and fix $\e>0$. 
For each integer $k\ge 1$ 
the closed balls $B(x_n,\frac1k)$
cover $E$, and therefore there exists an index $N_k\ge 1$ such that
$$\P\Big\{X\in \bigcup_{n=1}^{N_k} B\big(x_n,\tfrac1{k}\big)\Big\} >
1-\frac{\e}{2^{k}}.$$  
The set 
$ K := \bigcap_{k\ge 1} \bigcup_{n=1}^{N_k} B\big(x_n,\frac1{k}\big)$
is closed and totally bounded. Since $E$ is complete, 
$K$ is compact.
Moreover, $\P\{ X\not\in K\} < \sum_{k\ge 1} 2^{-k}\e = \e.$
\end{proof}
 
This result implies that the distribution $\mu_X$ is a {\em Radon measure}, i.e.
for all $B\in\mathscr{B}(E)$ and $\e>0$ there exists a compact set $K\subseteq
B$ such that $\mu_X(B\setminus K)<\e$. Indeed, the proposition allows us to
choose a compact subset $C$ of $E$ such that $\mu_X(C)>1-\frac12\e$, and by the
inner regularity of Borel measures on complete separable metric spaces there is
a closed set $F\subseteq B$ with $\mu(B\setminus F)<\frac{1}{2}\e$. The set
$K=C\cap F$ has the desired properties. Conversely, every Radon measure $\mu$ on
$E$ is the distribution of the random variable $X(x)= x$ on the probability
space $(E,\mathscr{B}(E),\mu)$. 
 
Motivated by the above proposition, a family $\mathscr{X}$ of $E$-valued random
variables is called {\em uniformly tight} if for every $\e>0$ there exists a
compact
set $K$ in $E$ such that $\P\{X\not\in K\} < \e$ for all $X\in \mathscr{X}$.

A sequence of $E$-valued random variables $(X_n)_{n\ge 1}$
is said to {\em converge in distribution} to an $E$-valued random variable $X$
if 
$\limn \E f(X_n) = \E f(X)$
for all $f\in C_{\rm b}(E)$, the space of all bounded continuous functions $f$
on $E$. 

\begin{proposition}[{\sc Prokhorov}'s theorem]
For a family $\mathscr{X}$ of $E$-valued random variables the following
assertions are equivalent:
\begin{enumerate}
\item[\rm(1)] $\mathscr{X}$ is uniformly tight; 
\item[\rm(2)] every
sequence in $\mathscr{X}$ has a subsequence which converges in distribution.
\end{enumerate}
\end{proposition}

Excellent accounts of this result and its ramifications can be found in {\sc
Billingsley} \cite{Bil} and {\sc Parathasarathy} \cite{Par}.

We continue with a maximal inequality.

\begin{proposition}[{\sc L\'evy}'s inequality]\label{prop:Levy}
 Let
$X_1,\dots,X_N$ be independent symmetric $E$-valued random variables,
and put $S_n:=\sum_{j=1}^n X_j$ for $n=1,\dots,N$.
Then for all $r\ge 0$ we have
$$\P\big\{\max_{1\le n\le N}\n S_n\n > r\big\}\le 2\P\{\n S_N\n> r\}.$$
\end{proposition}

This inequality will be used in Section \ref{sec:HJ}. It is also the main
ingredient of a theorem of {\sc It\^o} and {\sc Nisio}, presented here only in
its simplest formulation which goes back to {\sc L\'evy}.

\begin{proposition}[{\sc L\'evy, It\^o-Nisio}]\label{prop:IN1}
Let
$(X_n)_{n\ge 1}$ be a sequence of independent symmetric $E$-valued random
variables,
and put $S_n:=\sum_{j=1}^n X_j$ for $n\ge 1$. 
The following assertions are equivalent:
\ben
\item[\rm(1)]
the sequence $(S_n)_{n\ge 1}$ converges in probability;
\item[\rm(2)]
the sequence $(S_n)_{n\ge 1}$ converges almost surely.
\een
\end{proposition}

Let
$(x_i)_{i\in I}$ be a family of elements of a Banach space $E$, indexed by a set
$I$. The sum $\sum_{i\in I} x_i$ is {\em summable} to an element $s\in E$ if for
all
$\e>0$ there is a finite subset $J\subseteq I$ such that for all finite subsets
$J'\subseteq I$ containing $J$ we have
$$\Big\n s - \sum_{j\in J'}x_j\Big\n<\e.$$
Stated differently, this means that $\lim_J s_J = s$, where 
$s_J := \sum_{j\in J} x_j$ and the limit is taken along the net of all finite
subsets $J\subseteq I$. 

As we shall see in Example \ref{ex:1}, this summability method adequately
captures the convergence of coordinate expansions with respect to arbitrary
maximal orthonormal systems in Hilbert spaces. 
For countable index sets $I$, summability is equivalent to unconditional
convergence. The `only if' part is clear, and the `if' part can be seen as
follows. Suppose, for a contradiction, that $\sum_{i\in I}x_n = s$
unconditionally while $\sum_{i\in I}x_n$ is not summable to $s.$  Let $I =
(i_n)_{n\ge 1}$ be an enumeration. There is an $\e>0$ and an increasing sequence
$J_1\subseteq J_2\subseteq \dots$ of finite subsets of $I$ such that
$\{i_1,\dots,i_k\}\subseteq J_k$ and $\n s - s_{J_k}\n\ge \e$. Clearly
$\bigcup_{k\ge 1}J_k = I$. If $I = (i_n')_{n\ge 1}$ is an enumeration with the
property that $J_k = \{i_1',\dots, i_{N_k}'\}$ for all $k\ge 1$ and suitable
$N_1\le N_2\le \dots$, the sum $\sum_{n\ge 1} x_{i_n'}$ fails to converge to
$s$. This contradicts the unconditional convergence of the sum $\sum_{i\in I}
x_i $ to $s$.

Convergence of sums of random variables in in $L^p(\O;E)$ has been investigated
systematically by {\sc Hoffmann-J\o rgensen} \cite{HofJor}. Here we only need
the following prototypical result:

\begin{proposition}\label{prop:IN}
Let $1\le p<\infty$, let $(X_i)_{i\in I}$ be an indexed family of independent
and symmetric random variables in $L^p(\O;E)$
and let $S\in L^p(\O;E)$.
The following assertions are equivalent:
\ben
\item[\rm(1)] $\sum_{i\in I} X_i$ is summable to $S$ in $L^p(\O;E)$
\item[\rm(2)] $\sum_{i\in I} \lb X_i,x\s\rb$ is summable to $\lb S,x\s\rb$ in
$L^p(\O)$ for all $x\s\in E\s$.
\een
\end{proposition}

\begin{proof}
We only need to prove the implication (2)$\Rightarrow$(1).

Let $[I]$ denote the collection of all finite subsets of $I$. For $J\in [I]$ set
$S_J := \sum_{j\in J} X_j$.
From (2) it easily follows that for all $J\in [I]$ and $x\s\in E\s$ 
the random variables $\lb S_J,x\s\rb$ and
$\lb S -S_{J}, x\s\rb$  are independent. If we denote by $\F_J$ the
$\sigma$-algebra generated by $\{X_j:\, j\in J\}$, for all $x\s\in E\s$ it
follows that
$$ \lb \E (S|\F_J), x\s\rb  = \E (\lb S,x\s\rb|\F_J) = \lb S_J,x\s\rb$$
in $L^p(\O)$. As a consequence, $$\E (S|\F_J) = S_J$$
in $L^p(\O;E)$.
Now (1) follows from the elementary version of the $E$-valued martingale
convergence theorem (see {\sc Diestel} and {\sc Uhl} \cite[Corollary
V.2]{DieUhl}).
\end{proof}

We continue with a useful comparison result for Rademacher sequences and
Gaussian sequences.
Recall that a {\em Rademacher sequence} is a sequence of independent random
variables taking the values $\pm 1$ with probability $\frac12$. A {\em Gaussian
sequence} is a sequence of independent real-valued standard Gaussian random
variables. 

\begin{proposition}\label{prop:Rad-phi} Let $(r_n)_{n\ge 1}$ be a Rademacher
sequence and 
$(\g_n)_{n\ge 1}$ a Gaussian sequence.
\begin{enumerate}
\item[\rm(1)] For all $1\le p<\infty$ and all finite sequences $x_1,\dots,x_N\in
E$ we have
$$ \E \Big\n \sum_{n=1}^N r_n x_n\Big\n^p \le
(\tfrac12 \pi)^\frac{p}{2} \E \Big\n \sum_{n=1}^N \g_n x_n\Big\n^p.
$$
\item[\rm(2)] If $E$ has finite cotype, then for all $1\le p<\infty$ there
exists a constant $C_{p,E}\ge 0$ such that 
for all finite sequences $x_1,\dots,x_N\in E$ 
 we have
$$ \E \Big\n \sum_{n=1}^N \g_n x_n\Big\n^p \le
C_{p,E}^p\E \Big\n \sum_{n=1}^N r_n x_n\Big\n^p.
$$
\end{enumerate}
\end{proposition}

For the definition of cotype we refer to Section \ref{sec:embedding}. 
We will only need part (1) which is an elementary consequence of the Kahane
contraction principle (see {\sc Kahane} \cite{Kah}) and the fact that the
sequences $(\g_n)_{n\ge 1}$ and $(r_n |\g_n|)_{n\ge 1}$ are identically
distributed when $(r_n)_{n\ge 1}$ is independent of $(\g_n)_{n\ge 1}$. 

We finish this section with the so-called Kahane-Khintchine inequalities.

\begin{proposition}[Kahane-Khintchine inequalities]\label{prop:KK-R}
Let $(r_n)_{n\ge 1}$ be a Rademacher sequence and $(\g_n)_{n\ge 1}$ a Gaussian 
sequence.
\begin{enumerate}
\item[\rm(1)]For all $1\le p,q< \infty$ there exists a constant $C_{p,q}$,
depending only on $p$ and $q$, such that for all finite sequences
$x_1,\dots,x_N\in E$ we have
$$ \Big(\E\Big\n\sum_{n=1}^N r_n x_n\Big\n^p\Big)^\frac1p 
\le C_{p,q} \Big(\E\Big\n\sum_{n=1}^N r_n x_n\Big\n^q\Big)^\frac1q.
$$
\item[\rm(2)]For all $1\le p,q< \infty$ there exists a constant $C_{p,q}^\g$,
depending only on $p$ and $q$, such that for all finite sequences
$x_1,\dots,x_N\in E$ we have
$$ \Big(\E\Big\n\sum_{n=1}^N \g_n x_n\Big\n^p\Big)^\frac1p 
\le C_{p,q}^\g \Big(\E\Big\n\sum_{n=1}^N \g_n x_n\Big\n^q\Big)^\frac1q.
$$
\end{enumerate}
\end{proposition}

The least admissible constants in these inequalities are called the {\em
Kahane-Khint\-chine constants} and are usually denoted by $K_{p,q}$ and
$K_{p,q}^\g$. 
Note that $K_{p,q}=1$ if $p\le q$ by H\"older's
inequality. It was shown by 
{\sc Lata{\l}a} and {\sc Oleszkiewicz} \cite{LatOle} 
that $K_{2,1} = \sqrt{2}$.

Part (2) of the proposition can be deduced from part (1) by a central limit
theorem argument (which can be justified by Lemma \ref{lem:sq-w-conv} below);
this gives the inequality $K_{p,q}^\g \le K_{p,q}$.

\section{$\g$-Radonifying operators}

After these preparations we are ready to introduce the main object of study, the
class of $\g$-radonifying operators.
Throughout this section $H$ is a real Hilbert space and $E$ is a real Banach
space. Gaussian random variables are always assumed to be centred.

\begin{definition}\label{def:isonormal}
An {\em $H$-isonormal process} on a probability space $(\O,\F,\P)$ is a mapping
$W: H \to
L^2(\O)$
with the following properties:
\ben
\item[\rm(i)] for all $h\in H$ the random
variable $W(h)$ is Gaussian;
\item[\rm(ii)] for all $h_1,h_2\in H$ we have $\E W(h_1)W(h_2) = [h_1,h_2].$
\een
\end{definition}

Isonormal processes lie at the basis of Malliavin calculus. We refer to {\sc
Nualart} \cite{Nua} for an introduction to this subject. We shall use isonormal
processes to set up an abstract version of the vector-valued It\^o stochastic
integral. 
As the scalar It\^o stochastic integral arises naturally within Malliavin
calculus, the theory developed below serves as a natural starting point for
setting up a vector-valued Malliavin calculus. This idea is taken up in {\sc
Maas}  \cite{Maa09} and {\sc Maas} and {\sc van Neerven} \cite{MaaNee08}.

We turn to some elementary properties of isonormal processes. 
From (ii) we have $$\E |W(c_1 h_1 + c_2 h_2) - (c_1 W(h_1) + c_2 W(h_2))|^2 =
0,$$ which shows that $W$ is linear. As a consequence, for all $h_1,\dots,h_N\in
H $ the random variables $W(h_1),\dots, W(h_N)$ are jointly Gaussian (which
means that every linear combination is Gaussian as well). Recalling that jointly
Gaussian random variables are independent if and only if they are uncorrelated,
another application of (ii) shows that $W(h_1),\dots, W(h_N)$ are independent if
and only if $h_1,\dots,h_N\in H$ are orthogonal.

\begin{example}\label{ex:1}
Let $H$ be a Hilbert space with maximal orthonormal system $(h_i)_{i\in I}$
and let $(\g_i)_{i\in I}$ be a family of independent standard Gaussian random
variables with the same index set. Then for all $h\in H$, $\sum_{i\in I} \g_i
[h,h_i]$ is summable in $L^2(\Om)$ and
$$ W(h):= \sum_{i\in I} \g_i [h,h_i], \quad h\in H,$$
defines an $ H $-isonormal process.
To see this let $h\in H$ be fixed. Given $\e>0$ choose indices $i_1,\dots,i_N\in
I$ such that
$$ \Big\n h - \sum_{n=1}^N [h,h_{i_n}]h_{i_n}\Big\n<\e.$$
For any finite set $J'\subseteq I$ containing $i_1,\dots,i_N$ we then have, by
the Pythagorean theorem,
$$  \Big\n h - \sum_{j\in J'} [h,h_{j}]h_{j}\Big\n <\e. $$
This implies that $\sum_{i\in I} [h,h_i]h_i$ is summable to $h$. Since $W$
clearly defines an isometric linear mapping from the linear span of $(h_i)_{i\in
I}$ into $L^2(\Om)$ satisfying $W(h_i)=\g_i$, $\sum_{i\in I} \g_i[h,h_i]$ is
summable in $L^2(\Om)$. Denoting its limit by $W(h)$, the easy proof that the
resulting linear map $W:H\to L^2(\Om)$ is isonormal is left to the reader.
\end{example}

\begin{example}\label{ex:2}
If $B$ is a standard Brownian motion, then the It\^o stochastic integral  
$$ W(h) := \int_0^\infty h\,dB,\quad h\in L^2(\R_+),$$
defines an $L^2(\R_+)$-isonormal process $W$. 
Conversely, if $W$ is an $L^2(\R_+)$-isonormal process, then $$B(t):=
W(\one_{[0,t]}), \quad t\ge 0,$$ 
is a standard Brownian motion. Indeed, this process is Gaussian
and satisfies $\E B(s)B(t)  = [\one_{(0,s)},\one_{(0,t)}]_{L^2(\R_+)} = s\wedge
t$ for all $s,t\ge 0$.
\end{example}

\begin{example}
Let $B$ be a Brownian motion with values in a Banach space $E$ and let $\H$ be
the closed linear span in $L^2(\O)$ spanned by the random variables
$\lb B(1),x\s\rb$, $x\s\in E\s$. Then $B$ induces an $L^2(\R_+;\H)$-isonormal
process by putting
$$ W(f\ot \lb B(1),x\s\rb) := \int_0^\infty f\,d\lb B,x\s\rb, \quad f\in
L^2(\R_+), \ x\s\in E\s.$$
To see this, note that since $\lb B,x\s\rb$ is a real-valued Brownian motion,
$$\E|\lb B(t),x\s\rb|^2 = t\E|\lb B(1),x\s\rb|^2$$
for all $t\ge 0$. Hence by normalising the Brownian motions $\lb B,x\s\rb$, the
It\^o isometry gives
$$
\bal
\E W(f\ot \lb B(1),x\s\rb)W(g\ot \lb B(1),y\s\rb)
& = \E\lb B(1),x\s\rb\lb B(1),y\s\rb [f,g]_{L^2(\R_+)} 
\\ & = [f\ot \lb B(1),x\s\rb,g\ot \lb B(1),x\s\rb]_{L^2(\R_+;\H)}.
\eal$$

\end{example}

\begin{remark} In many papers, $\H$-cylindrical Brownian motions  are defined
as a family $W = (W(t))_{t\ge 0}$ of bounded linear operators from $\H$ to
$L^2(\O)$ with the following properties:
\begin{enumerate}
\item[\rm(i)] for all $h\in \H$, the process $(W(t)h)_{t\ge 0}$ is a Brownian
motion;
\item[\rm(ii)] for all $t_1,t_2\ge 0$ and $h_1,h_2\in \H$ we have
$$ \E (W(t_1)h_1\cdot W(t_2)h_2) = (t_1\wedge t_2)[h_1,h_2].$$
\end{enumerate}
Subsequent arguments frequently use that the family 
$\{W(t)h: \, t\ge 0, \, h\in \H\}$ is jointly Gaussian, something that is not
obvious from (i) and (ii). If we add this as an additional assumption, then
every $\H$-cylindrical Brownian motion defines an $L^2(\R_+;\H)$-isonormal
process in a natural way and vice versa.

In the special case $\H=L^2(D)$, where $D$ is a domain in $\R^d$,
$L^2(D)$-cylindrical Brownian motions provide the rigorous mathematical model of
{\em space-time white noise} on $D$. 
\end{remark}

In what follows, $W: H \to L^2(\O)$ will always denote a fixed $H$-isonormal
process. 
For any Banach space $E$, $W$ induces a linear mapping
from $ H \otimes E$ to $L^2(\O)\otimes E$, also denoted by $W$, 
by putting
$$ W(h\otimes x) := W(h)\otimes x$$
and extending this definition by linearity. The problem we want to address is
whether there is a norm on
$ H \otimes E$ turning $W$ into a {\em bounded} operator from $ H \otimes E$
into
$L^2(\O;E)$. 

\begin{example}\label{ex:SI}
Let $B$ be a Brownian motion and let 
$W:L^2(\R_+)\to L^2(\Om)$ be the associated isonormal process. 
Identifying $E$-valued step functions
with elements in $L^2(\R_+)\otimes E$ we have
$$W(\one_{(a,b)}\ot x) = \int_0^\infty \one_{(a,b)}\ot x\,dB.$$
Thus, $W$ can be viewed as an $E$-valued extension of the stochastic integral
with respect to $B$.
In the same way, for isonormal processes
$W:L^2(\R_+;\H)\to L^2(\Om)$ we have  
$$ W(\one_{(a,b)}\ot h)\ot x)  = \int_0^\infty \one_{(a,b)}\otimes (h\ot x)
\,dW,$$ where the right-hand side is the side the stochastic integral for 
$\H\ot E$-valued step functions with respect to $\H$-cylindrical Brownian
motions introduced in {\sc van Neerven} and {\sc Weis} \cite{NeeWei05a}.
\end{example}

Suppose an element in $ H \otimes E$ of the form
$\sum_{n=1}^N h_n\otimes x_n$ is given with $h_1,\dots,h_N$ orthonormal in $H$.
Then the random variables $W(h_1),\dots,W(h_N)$ are independent and standard
Gaussian and therefore
$$\E \Big\n \sum_{n=1}^N W(h_n)\otimes x_n\Big\n^2 
=  \E \Big\n \sum_{n=1}^N \g_n x_n\Big\n^2,
$$
where $(\g_n)_{n=1}^N$ is any Gaussian sequence.
The right-hand side is independent of the representation of 
the element in $ H \otimes E$ as a finite sum $\sum_{n=1}^N h_n\otimes x_n$ as
long as we choose the vectors $h_1,\dots,h_N$ orthonormal in $ H $.
Indeed, suppose we have a second representation
$$\sum_{n=1}^N h_n\otimes x_n = \sum_{m=1}^M h_m'\otimes x_m',$$ where the
vectors $h_1',\dots,h_M'$ are
again orthonormal in $ H $.
There is no loss in generality if we assume that the sequences
$(h_n)_{n=1}^N$ and 
$(h_m')_{m=1}^M$ span the same finite-dimensional subspace $G$ of
$ H $. In fact we may
consider the linear span of the set $\{h_1,\dots,h_N,h_1',\dots,h_M'\}$ and
complete both sequences to orthonormal bases,  
say $(h_k)_{k=1}^K$ and $(h_k')_{k=1}^K$, for this linear span. Then
we may write 
$$ \sum_{k=1}^K h_k\otimes x_k = \sum_{k=1}^K h_k'\otimes x_k'$$
with $x_k=0$ for $k=N+1,\dots, K$ and $x_m'=0$ for $k=M+1,\dots,K$.
Under this assumption, we
have $M=N=K$ and there is an orthogonal transformation $O$ on $G$
such that $Oh_k' = h_k$ for all $k=1,\dots,K$.
Then
$$x_k =    \sum_{j=1}^K [h_j',h_k] x_j'
=  \sum_{j=1}^K[Oh_j,h_k] x_j'.$$
Let $O = (o_{jk})$ denote the matrix representation with respect to the basis
$(h_k)_{k=1}^K$. Then,
$$ \E \Big\n \sum_{k=1}^K \g_k x_k\Big\n^2
= \E \Big\n \sum_{k=1}^K \g_k \sum_{j=1}^K o_{jk} x_j'\Big\n^2
= \E \Big\n\sum_{j=1}^K \Big(\sum_{k=1}^K   o_{jk} \g_k\Big) x_j'\Big\n^2
= \E \Big\n\sum_{j=1}^K \g_j' x_j'\Big\n^2,
$$
where
$ \g_j' : = \sum_{k=1}^K   o_{jk}\g_k.$
Writing $\g = (\g_1,\dots,\g_K)$ and $\g' = (\g_1',\dots,\g_K')$,
this means that $$\g'= O\g.$$ As $\R^d$-valued Gaussian random variables, 
$\g$ and $\g'$ have covariance matrices $I$ (by assumption) and $OIO^* = I$
(since $O$ is orthogonal), respectively. Stated differently, the random
variables $\g_j'$ form a standard Gaussian sequence, and thereby we have proved
the asserted  well-definedness. 

\begin{definition}\label{def:g-rad}
The Banach space $\g( H ,E)$ is defined as the completion of $ H \otimes E$ with
respect to
the norm
$$ \Big\n \sum_{n=1}^N h_n\otimes x_n\Big\n_{\g(H,E)}^2 
:=  \E \Big\n \sum_{n=1}^N \g_n x_n\Big\n^2,
$$
where it is assumed that $h_1,\dots,h_N$ are orthonormal in $ H $.
\end{definition}

The following example is used frequently in the context of stochastic integrals,
where $H_1 = L^2(\R_+)$ and $H_2=\H$ is some abstract Hilbert space:
\begin{example}\label{ex:f-ot-T} 
Let $H_1$ and $H_2$ be Hilbert spaces and let $H_1\wh\ot H_2$ denote the Hilbert
space completion of their tensor product. Then for all
$h\in H_1$ and $h_1,\dots,h_N\in H_2$, $x_1,\dots,x_N\in E$, 
$$ \Big\n \sum_{n=1}^N (h \ot h_n) \ot x_n \Big\n_{\g(H_1\wh\ot H_2, E)} = \n
h\n_{H_1} \Big\n \sum_{n=1}^N h_n \ot x_n\Big\n_{\g(H_2,E)}.$$
\end{example} 

The preceding discussion can be summarized as follows.

\begin{proposition}[It\^o isometry]\label{prop:Ito}
Every isonormal process $W:  H \to L^2(\O)$ induces an isometry, also denoted by
$W$, from
$\g( H ,E)$ into $L^2(\O;E)$.
\end{proposition}

For $H = L^2(\R_+;\H)$ this result reduces to the It\^o isometry for the
stochastic integral with respect to $\H$-cylindrical Brownian motions of {\sc
van Neerven} and {\sc Weis} \cite{NeeWei05a}.

We continue with some elementary mapping properties of the spaces $\g(H,E)$.
The first is an immediate consequence of Definition \ref{def:g-rad}.

\begin{proposition}\label{prop:incl}
Let $H_0$ be a closed subspace of $H$. The inclusion mapping
$i_0: H_0\to H$ induces an isometric embedding $i_0:\g(H_0,E)\to \g(H,E)$ by
setting $$i_0(h_0\ot x) := i_0 h_0\ot x.$$
\end{proposition}

The next proposition is in some sense the dual version of this result:

\begin{proposition}[Composition with orthogonal projections]\label{prop:proj}
Let $H_0$
be a closed subspace of $H$. Let $P_0$ be the orthogonal projection in $H$ onto
$H_0$ and let $\E_0$ denote the conditional expectation operator with respect to
the $\sigma$-algebra $\F_0$ generated by the family of random variables
$\{W(h_0): \ h_0\in H_0\}$. The operator $P_0$ extends to a surjective linear
contraction  $P_0:\g(H,E)\to \g(H_0,E)$ by setting $$P_0(h\ot x):= P_0 h\ot x$$
and the following diagram commutes:
$$
 \begin{CD}
    \g( H,E)    @> W >> L^2(\Om;E)    \\
     @V P_0 VV                @V \E_0 VV \\
     \g( H_0 ,E)  @> W >>   L^2(\Om,\F_0;E)
     \end{CD}
$$
\end{proposition}
\begin{proof}
For $h\in H_0$ we have $\E_0 W(h) = W(h) = W(P_0 h)$.
For $h\perp H_0$, the random variable $W(h)$ is independent of
$\{W(h_1),\dots,W(h_N)\}$ for all
$h_1,\dots,h_N\in H_0$, and therefore $W(h)$ is independent of $\F_0$.
Hence, $$\E_0 W(h) = \E W(h) = 0 = W(0) = W(P_0(h)).$$ This proves the
commutativity of the diagram 
$$
 \begin{CD}
    H    @> W >> L^2(\Om)    \\
     @V P_0 VV                @V \E_0 VV \\
    H_0  @> W >>   L^2(\Om,\F_0)
     \end{CD}
$$
For elementary tensors $h\ot x\in H\ot E$ it follows that 
$$\E_0 W(h\ot x) = \E_0 W(h)\otimes x = W(P_0 h)\otimes x_n = W(P_0(h\ot x)).$$
By linearity, this proves that the $E$-valued diagram commutes as well. 
That $P_0$ extends to a linear contraction from $\g(H,E)$ to $\g(H_0,E)$ now
follows from the facts that $\E_0$ is a contraction from $L^2(\O;E)$ to
$L^2(\O,\F_0;E)$ and  both $W:\g(H,E)\to L^2(\Om;E)$ and 
$W:\g(H_0,E)\to L^2(\Om,\F_0;E)$ are isometric embeddings.
The surjectivity of $P_0$ follows from
the surjectivity of $\E_0$.
\end{proof}

\begin{proposition}[Composition with functionals]\label{prop:xs}
Every functional $x\s\in E\s$ extends to a bounded operator $x\s: \g(H,E) \to H$
by setting $$x\s(h\ot x):= \lb x,x\s\rb h$$ and the following diagram commutes:
$$
 \begin{CD}
    \g( H,E)    @> W >> L^2(\Om;E)    \\
     @V x\s VV                @V x\s VV \\
     H  @> W >>   L^2(\Om)
     \end{CD}
$$
\end{proposition}
\begin{proof}
For elementary tensors we have
$$W(x\s(h\ot x)) = \lb x,x\s\rb W(h) = \lb W(h\ot x),x\s\rb .$$
By linearity this proves that $W\circ x\s = x\s\circ W$ on $H\ot E$.
That $x\s$ extends to a bounded operator from
$\g(H,E)\to H$ now follows from the fact that both $W:H\to L^2(\Om)$ and 
$W:\g(H,E)\to L^2(\Om;E)$ are isometric embeddings.
\end{proof} 

In particular it follows, for $T\in \g(H,E)$, that the $E$-valued random
variables $W(T)$ are {\em Gaussian} (cf. Definition \ref{def:Gaussian}). This
point will be taken up in more detail in Section \ref{sec:Gaussian}. 

So far we have treated $H\ot E$ as an abstract tensor product of $H$ and $E$.   
The elements of $ H \otimes E$ define bounded linear operators from $ H $ to $E$
by the formula
$$ (h\otimes x)h' := [h,h']x, \quad h'\in H,$$
and we have 
$$
\bal
 \Big\n 
\sum_{n=1}^N h_n\otimes x_n\Big\n_{\calL(H,E)}^2
& = \sup_{\n h\n\le 1} \Big\n\sum_{n=1}^N [h_n,h]x_n\Big\n^2 
= \sup_{\n (a_n)_{n=1}^N\n_2\le 1} \Big\n \sum_{n=1}^N a_n x_n\Big\n^2 
\\ & \le \E \Big\n
\sum_{n=1}^N \g_n x_n\Big\n^2
= \Big\n \sum_{n=1}^N h_n\otimes x_n\Big\n_{\g(H,E)}^2 
\eal
$$ 
where the inequality follows from the fact that for any $x\s\in E\s$ of norm one
and any choice $ (a_n)_{n=1}^N \in \ell_N^2$ of norm $\le 1$ we have
$$
\bal\Big| \sum_{n=1}^N a_n\lb x_n,x\s\rb\Big|^2 
& \le \sum_{n=1}^N |a_n|^2 
\sum_{n=1}^N |\lb x_n,x\s\rb|^2  \le \sum_{n=1}^N |\lb x_n,x\s\rb|^2 
\\ & = \E \Big| \sum_{n=1}^N \g_n \lb x_n,x\s\rb\Big|^2 \le \E \Big\n
\sum_{n=1}^N \g_n x_n\Big\n^2.
\eal$$
This shows that the identity map on $H\ot E$ has a unique extension to a
continuous and contractive linear operator 
$$ j:\g( H ,E)\to \calL( H ,E).$$
To prove that $j$ is injective let $W: H \to L^2(\O)$ be an isonormal process. 
For all $T \in  H \ot E$, say $T = \sum_{n=1}^N h_n \ot x_n$ as before,
the adjoint operator $(jT)\s\in\calL(E\s, H )$ is given by $(jT)\s x\s =
\sum_{n=1}^N \lb x_n,x\s\rb h_n$, so
$$   \E |\lb W(T),x\s\rb|^2 
= \E \Big|\sum_{n=1}^N \g_n \lb x_n,x\s\rb\Big|^2 
= \sum_{n=1}^N |\lb x_n,x\s\rb|^2 =
\n (jT)\s x\s\n^2.$$
By approximation, the identity of the left- and right-hand sides extends to
arbitrary $T\in\g( H ,E)$.  
Now if $jT=0$ for some $T\in \g( H ,E)$, then    
$$ \E |\lb W(T),x\s\rb|^2 = \n (jT)\s x\s\n^2 = 0$$
for all $x\s\in E\s$,
so $W(T) = 0$ and therefore $T=0$. 

\begin{definition} An operator $T\in \calL( H ,E)$ is called
{\em $\g$-radonifying} if it belongs to $\g( H ,E)$.
\end{definition}

From now on we shall always identify $\g(H,E)$ with a linear subspace of
$\calL(H,E)$. 

\begin{proposition}\label{prop:compact}
Every operator $T\in\g(H,E)$ is compact.
\end{proposition}
\begin{proof}
Let $\limn T_n = T$ in $\g(H,E)$ with each operator $T_n$ of finite rank. Then
$\limn T_n = T$ in $\calL(H,E)$ and therefore $T$ is compact, it being the
uniform limit of a sequence of 
compact operators.
\epf

The degree of compactness of an operator can be quantified by its entropy
numbers. Proposition \eqref{prop:compact} can be refined accordingly; see
Section \ref{sec:conditions}.

Under the identification of $\g(H,E)$ with a linear subspace of $\calL(H,E)$,
Proposition \ref{prop:xs} states that if $W$ is an $H$-isonormal process, then
for all $T\in \g(H,E)$ and $x\s\in E\s$ we have 
$$\lb W(T),x\s\rb = W(T\s x\s).$$ 
Similarly, Proposition \ref{prop:proj} states that for all $T\in \g(H,E)$ and
orthogonal projections $P$ from $H$ onto a closed subspace $H_0$ we have
$T|_{H_0}\in \g(H_0,E)$ and 
$$\n T|_{H_0}\n_{\g(H_0,E)}\le \n T\n_{\g(H,E)}.$$

As an application we deduce a representation for the norm of $\g(H,E)$ in terms
of finite orthonormal systems.

\begin{proposition}\label{prop:norm-gamma}
For all $T\in \g(H,E)$ we have
$$ \n T\n_{\g(H,E)}^2 = \sup_{h} \ \E \Big\n \sum_{n=1}^N \g_n Th_n\Big\n^2\,$$
where the supremum is over all finite orthonormal systems
${h} = \{h_{1},\dots,h_{N}\}$ in $H$.
\end{proposition}
\begin{proof}
The inequality `$\le$' is obtained by approximating $T$ with elements from
$H\otimes E$. For the inequality `$\ge $' we note that for all
finite-dimensional subspaces $H_0$ of $H$ we have
$\n T\n_{\g(H,E)}\ge \n T|_{H_0}\n_{\g(H_0,E)}$. 
The operator $T|_{H_0}$, being of finite rank from $H_0$ to $E$, may be
identified with an element of $H_0\ot E$, and the desired inequality follows
from this.
\end{proof}

\begin{definition}
An operator $T\in\calL(H,E)$ satisfying
$$\sup_{h} \ \E \Big\n \sum_{n=1}^N \g_n Th_n\Big\n^2 <\infty,$$ 
where the supremum is over all finite orthonormal systems
${h} = \{h_{1},\dots,h_{N}\}$ in $H$, is called {\em $\g$-summing}.
\end{definition} 

The class of $\g$-summing operators was introduced by {\sc Linde} and {\sc
Pietsch} \cite{LinPie}. 

\begin{definition}
 The space of all $\g$-summings operator from $H$ to $E$ is denoted by $\g_\infty(H,E)$.
\end{definition}

With respect to the norm
$$ \n T\n_{\g_\infty(H,E)}^2 := \sup_{h} \ \E \Big\n \sum_{n=1}^N \g_n Th_n\Big\n^2,$$
$\g_\infty(H,E)$ is easily seen to be a Banach space. 
Proposition \ref{prop:norm-gamma} asserts that every
$\g$-radonifying operator $T$ is $\g$-summing and 
$$\n T\n_{\g_\infty(H,E)} = \n T\n_{\g(H,E)}.$$
Stated differently, $\g(H,E)$ is isometrically contained in $\g_\infty(H,E)$ as a closed subspace.
In the next section we shall prove that if $E$ does not contain a closed subspace isomorphic to $c_0$,
then $$\g_\infty(H,E)=\g(H,E),$$ that is, every $\g$-summing operator is $\g$-radonifying.

The next proposition is essentiall due to {\sc Kalton} and {\sc Weis} \cite{KalWei07}.

\begin{proposition}[$\g$-Fatou lemma]\label{prop:g-Fatou} 
Consider a bounded sequence $(T_n)_{n\ge 1}$ in $\g_\infty(H,E)$.
If $T\in\calL(H,E)$ is an operator such that $$\limn \lb T_n h, x\s\rb = \lb Th,
x\s\rb\quad  h\in H, \ x\s\in E\s,$$ then $T\in \g_\infty(H,E)$ and 
$$\n T\n_{\g_\infty(H,E)}\le \liminf_{n\to\infty}\n T_n\n_{\g_\infty(H,E)}.$$
\end{proposition}
\begin{proof}
Let $h_1,\dots, h_K$ be a finite orthonormal system in $H$.
Let $(x_m\s)_{m\ge 1}$ be a sequence of unit vectors in $E\s$ which is norming 
for the linear span of $\{Th_1,\dots, Th_K\}$.
For all $M\ge 1$ we have, by the Fatou lemma,
$$
\bal
 \E\sup_{m=1,\dots,M}\Big|\Big\lb \sum_{k=1}^{K} \g_k T h_k,
 x_m\s\Big\rb\Big|^2
& \le \liminf_{n\to\infty} 
\E\sup_{m=1,\dots,M}\Big|\Big\lb \sum_{k=1}^{K} \g_k T_n h_k,
x_m\s\Big\rb\Big|^2
\\ & \le \liminf_{n\to\infty}\n T_n\n_{\g_\infty(H,E)}^2.
\eal$$
Taking the limit $M\to\infty$ we obtain, by the monotone convergence theorem,
$$
 \E\Big\n\sum_{k=1}^{K} \g_k T h_k\Big\n^2
\le \liminf_{n\to\infty}\n T_n\n_{\g_\infty(H,E)}^2.
$$
\end{proof}

We continue with a useful criterion for membership of $\g_\infty(H,E)$. 
Its proof stands a bit apart from the main line of development and depends 
on an elementary comparison result in Section \ref{sec:ideal}, but 
for reasons of presentation we prefer to present it here.

\begin{proposition}[Testing against an orthonormal basis]\label{prop:summing-sep}
Let $H$ be a separable Hilbert space with orthonormal basis $(h_n)_{n\ge 1}$.
An operator $T\in \calL(H,E)$ belongs to $\g_\infty(H,E)$ if and only if
$$
 \sup_{N\ge 1}\, \E\Big\n \sum_{n=1}^N \g_n\,
Sh_n\Big\n^2< \infty.$$
In this situation we have $$ \n S\n_{\g_p^\infty(H,E)}^2 =
 \sup_{N\ge 1}\, \E\Big\n \sum_{n=1}^N \g_n\,
Sh_n\Big\n^2.$$
\epr
\bpf
Let $\{h_1',\dots, h_k'\}$ be an orthonormal system in $H$. 
For $K\ge 1$ let $P_K$ denote the orthogonal projection onto the span of 
$\{h_1,\dots,h_K\}$.
For all $x\s\in E\s$ and $K\ge k$ we have
$$
 \sum_{j=1}^{k} \lb S P_K h_j',  x\s\rb^2
\le \n P_K S\s x\s\n^2
 =  \sum_{n=1}^{K} \lb S h_n, x\s\rb ^2.
$$
From Lemma \ref{lem:g-compar} below it follows that 
$$  
\E\Big\n \sum_{j=1}^{k} \g_j \,SP_Kh_j'\Big\n^2
\le \E\Big\n \sum_{n=1}^{K} \g_n\,
Sh_n\Big\n^2\le \sup_{N\ge 1}\E\Big\n \sum_{n=1}^N \g_n\,
Sh_n\Big\n^2.$$
Hence by Fatou's lemma,
$$\E\Big\n \sum_{j=1}^{k} \g_j \,Sh_j'\Big\n^2
\le \liminf_{K\to\infty}
\E\Big\n \sum_{j=1}^{k} \g_j\,
SP_K h_j'\Big\n^2\le \sup_{N\ge 1}\E\Big\n \sum_{n=1}^N \g_n\,
Sh_n\Big\n^2.$$
It follows that
$$ \n S\n_{\g_\infty(H,E)}^p\le \sup_{N\ge 1}\, \E\Big\n \sum_{n=1}^N \g_n\,
Sh_n\Big\n^2.$$
The converse inequality trivially holds and the proof is complete.
\epf
 
We continue with two criteria for $\g$-radonification. The first is stated in
terms of maximal orthonormal systems. 

\begin{theorem}[Testing against a maximal orthonormal system]\label{thm:ONS}
Let $H$ be a Hilbert space with a maximal orthonormal system $(h_i)_{i\in I}$
and let $(\g_i)_{i\in I}$ be a family of independent standard Gaussian random
variables with the same index set. 
An operator $T\in\calL(H,E)$ belongs to $\g(H,E)$ if and only if 
$$\sum_{i\in I} \g_i Th_i$$ is summable in $L^2(\Om;E)$. In this situation we have
$$ \n T\n_{\g(H,E)}^2 = \E\Big\n\sum_{i\in I} \g_i  Th_i\Big\n^2.$$
\end{theorem}
\begin{proof}
We may assume that $\g_i = W(h_i)$ for some 
$H$-isonormal process $W$.

We begin with the `if' part and put $X:=\sum_{i\in I} \g_i Th_i$.
Given $\e>0$ choose $i_1,\dots,i_N\in I$ such that for all finite subsets
$J\subseteq I$ containing $i_1,\dots,i_N$ we have
$ \E\n X - X_J\n^2<\e^2,$ where $X_J:= \sum_{j\in J} \g_j Th_j$.
Set $T_J := \sum_{j\in J} h_j\otimes Th_j$. Then for all finite subsets
$J,J'\subseteq I$ containing $i_1,\dots,i_N$
we have $$\n T_J - T_{J'}\n_{\g(H,E)} 
= \n W(T_J) - W(T_{J'})\n_{L^2(\O;E)} = \n X_J - X_{J'}\n_{L^2(\O;E)} < 2\e.$$
It follows that the net $(T_J)_J$ is Cauchy in $\g(H,E)$ and therefore 
convergent to some $S\in \g(H,E)$. 
From $$W(S\s x\s) = \sum_{i\in I} \g_i \lb Th_i,x\s\rb = W(T\s x\s)$$ it follows
that $S\s x\s=T\s x\s$ for all $x\s\in E\s$ and therefore $S=T$.

For the `only if' part we note that
$\sum_i [h_i,T\s x\s]h_i$ is summable to $T\s x\s$ (cf. Example \ref{ex:1}) and
therefore
$$ \lb W(T),x\s\rb =  W(T\s x\s) = \sum_{i\in I} [h_i,T\s x\s]W(h_i)
=  \sum_{i\in I} \g_i [h_i,T\s x\s] 
=
\sum_{i\in I} \g_i \lb Th_i,x\s\rb 
$$
for all $x\s\in E\s$.
Hence by Proposition \ref{prop:IN}, $\sum_{i\in I} \g_i Th_i = W(T)$ in
$L^2(\O;E)$.
Finally, by Proposition \ref{prop:Ito},
$ \E\n\sum_{i\in I} \g_i Th_i\n^2  = \E \n W(T)\n^2 = \n T\n_{\g(H,E)}^2.$
\end{proof}

For operators $T\in \calL(H,E)$ we have an orthogonal decomposition
\begin{equation}\label{eq:sep-supp}
 H = \textrm{ker}(T)\oplus \ov{\textrm{ran}(T\s)}.
\end{equation}
The following argument shows that for all $T\in \g(H,E)$ the subspace
$\ov{\textrm{ran}(T\s)}$ is separable. Let $T_n\to T$ in $\g(H,E)$ with each
$T_n\in H\ot X$. The range of each adjoint operator $T_n\s$ is
finite-dimensional. Therefore the closure of $\bigcup_{n\ge
1}\textrm{ran}(T_n\s)$ is a separable closed subspace $H_0$ of $H$. 
By the Hahn-Banach theorem, $H_0$ is weakly closed. Hence upon passing to the
limit for $n\to \infty$ we infer that $\textrm{ran}(T\s)\subseteq H_0$ and the
claim is proved.

If $(h_n)_{n\ge 1}$ is an orthonormal basis for any separable closed subspace
$H'\subseteq H$ containing $\ov{\textrm{ran}(T\s)}$, then
Theorem \ref{thm:ONS} implies that an operator $T\in \calL(H,E)$ belongs to
$\g(H,E)$ if and only if the sum $\sum_{n\ge 1} \g_n Th_n$ converges in
$L^2(\O;E)$, in which case we have
$$ \n T\n_{\g(H,E)}^2 = \E\Big\n\sum_{n\ge 1} \g_n Th_n\Big\n^2.$$ 
In particular, if $H$ itself is separable this criterion may be applied
for any orthonormal basis $(h_n)_{n\ge 1}$ of $H$ and we have proved:

\begin{corollary}\label{cor:ONB}
If $H$ is a separable  Hilbert space with orthonormal basis $(h_n)_{n\ge 1}$,
and if $(\g_n)_{n\ge 1}$ is a Gaussian sequence, 
then a bounded operator $T:H\to E$ belongs to $\g(H,E)$ if and only if 
$\sum_{n\ge 1} \g_n Th_n$ converges in $L^2(\Om;E)$. In this situation we have
$$ \n T\n_{\g(H,E)}^2 = \E\Big\n\sum_{n\ge 1} \g_n Th_n\Big\n^2.$$
\end{corollary}

In many papers, this result is taken as the definition of the space $\g(H,E)$.
The obvious disadvantage of this approach is that it imposes an unnecessary
separability assumption on the Hilbert spaces $H$.
We mention that  
an alternative proof of the corollary 
could be given along the lines of Proposition \ref{prop:summing-sep}. 

The next criterion for membership of $\g(H,E)$ is phrased in terms of
functionals:

\begin{theorem}[Testing against functionals]\label{thm:functionals}
Let $W: H \to L^2(\O)$ be an isonormal process.
A bounded linear operator $T:  H \to E$ belongs to $\g( H ,E)$
if and only if  there
exists a random variable $X\in L^2(\O;E)$ such that for all $x\s\in E\s$ we have
$$ W(T\s x\s) = \lb X,x\s\rb$$
in $L^2(\O)$. In this situation we have $W(T) = X$ in $L^2(\O;E)$.
\end{theorem}

\begin{proof}
To prove the `only if' part, take $X= W(T)$.

For the `if' part we need to work harder.
Let $G$ be the closed subspace in $L^2(\O)$ spanned by the random variables of
the form $\lb X,x\s\rb$, $x\s\in E\s$. By a Gram-Schmidt argument, choose a
maximal orthonormal system
$(g_i)_{i\in I}$ in $G$ of the form $g_i = \lb X,x_i\s\rb$ for suitable
$x_i\s\in E\s$. 
Then $(g_i)_{i\in I}$ is a family of independent standard Gaussian random
variables. Put $h_i = T\s x_i\s$ and $x_i = Th_i$. From
$$[h_i,h_j ] = \E W(h_i)W(h_j) = \E \lb X,x_i\s\rb\lb X,x_j\s\rb= \E g_ig_j =
0\quad (i\not=j)$$ we infer that $(h_i)_{i\in I}$ is a maximal orthonormal
system
for its closed linear span $H_0$ in $H$. Expanding against $(g_i)_{i\in I}$, 
for all $x\s\in E\s$ we have
$ \lb X,x\s\rb = \sum_{i\in I} c_i \lb X,x_i\s\rb = \sum_{i\in I} c_i g_i $
with summability in $L^2(\Om)$ (cf. Example \ref{ex:1}), where
$$c_i =  \E \lb X,x\s\rb\lb X,x_i\s\rb
= [T\s x\s, T\s x_i\s] = \lb Th_i,x\s\rb.$$
Hence,
$ \lb X,x\s\rb = \sum_{i\in I} g_i \lb Th_i,x\s\rb$
with summability in $L^2(\O)$. This being true for all $x\s\in E\s$, 
by Proposition \ref{prop:IN} we then have 
$X = \sum_{i\in I} g_i Th_i$
with summability in $L^2(\O;E)$. Now Theorem \ref{thm:ONS} implies that
$T\in\g(H_0,E)$.
Since $T$ vanishes on $H_0^\perp$, Proposition \ref{prop:incl} implies that
$T\in \g(H,E)$.

The final assertion follows from $\lb W(T),x\s\rb = W(T\s x\s) = \lb X,x\s\rb$.
\end{proof}

A bounded operator $T$ from a separable Hilbert space into another Hilbert space
$E$ is $\g$-radonifying if and only if $T$ is Hilbert-Schmidt, i.e., for all
orthonormal bases $(h_n)_{n\ge 1}$ of $H$ we have $$\sum_{n\ge 1} \n
Th_n\n^2<\infty.$$
The simple proof is contained in Proposition \ref{prop:HS}.
Without proof we mention the following extension of this result to Banach spaces
$E$, due to {\sc Kwapie\'n} and {\sc Szyma\'nski} \cite{KwaSzy} (see also
\cite[Theorem 3.5.10] {Bog}):

\begin{theorem} Let $H$ is a separable Hilbert space and $E$ a Banach space. If
$T\in \g(H,E)$, then there exists an
orthonormal basis $(h_n)_{n\ge 1}$ of $H$ such that $$\sum_{n\ge 1} \n
Th_n\n^2<\infty.$$
\end{theorem}

\section{The theorem of Hoffmann-J{\o}rgensen and Kwapie\'n}
\label{sec:HJ}

In the previous section we have seen that every $\g$-radonifying operator is
$\g$-summing.
The main result of this section is 
the following converse, essentially due to {\sc Hoffmann-J\o rgensen} and {\sc
Kwapie\'n}: if $E$ does not contains a closed subspace isomorphic to $c_0$, then
every $\g$-summing operator is $\g$-radonifying.

We begin with some preparations.
A sequence of $E$-valued random variables $(Y_n)_{n\ge 1}$ is said to be 
{\em bounded in probability}
if for every $\e>0$ there exists an $r\ge 0$ such that
$$ \sup_{n\ge 1}\, \P \,\{\n Y_n\n > r\} <  \e.$$

\begin{lemma}\label{lem:bdd-a.s.}
Let $(X_n)_{n\ge 1}$ be a sequence 
of independent symmetric $E$-valued random variables and let $S_n = \sum_{j=1}^n
X_j$.
The following assertions are equivalent:
\begin{enumerate}
\item[\rm(1)] the sequence $(S_n)_{n\ge 1}$ is bounded almost surely; 
\item[\rm(2)] the sequence $(S_n)_{n\ge 1}$ is bounded in probability.
\end{enumerate}
\end{lemma}
\begin{proof}
(1)$\Rightarrow$(2): 
Fix $\e>0$ and choose $r\ge 0$
so that  $\P\{\sup_{n\ge 1} \n S_n\n > r\} < \e.$ Then
$$ \P\{\n S_n\n > r\} \le \P\{\sup_{n\ge 1} \n S_n\n > r\} < \e$$ for all $n\ge
1,$
and therefore  $(S_n)_{n\ge 1}$ is bounded in probability. 
 
(2)$\Rightarrow$(1): Fix $\e>0$ arbitrary
and choose $r\ge 0$ so large that
$\P\{ \n S_n\n  > {r}\} < \e$ for all $n\ge 1$.
By Proposition \ref{prop:Levy}, for all $n\ge 1$ we have
$$ \P\Bigl\{\sup_{1\le k\le n}\n S_k\n > r\Bigr\}
\le 2\, \P\Bigl\{ \n S_n \n> r\Bigr\} < 2\e.
$$
It follows that
$ \P\{\sup_{k\ge 1} \n S_k\n > r\} \le 2\e$ for all $r>0$, so
$\P\{\sup_{k\ge 1} \n S_k\n =\infty\} \le 2\e$.
Since $\e>0$ was arbitrary, this shows that 
$(S_n)_{n\ge 1}$ is bounded almost surely.
\end{proof}

In the proof of the next theorem we shall apply the following criterion, due to
{\sc Bessaga} and {\sc Pelczy\'nski}
(see \cite{AlbKal}), to detect isomorphic copies of the Banach space $c_0$:
if $\seqy$ is a sequence in $E$ such that
\begin{enumerate}
\item[\rm(i)] $\limsup_{n\to\infty} \n y_n\n > 0$;
\item[\rm(ii)] there exists $M\ge 0$ such that 
$\n \sum_{j=1}^k a_j y_j\n \le M$
for all $k\ge 1$ and all $a_1,\dots, a_k\in\{-1,1\}$,
\end{enumerate} then $\seqy$ has a subsequence whose closed linear span is
isomorphic to $c_0$.

\begin{theorem}[{\sc Hoffmann--J\o rgensen} and {\sc Kwapie\'n} \cite{HofJor,
Kwa}]\label{thm:HJ-K}
For a Banach space $E$ the following assertions are equivalent:
\begin{enumerate}
\item[\rm(1)] for all sequences $(X_n)_{n\ge 1}$ of independent symmetric
$E$-valued
random variables, the almost sure boundedness of the partial sum sequence
$(S_n)_{n\ge 1}$ implies the almost sure convergence of 
$(S_n)_{n\ge 1}$;
\item[\rm(2)] the space $E$ contains no closed subspace isomorphic to $c_0$.
\end{enumerate}
\end{theorem}

\begin{proof}
We shall prove the implications
(1)$\Rightarrow$(3)$\Rightarrow$(2)$\Rightarrow$(3)$\Rightarrow$(1), where

\ben
\item[\rm(3)] {\em for all sequences $(x_n)_{n\ge 1}$ in $E$, the almost sure
boundedness of the partial sums of
$\sum_{n\ge 1} r_n x_n$ implies $\limn x_n = 0$.}
\een

(1)$\Rightarrow$(3):\ This implication  is trivial. 

(3)$\Rightarrow$(2):\ Let $u_n$ denote the $n$-th unit vector of
$c_0$. The sum $\sum_{n\ge 1} r_n(\om) 
u_n$ fails to converge for all $\om\in \O$ while its partial sums are uniformly 
bounded.

(2)$\Rightarrow$(3):\ 
Suppose (3) does not hold. Then there exists a sequence $\seqx$ in $E$
with $\limsup_{n\to\infty}\n x_n\n>0$ such that the partial sums of  $\sumn r_n
x_n$ are 
bounded almost surely.

Let $\G$ denote the $\sigma$-algebra generated by the 
sequence $(r_n)_{n\ge 1}$. We claim that for all $B\in\G$,
$$ \limn \P\bigl(B\cap\{r_n = -1\}\bigr) 
=  \limn \P\bigl(B\cap\{r_n = 1\}\bigr) = \tfrac12\P(B).$$
For all $B\in \G_N$, the $\sigma$-algebra generated
by $r_1,\dots,r_N$, this follows immediately from the fact that 
$r_n$ is independent of
$\G_N$ for all $n > N$. The case for $B\in \G$ 
now follows from the general fact of measure theory that
for any $B\in\G$ and any $\e>0$ there exist $N$ sufficiently large and
$B_N\in\G_N$ such that $\P(B_N\Delta B)<\e$.

Choose $M\ge 0$ in such a way that
$$\P\Bigl\{ \sup_{n\ge 1} \ \Bigl\n \sum_{j=1}^n r_j x_j\Bigr\n \le M \Bigr\}  >
\frac12. $$
By the observation just made we can find an index $n_1\ge 1$ large enough such
that
for all $a_1\in\{-1,1\}$ we have 
$$ \P\Bigl\{ \sup_{n\ge 1}\ \Bigl\n \sum_{j=1}^n r_j x_j\Bigr\n\le M,\ \ r_{n_1}
= a_1 \Bigr\}  > \frac14. $$
Continuing inductively, we find a sequence $1\le n_1 < n_2\dots$
such that for all all choices $a_1,\dots, a_k\in\{-1,1\}$,
$$ \P\Bigl\{ \sup_{n\ge 1}\ \Bigl\n \sum_{j=1}^n r_j x_j\Bigr\n\le M,\ \ r_{n_1}
= a_1, \ \dots\,, \ r_{n_k} = a_k\Bigr\}  > \frac1{2^{k+1}}. $$
Now define 
$$ r_j' := \Bigl\{
\begin{array}{rl}
r_j, & j = n_k\ \hbox{for some $k\ge 1$}, \\
-r_j, & \hbox{else}. 
\end{array}
\Bigr.
$$
Then by symmetry, for all $k\ge 1$ we have
$$ \P\Bigl\{ \sup_{n\ge 1}\ \Bigl\n \sum_{j=1}^n r_j' x_j\Bigr\n\le M,\ \
r_{n_1} = a_1, \ \dots\,,  \ r_{n_k} = a_k\Bigr\}  > \frac1{2^{k+1}}. $$
Since 
$$\P\Bigl\{ r_{n_1} = a_1, \ \dots\,, \ r_{n_k} = a_k\Bigr\}  = \frac1{2^{k}} $$
it follows that for all $k\ge 1$ and all choices $a_1,\dots, a_k\in\{-1,1\}$,
the event
$$ \Big\{ \sup_{n\ge 1}\ \Bigl\n \sum_{j=1}^n r_j x_j\Bigr\n\le M, \quad  
 \sup_{n\ge 1}\ \Bigl\n \sum_{j=1}^n r_j' x_j\Bigr\n\le M, \quad
r_{n_1} = a_1, \ \ \dots\,,\  \ r_{n_k} =a_k\Big\}$$
has positive probability. For any $\om$ in this event,
$$ \Bigl\n \sum_{j=1}^k a_j x_{n_j} \Bigr\n =
\Bigl\n \frac12\sum_{j=1}^{n_k} r_{j}(\om)x_j + \frac12\sum_{j=1}^{n_k}
r_{j}'(\om)x_j \Bigr\n  \le M.$$
Since this holds for all choices $a_1,\dots, a_k\in\{-1,1\}$,
the Bessaga-Pelczy\'nski criterion implies that
the sequence $(x_{n_j})_{j\ge 1}$ has a subsequence whose closed linear 
span is isomorphic to $c_0$.

(3)$\Rightarrow$(1):\ 
Suppose the partial sums of $\sumn X_n$ are bounded almost surely.

Let $1\le n_1 < n_2 < \dots$ be an arbitrary increasing sequence of indices
and let $Y_k := S_{n_{k+1}} - S_{n_k}$.
The partial sums of $\sumk Y_k$ are bounded almost surely.

On a possibly larger probability space, let $(r_n)_{n\ge 1}$ be a Rademacher
sequence independent of $(X_n)_{n\ge 1}$.
By Lemma \ref{lem:bdd-a.s.}, the partial sums of $\sumk Y_k$ are 
bounded in probability on $\Om_X$, and because $(Y_n)_{n\ge 1}$ and
$(r_nY_n)_{n\ge 1}$ are identically distributed the same is true for the 
partial sums of $\sumk r_k Y_k$. Another application of 
Lemma \ref{lem:bdd-a.s.} shows that the partial sums of this sum
are bounded almost surely.
By Fubini's theorem it follows that for almost all $\om\in\Om$, the 
partial sums of $\sumk r_k Y_k(\om)$ are bounded almost surely.
By (3), $\limk Y_k(\om)=0$ for almost all $\om\in\Om$. This
implies that $\limk Y_k = \limk S_{n_{k+1}} - S_{n_k} = 0$ in probability.

Suppose now that the sequence $\seqS$ fails to converge almost surely.
Then by Proposition \ref{prop:IN1} it fails to converge in probability, and
there exists
an $\e>0$ and increasing sequence $1\le n_1 < n_2 < \dots$ such that
$$\P\{\n S_{n_{k+1}} - S_{n_k}\n\ge \e\} \ge \e\qquad\forall k=1,3,5,\dots$$
This contradicts the assertion just proved. 
\end{proof}

Now we are in a position to state and prove a converse to Proposition
\ref{prop:norm-gamma}.

\begin{theorem}\label{thm:c0}
Let $H$ be a Hilbert space and $E$ a Banach space not containing a closed
subspace isomorphic to $c_0$. Then $\g_\infty(H,E) = \g(H,E)$ isometrically.
\end{theorem}

This result implies that when $E$ does not contain a copy of $c_0$, results 
involving $\g$-summing operators (such as the $\g$-Fatou lemma (Proposition \ref{prop:g-Fatou})
and the $\g$-multiplier theorem (Theorem \ref{thm:KW}) may be reformulated in terms
of $\g$-radonifying operators.

\begin{proof}
Let $T\in\g_\infty(H,E)$ be given and fixed; we must show that $T\in\g(H,E)$.
Once we know this, the equality of norms $\n T\n_{\g_\infty(H,E)} =\n T\n_{\g(H,E)}$
follows from Proposition \ref{prop:norm-gamma}.
 
We begin by proving that there exists a separable closed subspace $H_1$ of $H$
such that $T$ vanishes on $H_1^\perp$. To this end let $H_0$ be the null space
of $T$ and let $(h_i)_{i\in I}$ be a maximal orthonormal system for 
$H_1:= H_0^\perp$. We want to prove that $H_1$ is separable, i.e., that the
index set $I$ is countable. Suppose the contrary. Then there exists an integer
$N\ge 1$ such that $\n Th_i\n\ge 1/N$ for uncountably many $i\in I$.
Put $J:=\{i\in I: \ \n Th_i\n\ge 1/N\}$. Let $(j_n)_{n\ge 1}$ be any sequence in
$J$ with no repeated entries. For all $N\ge 1$ we have
$$ \E \Big\n \sum_{n=1}^N \g_n Th_{j_n}\Big\n^2\le M,$$ where $M$ is the
supremum in the statement of the theorem. 
This means that the sequence of random variables
$S_N := \sum_{n=1}^N \g_n Th_{j_n}$, $N\ge 1$, is bounded in $L^2(\O;E)$, and
therefore
bounded in probability. By Lemma \ref{lem:bdd-a.s.}, this sequence is bounded
almost surely. An application of Theorem \ref{thm:HJ-K} then shows that the sum
$\sumn \g_n Th_{j_n}$ converges almost surely. Now Proposition
\ref{prop:Rad-phi} can be used to the effect that the Rademacher sum $\sumn r_n
Th_{j_n}$ converges almost surely as well. But this forces $\limn Th_{j_n} = 0$,
contradicting the fact that $j_n\in J$ for all $n\ge 1$. This proves the claim.

By the claim we may assume that $H$ is separable; let $(h_n)_{n\ge 1}$ be an
orthonormal basis for $H$.
Repeating the argument just used, $\sumn \g_n Th_{n}$ converges almost surely.
To prove the $L^2(\O;E)$-convergence of this sum, put $X_N := \sum_{j=1}^N \g_j
Th_{j}$ and $X:= \sum_{n\ge 1} \g_n Th_{n}$.
By Fubini's theorem and Proposition \ref{prop:Levy},
$$
\bal
\E \sup_{1\le n\le N}\n X_{n}\n^2
& =
\int_0^\infty 2r \P\big\{\sup_{1\le n\le N}\n X_{n}\n> r\big\}\,dr
\\ & \le \int_0^\infty 4r \P\{\n X_{N}\n> r\}\,dr
= 2\E \n X_{N}\n^2.
\eal
$$
Hence $\E \sup_{n\ge 1}\n X_{n}\n^2\le 2\sup_{n\ge 1}\E \n X_n\n^2$ by the
monotone convergence theorem, and this supremum is finite by assumption. Hence
$\lim_{n\to\infty} \E\n X_{n} - X\n^2 = 0$ by the dominated convergence theorem.

An appeal to Theorem \ref{thm:ONS} and the remark following it finishes the
proof.
\end{proof}

The assumption that $E$ should not contain an isomorphic copy of 
$c_0$ cannot be omitted, as is shown by the next example due to {\sc Linde} and
{\sc Pietsch} \cite{LinPie}. 

\begin{example}\label{ex:LP}
The multiplication operator 
$T: \ell^2\to c_0$ defined by 
$$ T\bigl((\a_n)_{n\ge 1}\bigr) := (\a_n/\sqrt{\log (n+1)})_{n\ge 1}$$ 
is $\g$-summing but fails to be
$\g$-radonifying.

To prove this we begin with some preliminary estimates. Let $\g$ be a standard
Gaussian random variable and put
$$G(r):= \P\{|\g|^2\le r\} =
\frac1{\sqrt{2\pi}}\int_{-\sqrt{r}}^{\sqrt{r}} e^{-\frac12 x^2}\,dx
= \frac1{\sqrt{2\pi}}\int_{0}^{r} \frac{e^{-\frac12 y}}{\sqrt y}\,dy.
$$
An integrations by parts yields, for all $r>0$,
\beq\label{eq:LP-1}
\bal G(r) & =  1-\frac1{\sqrt{2\pi}}\int_{r}^\infty 
\frac{e^{-\frac12 y}}{\sqrt y}\,dy
\\ & = 1-\frac2{\sqrt{2\pi}}\frac{e^{-\frac12 r}}{\sqrt r}
+\frac1{\sqrt{2\pi}}\int_{r}^\infty 
\frac{e^{-\frac12 y}}{y\sqrt y}\,dy
 \ge  1-\frac2{\sqrt{2\pi}}\frac{e^{-\frac12 r}}{\sqrt r}.
\eal
\eeq
Another integration by parts yields, for $r\ge 2$,
\beq\label{eq:LP-2}
\bal G(r) &
=  1-\frac2{\sqrt{2\pi}}\frac{e^{-\frac12 r}}{\sqrt r}
+\frac2{\sqrt{2\pi}}\frac{e^{-\frac12 r}}{r\sqrt r}
-\frac3{\sqrt{2\pi}}\int_{r}^\infty 
\frac{e^{-\frac12 y}}{y^2\sqrt y}\,dy
\\ & \le 1-\frac2{\sqrt{2\pi}}(1-r^{-1})\frac{e^{-\frac12 r}}{\sqrt r} 
 \le 1-\frac1{\sqrt{2\pi}}\frac{e^{-\frac12 r}}{\sqrt r}.
\eal
\eeq

Let $(u_n)_{n\ge 1}$ be the standard unit basis of $\ell^2$.
We check that the assumptions of Proposition \ref{prop:summing-sep} are satisfied
by showing that
$$
 \sup_{N\ge 1}\, 
 \E\Big\n \sum_{n=1}^N \g_n Tu_n\Big\n_{c_0}^2 =\sup_{N\ge 1}\,\E\Big(
\sup_{1\le n\le N} \frac{|\g_n|^2}{\log(n+1)}\Big) < \infty.
$$
Using \eqref{eq:LP-1} 
we estimate, for $t\ge 4$,
\begin{align*}
 \P\Big\{\sup_{1\le n\le N} \frac{|\g_n|^2}{\log(n+1)} > t\Big\}
& =1-  \prod_{n=1}^N G(t\log(n+1)) 
\\ &  \le 1- \prod_{n=1}^N \Big( 1-\frac2{\sqrt{2\pi}}
\frac{1}{\sqrt{(n+1)^t\,t\log(n+1)}}\Big)
\\ &  \le \frac2{\sqrt{2\pi}} \sum_{n=1}^N\frac{1}{\sqrt{(n+1)^t\,t\log (n+1)}}
\\ &  \le \frac2{\sqrt{2\pi\log 2}} \frac{1}{\sqrt{2^{t-4}\,t}}
\sum_{n\ge 1}\frac{1}{(n+1)^2}.
\end{align*}
In the last line we used that for $t\ge 4$ we have 
$(n+1)^t = (n+1)^{t-4} (n+1)^4 \ge 2^{t-4} (n+1)^4$. Therefore, 
$$ \E\Big( \sup_{1\le n\le N} \frac{|\g_n|^2}{\log(n+1)}\Big) 
\le 4+\frac2{\sqrt{2\pi\log 2}} \sum_{n\ge 1}\frac{1}{(n+1)^2}
\int_4^\infty \!\!\frac{1}{\sqrt{2^{t-4}\,t}}\,dt < \infty.
$$

To prove that $T$ is not $\g$-radonifying we argue by contradiction.
If $T$ is $\g$-radonifying, then the sum $X:= \sum_{n\ge 1} \g_n Tu_n$ converges
in $L^2(\O;c_0)$. 
The relation 
$$ c_0 = \bigcup_{N\ge 1}\bigcap_{n\ge N} \big\{ \seqx\in c_0:\ |x_n|\le
1\big\}
$$
implies 
$$ \sum_{N\ge 1} \prod_{n\ge N} \P
\big\{|\g_n|^2\le \log(n+1)\big\} = 
\sum_{N\ge 1} \P\Big\{\bigcap_{n\ge N} 
\big\{|X_n|\le 1\big\}\Big\}\ge 1.
$$
where $X_n$ is the $n$-th coordinate of $X$. 
But for $N\ge 7$ we have $\log(n+1)\ge 2$ for all $n\ge N$ and \eqref{eq:LP-2}
gives
$$
\prod_{n\ge N} \P \{|\g_n |^2\le \log(n+1)\}
\le \prod_{n\ge N} \Big(
1-\frac1{\sqrt{2\pi}}\frac{1}{\sqrt{(n+1)\log(n+1)}}\Big) =0,
$$
noting that
$$ \sum_{n\ge N}\frac{1}{\sqrt{(n+1)\log(n+1)}} = \infty.$$
This is contradiction concludes the proof.
\end{example}

\section{The $\g$-multiplier theorem}\label{sec:g-boundedness}

The main result of this section states that 
functions with $\g$-bounded range act as multipliers on certain spaces of
$\g$-radonifying operators. This establishes a connection between the notions of
$\g$-radonification and $\g$-boundedness.

\begin{definition}\label{def:g-bounded} Let $E$ and $F$ 
be Banach spaces.
An operator family 
$\mathscr{T}\subseteq\calL(E,F)$ is said to be {\em $\g$-bounded}
if there exists a constant $M\ge 0$ such that
$$
\Bigl(\E\Bigl\|\sum_{n=1}^N \g_n T_n x_n \Bigr\|^2\Bigr)^\frac12 \leq
M\Bigl(\E\Bigl\|\sum_{n=1}^N \g_n x_n \Bigr\|^2\Bigr)^\frac12,
$$
for all $N\ge 1$, all $T_1,\dots, T_N\in \mathscr{T}$, and all
$x_1,\dots, x_N\in E$.
\end{definition}

The least admissible constant $M$ is called
the {\em $\g$-bound} of $\mathscr{T}$, notation: $\g({\mathscr T})$.
Every $\g$-bounded family $\cT$ is uniformly bounded and we have
$$ \sup_{T\in \cT}\n T\n \le \g(\cT).$$

Replacing Gaussian random variables by Rademacher variables in the above
definition we arrive at the related notion of {\em $R$-boundedness}.
By a simple randomization argument, every $R$-bounded family is $\g$-bounded;
the converse holds if $E$ has finite cotype (since in that case Gaussian sums
can be estimated in terms of Rademacher sums; see Proposition
\ref{prop:Rad-phi}).
The notion of $R$-boundedness plays an important role in vector-valued harmonic
analysis as a tool for proving Fourier multiplier theorems; we refer to 
{\sc Cl\'ement, de Pagter, Sukochev, Witvliet} \cite{CPSW} and the lecture notes
of {\sc Denk, Hieber, Pr\"uss} \cite{DHP} and {\sc Kunstmann} and {\sc Weis}
\cite{KunWei} 
for an introduction to this topic and further references.

It is not hard to prove that closure of the convex hull of a $\g$-bounded family
in the strong operator topology is $\g$-bounded with the same $\g$-bounded. 
From this one deduces the useful fact that integral means of $\g$-bounded
families are $\g$-bounded; this does not increase the $\g$-bound.

Let $(A,\calA,\mu)$ be a $\sigma$-finite measure space. With slight abuse of
terminology, a function $\phi: A\to \calL(E,F)$ is called {\em strongly
measurable} if $\phi x: A\to F$ is strongly measurable for all $x\in E$. 
For a bounded and strongly measurable function $\phi: A\to \calL(H,E)$
we define the operator $T_{\phi} \in \calL(L^2(A;H),E)$ by
$$ T_\phi f := \int_A \phi f\,d\mu.$$
Note that if $\phi$ is a simple function with values in $H\ot E$ (such a
function will be called a {\em finite rank simple function}), then 
$T_\phi\in \g(L^2(A;H),E)$.

Now we are ready to state and prove the main result of this section, due to {\sc
Kalton} and {\sc Weis} \cite{KalWei07} in a slightly simpler formulation.

\begin{theorem}[$\g$-Bounded functions as $\g$-multipliers]\label{thm:KW} 
Let $(A,\calA,\mu)$ be a $\sigma$-finite measure space.
Suppose that $M:A\to \calL(E,F)$ is strongly measurable
and has $\gamma$-bounded range $\mathscr{M} := \{M(t): \ t\in A\}$. 
Then for every finite rank simple function $\phi:A\to \ga(H,E)$ the operator 
$T_{M\phi}$ belongs to $\g_\infty(L^2(A;H),F)$ and 
$$
 \n T_{M\phi}\n_{\g_\infty(L^2(A;H),F)} \le \g(\mathscr{M})\,\n
T_\phi\n_{\g(L^2(A;H),E)}.
$$
As a result, the map $\wt M: T_\phi\mapsto T_{M\phi}$ has a unique extension to
a bounded operator
$$ \wt M:  \g(L^2(A;H),E) \to \g_\infty(L^2(A;H),F)$$ of norm $\n \wt M\n 
\le \g(\mathscr{M})$. 
\end{theorem}

\begin{proof}
The uniqueness part follows from the fact that $(L^2(A)\ot H)\ot E$ is dense in
$\g(L^2(A;H),E)$.

To prove the boundedness of $\widetilde M$ we let $\phi: A\to H\ot E$ be a
finite rank simple function 
which is kept fixed throughout the
proof. Since we are fixing $\phi$ there is no loss of generality if we assume
$H$ to be finite-dimensional, say with orthonormal basis $(h_n)_{n=1}^N$. 
Also, by virtue of the strong measurability of $M$, we may assume that the
$\sigma$-algebra $\calA$ is countably generated. This implies that $L^2(A)$ is
separable, say with orthonormal basis $(g_m)_{m\ge 1}$.

{\em Step 1} --  In this step 
we consider the special case of the theorem where $M$ is a simple function. By
passing to a
common refinement we may suppose that 
$$ \phi = \sum_{j=1}^k \one_{B_j} U_j, \qquad M = \sum_{j=1}^k \one_{B_j}M_j,$$
with disjoint sets $B_j\in \calA$ of finite positive measure; the operators
$U_j\in H\ot E$ are of finite rank and the operators $M_j$ belong to
$\mathscr{M}$. 
Then,
$$ M\phi = \sum_{j=1}^k \one_{B_j} M_jU_j. $$ This is a simple
function with values in $H\otimes F$ which defines an operator 
 $T_{M\phi}\in \g(L^2(A;H),F)$, and 
\begin{align*} \n T_{M\phi}\n_{ \g(L^2(A;H),F)}^2
& = \E \Big\n\sum_{j=1}^k \sum_{n=1}^N\g_{jn} \sqrt{\mu(B_j)} M_j \Phi_j h_n \Big\n^2
\\ & \le (\g(\mathscr{M}))^2 \E \Big\n\sum_{j=1}^k \sum_{n=1}^N\g_{jn} \sqrt{\mu(B_j)} \Phi_j h_n \Big\n^2
\\ & =  (\g(\mathscr{M}))^2\n T_{\phi}\n_{ \g(L^2(A;H),E)}^2.
\end{align*}

{\em Step 2} -- 
Let $(A_j)_{j\ge 1}$ be a generating collection of sets in $\calA$ and let, for all $k\ge
1$, $\calA_k := \sigma(A_1,\dots,A_k)$.
Define the functions $M_k:A\to \calL(E,F)$ by $$M_k x :=
\E(Mx|\calA_k).$$ Since $\calA_k$ is a finite $\sigma$-algebra, $M_k$ is 
a simple function.  
It is easily checked that for all $f\in L^2(A;H)$ we have 
$T_{M_k\phi} f = T_{M\phi} \E(f|\calA_k)$,
and therefore 
$$\limk T_{M_k\phi} f = T_{M\phi} f$$ strongly in $F$.
By the $\g$-Fatou lemma (Proposition \ref{prop:g-Fatou}) it follows that
$T_{M\phi}\in \g_\infty(L^2(A;H),E)$ and 
$$ \n T_{M\phi} \n_{\g_\infty(L^2(A;H),E)} 
\le \liminf_{k\to\infty}\n T_{M\phi}\n_{ \g(L^2(A;H),F)} 
\le \g(\mathscr{M})\n T_{\phi}\n_{ \g(L^2(A;H),E)}.
$$
\end{proof}

It appears to be an open problem whether the operator $\widetilde M$ actually takes 
values in $\g(L^2(A;H),E)$ even in the simplest possible setting 
$A = (0,1)$ and $H=\R$. Of course, an affirmative answer for Banach spaces $E$
not containing an isomorphic copy of $c_0$ is obtained through an application of 
Theorem \ref{thm:HJ-K}.

We continue with some examples of $\g$-bounded families.
The first two results are due to {\sc Weis} \cite{Wei}.

\begin{example}
Let $(A,\calA,\mu)$ be a $\sigma$-finite measure space and let $\cT$ be a
$\g$-bounded subset of $\calL(E,F)$. Suppose $f:A\to \calL(E,F)$
is a function with the following 
properties:
\ben
\item[\rm(i)]
the function $\xi\mapsto f(\xi)x$ is
strongly $\mu$-measurable for all $x\in E$;
\item [\rm(ii)]
we have $f(\xi)\in \cT$
for $\mu$-almost all $\xi\in A$.
\een
For $\phi\in L^1(A)$ define $T_f^\phi\in \calL(E,F)$ by
$$ T_f^\phi x := \int_A \phi(\xi) f(\xi)x\,d\mu(\xi), \qquad x\in E,$$
The family $\cT_f^\phi 
: =\{T_f^\phi: \ \n \phi\n_1 \leq 1\}$ 
is $\g$-bounded and 
$ \g(\cT_f^\phi)\le \g(\cT).$
\end{example}

\begin{example}
Let $f: (a,b) \to \calL(E,F)$ be continuously differentiable with $$\int_a^b
\n{f'(s)}\n\, ds < \infty.$$ 
Then $\cT_f:= \{f(s): \ s\in (a,b)\}$ is $\g$-bounded and
$ \g(\cT_f) \le \n f(a)\n + \int_a^b \n f'(s)\n\,ds.$
\end{example}

The next example is taken from  {\sc Hyt\"onen} and {\sc Veraar}
\cite{HytVer09}. A related example, where Fourier type instead of type is used
and the cotype is not taken into account, is due to {\sc Girardi} and {\sc Weis}
\cite{GirWei03c}.

\begin{example}\label{ex:HV}
If $X$ has type $p$ and cotype $q$, then the range of any function
$f\in B_{r,1}^{d/r}(\R^d;\calL(X,Y))$ is $\g$-bounded. Here
$B_{r,1}^{d/r}(\R^d;\calL(X,Y))$ is the Besov space of exponents $(r,1,d/r)$.
\end{example}

The next example is due to {\sc Kaiser} and {\sc Weis} \cite{KaiWei} (first
part) and {\sc Hyt\"onen} and {\sc Veraar} \cite{HytVer09} (second part).

\begin{example}
Define, for every $h\in H$, the operator $U_h: E\to \g(H,E)$
by
$$ U_h x := h\ot x, \quad x\in E.$$
If $E$ has finite cotype, the family $\{U_h: \ \n h\n\le 1\}$ is $\g$-bounded.
Dually, 
define, for every $h\in H$, the operator $M_h: \g(H,E)\to E$
by
$$ M_h T := Th, \quad T\in \g(H,E).$$
If $E$ has finite type, the family $\{M_h: \ \n h\n\le 1\}$ is $\g$-bounded.
\end{example}
    
The final example is due to {\sc Haak} and {\sc Kunstmann} \cite{HaakKun} and
{\sc van Neerven} and {\sc Weis} \cite{NeeWei07}; it extends a previous result
for $L^p$-spaces of {\sc Le Merdy} \cite{LeM}.

\begin{example} $(A,\calA,\mu)$ be a $\sigma$-finite measure space, let $E$ have
property $(\a)$ (see Definition \ref{def:alpha} below) and let 
 $\phi :A\to\calL(E)$ be a strongly measurable function  
with the property that integral operators with kernel $\phi x$  
belong to $\g(L^2(A),E)$ for all $x\in E$. 
For $g\in L^2(A)$ we may define an operator
$T_g\in\calL(E)$ by
$$T_g x := \int_A g\, \phi x\,d\mu.$$  
Then
the family $\{T_g: \ \n g\n_{L^2(A)}\le 1\}$ is $\g$-bounded.
\end{example}
   
This list of examples could be enlarged {\em ad libitum}. We refrain from doing
so and refer instead to the references cited after Definition
\ref{def:g-bounded}.

\section{The ideal property}\label{sec:ideal}

Our next aim is to prove that $\g(H,E)$ is an operator ideal in
$\calL(H,E)$.
The proof of this fact relies on a classical domination result for finite
Gaussian sums in
$E$. Although a more general comparison principle for Gaussian random variables
will be presented in Section \ref{sec:domination}, we shall give an elementary
proof which is taken from {\sc Albiac} and {\sc Kalton} \cite{AlbKal}. 

\begin{lemma}[Covariance domination I]\label{lem:g-compar}
Let $x_1,\dots,x_M$ and
$y_1,\dots,y_N$ be elements of $E$ satisfying 
$$ \sum_{m=1}^M \lb x_m, x\s\rb^2 \le  \sum_{n=1}^N \lb y_n, x\s\rb^2$$
for all $x\s\in E\s$. Then for all $1\le p<\infty$,
$$ \E\Big\n \sum_{m=1}^M \g_m x_m\Big\n^p \le \E\Big\n\sum_{n=1}^N \g_n
y_n\Big\n^p.$$
\end{lemma}
\begin{proof}
Denote by $F$ the linear span of $\{x_1,\dots, x_M, y_1,\dots,y_N\}$ in $E$.
Define $Q\in \calL(F\s,F)$ by
$$ Qz\s:=  \sum_{n=1}^N \lb y_n, z\s\rb y_n - \sum_{m=1}^M \lb x_m, z\s\rb x_m,
\qquad z\s\in F\s.$$
The assumption of the theorem implies that $\lb Qz\s,z\s\rb\ge 0$ for all
$z\s\in F\s$,
and it is clear that $\lb Qz_1\s,z_2\s\rb = \lb Qz_2\s,z_1\s\rb$ for all 
$z_1\s,z_2\s\in F\s$.
Since $F$ is finite-dimensional,
by linear algebra we can find a sequence $(x_j)_{j=M+1}^{M+k}$ in $F$ such that
$Q$ 
is represented as
$$ Qz\s = \sum_{j=M+1}^{M+k} \lb x_j,z\s\rb x_j, \quad z\s\in F\s.$$
Now, 
$$ 
\sum_{m=1}^{M+k} \lb x_m, z\s\rb^2 =  \sum_{n=1}^N \lb y_n, z\s\rb^2, \qquad
z\s\in F\s.
$$
The random variables 
$X:=\sum_{m=1}^{M+k} \g_m x_m$ and $Y:=\sum_{n=1}^N \g_n'  y_n$ 
have Fourier transforms
$$ 
\bal
\E \exp(-i\lb X,x\s\rb) &= 
\exp\big(\!-\frac12\sum_{m=1}^{M+k}\lb x_m,x\s\rb^2\big), \\
\E \exp(-i\lb Y,x\s\rb) &= 
\exp\big(\!-\frac12\sum_{n=1}^{N}\lb y_n,x\s\rb^2\big).
\eal
$$
Hence by the preceding identity and the uniqueness theorem for the Fourier
transform,  
$X$ and $Y$  are identically distributed.
Thus, for all $1\le p<\infty$,
$$
 \E\Big\n \sum_{m=1}^{M+k} \g_m x_m\Big\n^p = \E'\Big\n\sum_{n=1}^N \g_n'
y_n\Big\n^p.
$$
Noting that
$$ \E\Big\n \sum_{m=1}^{M} \g_m x_m\Big\n^p \le \E\Big\n \sum_{m=1}^{M+k} \g_m
x_m\Big\n^p, $$
the proof is complete. This inequality follows, e.g., by noting that if $X$ and
$Y$ are independent $E$-valued random variables, with $Y$ symmetric,
then for all $1\le p<\infty$ we have
$ \E \n X\n^p \le \E \n X+Y\n^p.$
Indeed, since $X-Y$ and $X+Y$ are identically distributed, by the triangle
inequality we have
$(\E \n X\n^p)^\frac1p
\le \tfrac12(\E \n X-Y\n^p)^\frac1p + \tfrac12(\E \n X+Y\n^p)^\frac1p =(\E \n
X+Y\n^p)^\frac1p.$
\end{proof}

We continue with a result which describes what is arguably the most important
property of spaces of $\g$-radonifying operators, the so-called ideal property.
It can be traced back to {\sc Gross} \cite[Theorem 5]{Gro62}. 

\begin{theorem}[Ideal property]\label{thm:ideal}
Let $H$ and $H'$ be Hilbert spaces and $E$ and $E'$
Banach spaces. For all $S\in\calL(H',H)$, $T\in\g_\infty(H,E)$,  
and $U\in\calL(E,E')$ we have $UTS\in\g_\infty(H',E')$ and
$$ \n U TS \n_{\g_\infty(H',E')} \le \n U\n \, \n T\n_{\g_\infty(H,E)} \n
S\n.$$
If $T\in\g(H,E)$, then $UTS\in\g(H',E')$ and
$$ \n U TS \n \le \n U\n \, \n T\n \n S\n.$$
\end{theorem}
\begin{proof}
The left ideal property is trivial. Thus the first assertion it suffices to prove that 
if $T\in \g_\infty(H,E)$, then $TS\in\g_\infty(H',E)$ 
and
$ 
 \n TS\n_{\g_\infty(H',E)} \le \n T\n_{\g_\infty(H,E)}\n S\n
$.

Let $(h_j')_{j=1}^k$ be any finite orthonormal system in  $H'$.
Denote by $\wt H'$, $\wt {H}$, $\wt E$ the spans in $H'$, $H$, $E$ 
of $(h_j')_{j=1}^k$,
$(Sh_j')_{j=1}^k$, $(TSh_j')_{j=1}^k$ respectively.
Then $T$ and $S$ restrict to operators $\wt T: \wt H\to \wt E$ and
$\wt S: \wt {H'}\to \wt H$.

Let $(\wt h_m)_{m=1}^M$ be an orthonormal basis for $\wt H$. 
For all $x\s\in \wt E\s$ we have
$$
 \sum_{j=1}^{k} \lb TS h_j', x\s\rb^2
  = \n \wt S\s \wt T\s x\s\n_{\wt H}^2
\le \n \wt S\s \n^2 \, \n \wt T\s x\s\n_{\wt H}^2
=  \n \wt S\n^2\,
 \sum_{m=1}^{M} \lb T\wt h_m, x\s\rb^2.
 $$
Hence, by Lemma \ref{lem:g-compar},
$$ \E\Big\n \sum_{j=1}^{k} \g_j TS h_j'\Big\n^2
\le  \n S\n^2\,\E\Big\n \sum_{m=1}^M \g_m T\wt h_m\Big\n^2
\le \n S\n^2\,\n T\n_{\g(H,E)}^2.
$$ 
The desired inequality follows by taking the supremum over all
finite orthonormal systems in $H'$.

Next let $T\in \g(H,E)$ be given. If
$T\in H\ot E$ is a finite rank operator, say $T= \sum_{n=1}^N h_n\otimes x_n$, 
then $TS = \sum_{n=1}^N S\s h_n\otimes x_n$ belongs to $H'\ot E$.
Hence $TS\in\g(H',E)$, and by Proposition
\ref{prop:norm-gamma} and the estimate above we have 
$  \n TS\n_{\g(H',E)} \le \n T\n_{\g(H,E)}\n S\n$.
For general $T\in \g(H,E)$ the result now follows by approximation.
\end{proof}

As a first application we show that arbitrary bounded Hilbert space operators
$S\in \calL(H_1,H_2)$ extend to bounded operators $\wt S\in
\calL(\g(H_1,E),\g(H_2,E))$ in a natural way.

\begin{corollary}[{\sc Kalton} and {\sc Weis} \cite{KalWei07}]
Let $H_1$ and $H_2$ be Hilbert spaces. For all $S\in \calL(H_1,H_2)$
the mapping $$ \wt S: h\ot x \mapsto Sh \ot x, \quad h\in H_1, \ x\in E,$$
has a unique extension to a bounded operator $\wt S\in
\calL(\g(H_1,E),\g(H_2,E))$
of the same norm.
\end{corollary}
\begin{proof} For rank one operators $T = h\ot x$ we have
$\wt S T h' = [h,S\s h']x = T S\s h'.$ 
By linearity, this shows that for all $T\in H\ot E$ we have $\wt S T = T\circ
S\s$. The boundedness of $\tilde S$ now follows from the right ideal property,
which also gives the estimate $\n\wt S\n\le \n S\n$. The reverse estimate is
trivial.
\end{proof}

If $\mathscr{S}\subseteq \calL(H_1,H_2)$ is a uniformly bounded family of
Hilbert space operators, the family 
$\mathscr{\wh S}\subseteq \calL(\g(H_1,E),\g(H_2,E))$ is uniformly bounded as
well. If $E$ has the so-called property $(\a)$ (see  (see Definition
\ref{def:alpha}), then
$\mathscr{\wh S}$ is actually $\g$-bounded (see Section \ref{sec:g-boundedness}
for the definition). This result is due to {\sc Haak} and {\sc Kunstmann}
\cite{HaakKun}.

We continue with two convergence results, taken from 
{\sc Cox} and {\sc van Neerven} \cite{CoxNee} and {\sc van Neerven, Veraar,
Weis} \cite{NVW07a}.

\begin{corollary}[Convergence by left multiplication]\label{cor:g-cont} 
If $E$ and $F$ are Banach spaces and $U_n, U\in \calL(E,F)$ satisfy $\limn U_n =
U$ strongly, then for all $T\in
\g(H,E)$ we have $U_n T = UT$ in $\g(H,F)$.
\end{corollary}

\begin{proof} 
Suppose first that $T$ is a finite rank operator, say $T = \sum_{j=1}^k
h_j\otimes x_j$ with $h_1,\dots, h_k$ orthonormal in $H$ and $x_1,\dots, x_k$
from $E$.
Then
$$\limn \n U_nT - UT\n_{\g(H,F)}^2 = \limn\E \Big\n \sum_{j=1}^k \g_j
(U_n-U)x_j\Big\n^2 = 0.$$ 
The general case follows from the density of the finite rank operators in
$\g(H,E)$, the norm estimate $\n U_n T- U T\n_{\g(H,F)}\le \n U_n-U\n \n
T\n_{\g(H,E)}$, and the uniform boundedness of the operators $U_n$.
\end{proof} 

\begin{corollary}[Convergence by right multiplication]\label{cor:ga-conv}
If $H$ and $H'$ are Hilbert spaces and $S_n,S\in \calL(H',H)$ satisfy
$\limn S_n^*= S^*$ strongly, 
then for all $T\in \g(H,E)$ we have $\limn T S_n = TS$ in
$\g(H',E)$.
\end{corollary}
\begin{proof}
By the uniform boundedness principle, the strong convergence $\limn S_n\s = S\s$
implies $\sup_{n\ge 1}\n S_n\n<\infty$. Hence 
by the estimate $\|T \circ (S_n-S)\|_{\g(H',E)}\leq \|T\|_{\g(H,E)}\|S_n-S\|$ 
it suffices to consider finite rank operators $T\in \g(H,E)$, say $T =
\sum_{m=1}^M h_m \otimes x_m$. If $h_1',\dots,h_k'$ are orthonormal in $H'$,
then by the triangle inequality,
\begin{align*}
\Big(\E \Big\n \sum_{j=1}^k \g_j T \circ (S-S_n)h_j'\Big\n^2\Big)^\frac12
& = \Big(\E \Big\n \sum_{m=1}^M \sum_{j=1}^k \g_j\,
[h_m,(S-S_n)h_j']x_m\Big\n^2\Big)^\frac12\\
& \le \sum_{m=1}^M \Big(\E \Big\n \sum_{j=1}^k  \g_j\,
[h_m,(S-S_n)h_j']x_m\Big\n^2\Big)^\frac12
\\ & = \sum_{m=1}^M\n x_m\n \Big(\E \Big| \sum_{j=1}^k  \g_j\, [(S\s-S_n
\s)h_m, h_j']\Big|^2\Big)^\frac12
\\ & \le \sum_{m=1}^M \|x_m\| \|S^* h_m - S_n^* h_m\|.
\end{align*}
Taking the supremum over all finite orthonormal systems in $H'$, from
Proposition \ref{prop:norm-gamma} we obtain
\[\|T \circ (S-S_n)\|_{\g(H',E)} \leq
\sum_{m=1}^M \|x_m\| \|S^* h_m - S_n^* h_m\|.\]
The right-hand side tends to zero as $n\to\infty$.
\end{proof}

Here is a simple illustration:

\begin{example}\label{ex:g-approx} 
Consider an operator $R\in \g(H,E)$ and let
$(h_n)_{n\ge 1}$ be an orthonormal basis for $({\rm ker}(R))^\perp$
(recall that this space is separable; see the discussion preceding Corollary
\ref{cor:ONB}).
Let $P_n$ denote the orthogonal projection in $H$ 
onto the span of $\{h_1,\dots, h_n\}$. 
Then $\limn RP_n =R$ in $\g(H,E)$. 
\end{example}

\begin{corollary}[Measurability]\label{cor:meas}
Let $(A,\calA,\mu)$ be a $\sigma$-finite measure space and $H$ a separable
Hilbert space. For a function $\phi:A\to \g(H,E)$ define $\phi h:A\to E$ by
$(\phi h)(t):=  \phi(t)h$ for $h\in H$. The following assertions are
equivalent:
\ben
\item[\rm(1)] $\phi$ is strongly $\mu$-measurable;
\item[\rm(2)] $\phi h$ is strongly $\mu$-measurable for all $h\in H$.
\een
\end{corollary}
\bpf It suffices to prove that (2) implies (1). If $(h_n)_{n\ge 1}$ is an
orthonormal
basis for $H$, then with the notations of the Example \ref{ex:g-approx} for
all $ \xi\in A$ we have  
$$\phi(\xi) =\limn \phi(\xi) P_n = \limn \sum_{j=1}^n [\,\cdot\,, h_j]\phi(\xi)
h_j,$$
with convergence in the norm of $\g(H,E)$. The result now follows from the
measurability of the right-hand side.
\epf

\section{Gaussian random variables}\label{sec:Gaussian}

An $\R^d$-valued random variable $X=  (X_1,\dots,X_d)$ is called {\em Gaussian}
if every linear combination $\sum_{j=1}^d c_j X_j$ is Gaussian. Noting that
$\sum_{j=1}^d c_j X_j = \lb X,c\rb$ with $c = (c_1,\dots,c_d)$, this suggests
the following definition.

\begin{definition}\label{def:Gaussian}
An $E$-valued random variable is called {\em Gaussian} if the real-valued random
variables $\lb X,x\s\rb$ are Gaussian for all $x\s\in E\s$. 
\end{definition}

Gaussian random variables have good integrability properties:

\begin{proposition}[{\sc Fernique}]\label{prop:fernique}
Let $\mathscr{X}$ a uniformly tight family of $E$-valued Gaussian random
variables.
Then there exists a constant $\b>0$ such that 
$$ \sup_{X\in\mathscr{X}} \E \exp(\b\n X\n^2) <\infty.$$
\end{proposition} 

\begin{proof} We follow {\sc Bogachev} \cite{Bog} and {\sc Fernique} \cite{Fer}.

For each $X\in\mathscr{X}$ let $X'$ be an independent copy of $X$. 
Then $X-X'$ and $X+X'$ are identically distributed.
Hence, for all $t\ge s > 0$,
\beq\label{eq:Fernique}\begin{aligned}
& \P\{\n X\n  \le s\}\cdot\P\{\n X'\n > t\}
\\ & \qquad\qquad =\P\Bigl\{ 
\Big\n \frac{X+X'}{\sqrt{2}}\Big\n\le s\Bigr\}
\cdot \P\Bigl\{\Big\n\frac{X-X'}{\sqrt{2}}\Big\n> t\Bigr\}
\\ & \qquad\qquad\le \P\Bigl\{\bigl|\,\n X\n  -\n X'\n \,\bigr| 
\le s\sqrt 2, \ \n X\n +\n X'\n  >  t \sqrt 2 \Bigr\}
\\ & \qquad\qquad \stackrel{(*)}{\le} \P\Bigl\{\n X\n  > 
\frac{t-s}{ \sqrt 2},\ \n X'\n  >  \frac{t-s}{ \sqrt 2}\Bigr\}
\\ &  \qquad\qquad  = \P\Bigl\{\n X\n  > \frac{t-s}{\sqrt 2}\Bigr\} 
\cdot \P\Bigl\{\n X'\n  > \frac{t-s}{\sqrt 2}\Bigr\},
\end{aligned}
\eeq
where in $(*)$ we used that  
$$\bigl\{|\xi-\eta|\le s\sqrt 2
\ \hbox{ and }\ \xi+\eta> t\sqrt 2  \bigr\}\subseteq \Bigl\{
\xi> \frac{t-s}{\sqrt 2}\ \hbox{ and }\ \eta >
\frac{t-s}{\sqrt 2}
\Bigr\}.
$$

By the uniform tightness of $\mathscr{X}$, there exists $r\ge 0$ such that
$ \P\{\n X\n\le r\} \ge \frac34$ for all $X\in\mathscr{X}.$
Then
$$\a_0 :=   \frac{\P\{\n X\n  > r\}}{ \P\{\n X\n  \le r\}} \le
\frac13.$$
Define $t_0:=r$ and $t_{n+1}:= r+\sqrt 2 t_n$ for $n\ge 0$. 
By induction it is easy to check that
$
t_{n} = r(1+\sqrt 2)\bigl((\sqrt 2)^{n+1}-1\bigr).
$
Put
$$\a_{n+1} :=\frac{\P\{\n X\n  > t_{n+1}\}}{ \P\{\n X\n  \le r\}}, \qquad n\ge
0.$$
By \eqref{eq:Fernique} and the fact that $X$ and $X'$ are identically
distributed,
$$\a_{n+1} 
=\frac{\P\{\n X\n  > r+\sqrt 2 t_n \}}{ \P\{\n X\n  \le r\}}
\le \left(\frac{\P\{\n X\n  > t_n\}}{ \P\{\n X\n  \le r\}}\right)^2 =
\a_n^2, \qquad\forall n\ge 0.
$$
Therefore, $\a_n\le \a_0^{2^n}\le 3^{-2^n}$
and
$
 \P\{\n X\n  > t_n\}
 =\P\{\n X\n   \le r\}\cdot\a_n \le \frac1{3^{2^n}}.
$
With $\b := (1/(24r^2))\log 3$ we have, for any $X\in\mathscr{X}$,
$$\begin{aligned}
  \E \exp(\b \n X\n^2)  & \le \P\{\n X\n  \le t_0\} \cdot \exp(\b t_0^2) 
+\sum_{n\ge 0}\P\{t_n < \n X\n  \le t_{n+1}\} \cdot 
 \exp(\b t_{n+1}^2)
\\ &\le \exp(\e r^2)
+ \sum_{n\ge 0}  \frac1{3^{2^n}}
\exp\Bigl(\b  r^2 (1+\sqrt 2)^2 \bigl((\sqrt 2)^{n+2}-1)^2\bigr)\Bigr)
\\ &\le \exp(\e r^2)+ 
\sum_{n\ge 0} \exp\Bigl(2^{n}\Bigl[-\log 3 + 4\b  r^2 (1+\sqrt 2)^2\Bigr]\Bigr),
\end{aligned}
$$
where we used that $t_0=r$ and $4(1+\sqrt 2)^2 < 24.$ 
By the choice of $\b$, the sum on the right-hand side if finite.
\end{proof}

It is known that 
$$ \E \exp\big(\frac1{2\a^2}\n X\n^2\big)<\infty$$ if and only if
$\a^2>\sigma_X^2$, where 
$$ \sigma_X^2 = \sup_{\n x\s\n\le 1} \E |\lb X,x\s\rb|^2$$
is the {\em weak variance} of $X$; see
{\sc Marcus} and {\sc Shepp} \cite{MarShe} and  {\sc Ledoux} and {\sc Talagrand}
\cite[Corollary 3.2]{LedTal}.

Fernique's theorem (or rather the much weaker statement that $\E\n
X\n^2<\infty$) allows us to define
the {\em covariance operator} of a Gaussian random variable $X$ as the operator
$Q\in\calL(E\s,E)$ by
$$ Q x\s := \E \lb X,x\s\rb X.$$ 
Noting that $\E \lb X,x\s\rb^2 = \lb Qx\s,x\s\rb$, the Fourier transform of $X$
can be expressed in terms of $Q$ by
$$ \E \exp(-i\lb X,x\s\rb) = \exp(-\tfrac12\lb Qx\s,x\s\rb).$$

If $T\in \g(H,E)$ is a $\g$-radonifying operator and $W$ is an $H$-isonormal
 process, then $W(T)$ is a Gaussian random variable. We shall prove next that
every Gaussian random variable $X:\O\to E$ canonically arises in this way.
To this end we define the Hilbert space $H_X$ as the closed linear span in
$L^2(\O)$ of the random variables $\lb X,x\s\rb$. The inclusion mapping $W_X :
H_X\to L^2(\O)$ is an isonormal process.

\begin{theorem}[{\sc Karhunen-Lo\`eve}]\label{thm:KL}
Let $X$ be an $E$-valued Gaussian random variable. Then the linear operator
$T_X: H_X\to E$ defined by
$$ T_X \lb X,x\s\rb := \E \lb X,x\s\rb X,$$
is bounded and belongs to $\g(H_X,E)$, and we have $$W_X(T_X) = X.$$ 
\end{theorem}
\begin{proof}
For all $x\s,y\s\in E\s$ we have
$$\bal |\lb T_X \lb X,x\s\rb, y\s\rb| & \le \E |\lb X,x\s\rb \lb X,y\s\rb| 
\le \n\lb X,x\s\rb\n_{L^2(\O)} \n\lb X,y\s\rb\n_{L^2(\O)}
\\ & =  \n\lb X,x\s\rb\n_{H_X} \n\lb X,y\s\rb\n_{L^2(\O)}\le M_X\n\lb
X,x\s\rb\n_{H_X} \n y\s\n,
\eal
$$
where $M_X$ is the norm of the bounded operator from $E\s$ to $L^2(\O)$ defined
by $x\s\mapsto \lb X,x\s\rb$.
This proves that $T_X$ is a bounded operator of  norm $\n T_X\n\le M_X$. To
prove that $T_X\in\g(H_X,E)$ we check the assumptions of Theorem
\ref{thm:functionals}: for all $x\s\in E\s$ we have $T_X\s x\s = \lb X,x\s\rb$
and therefore
$ W_X (T_X\s x\s) = W_X(\lb X,x\s\rb) = \lb X,x\s\rb.$
\end{proof}

These results are complemented by the next characterisation of $\g$-radonifying
operators in terms of Gaussian random variables.

\begin{theorem}\label{thm:X}
For a bounded linear operator $T\in\calL(H,E)$ the following are equivalent:
\ben
\item[\rm(1)] $T\in \g(H,E)$;
\item[\rm(2)] there exists an $E$-valued Gaussian random variable $X$ satisfying
$$\E \lb X,x\s\rb^2 = \n T\s x\s\n^2, \quad x\s\in E\s.$$
\een
In this situation we have $\n T\n_{\g(H,E)}^2 = \E \n X\n^2.$
\end{theorem}
\begin{proof}

(1)$\Rightarrow$(2): Take $X = W(T)$, where $W$ is any $H$-isonormal process.

(2)$\Rightarrow$(1): Let $G$ be the closure of the range of $T\s$ in $H$. 
Then  $ W(T\s x\s) := \lb X,x\s\rb$ defines a $G$-isonormal process, 
and Theorem \ref{thm:functionals} implies that $T\in \g(G,E)$. Since $T\equiv 0$
on $G^\perp$ it follows that $T\in \g(H,E)$. 

To prove the final identity we note that for all 
$x\s\in E\s$ we have $\E\lb W(T),x\s\rb^2=\E\lb X,x\s\rb^2$. This implies that
the Gaussian random variables $W(T)$ and $X$ are identically distributed.
Therefore by Proposition \ref{prop:Ito},
$ \E \n X\n^2 = \E\n W(T)\n^2 = \n T\n_{\g(G,E)}^2 = \n T\n_{\g(H,E)}^2.$
\end{proof}

\section{Covariance domination}\label{sec:domination}

Our next aim is to generalise the simple covariance domination inequality of
Lemma \ref{lem:g-compar}. 
 
We begin with a classical inequality for Gaussian random variables
with values in $\R^d$ due to {\sc Anderson} \cite{And}. 
The Lebesgue measure of a Borel subset $B$ of $\R^d$ is denoted by $|B|$. 
 
\begin{lemma}
If $C$ and $K$ are symmetric convex subsets of $\R^d$, 
then for all $x\in\R^d$ we have
$$ |(C-x)\cap K| \le  |C\cap K|.$$
\end{lemma} 
\bpf
By the Brunn-Minkowski inequality (see {\sc Federer} \cite[Theorem
3.2.41]{Fed}),
$$ |\tfrac12(C+x)\cap K + \tfrac12[(C-x)\cap K]|^\frac1d \ge 
\tfrac12 |(C+x)\cap K|^\frac1d + \tfrac12 |(C-x)\cap K|^\frac1d.$$
Now $(C-x)\cap K =-[(C+x)\cap K]$ and therefore $|(C-x)\cap K| =|(C+x)\cap K|.$
Plugging this into the estimate and raising both sides to the power $d$ we
obtain
$$|\tfrac12(C+x)\cap K + \tfrac12[(C-x)\cap K]|\ge |(C-x)\cap K|.$$
Since
$ \tfrac12[(C+x)\cap K] + \tfrac12[(C-x)\cap K] \subseteq C \cap K $
this gives the desired inequality.
\epf

Recall our convention that Gaussian random variables are always centred.

\begin{theorem}[{\sc Anderson}]
Let $X$ be an $\R^d$-valued Gaussian random variable
and let $C\subseteq \R^d $ be a symmetric convex set.
Then for all $x\in \R^d$ we have
$$ \P\{X+x\in C\}\le \P\{X\in C\}.$$ 
\end{theorem}
\bpf 
If $K$ is symmetric and convex, then by the lemma,
$$\int_{\R^d}\one_{C-x}(y)\one_K(y)\,dy 
\le \int_{\R^d} \one_{C}(y)\one_K(y)\,dy.
$$
Approximating $y\mapsto \exp(-\frac12 y^2)$ from below by positive 
linear combinations of indicators of symmetric convex sets,
with monotone convergence we conclude that
$$
\bal
\P\{X+x\in C\}& =\frac1{\sqrt{(2\pi)^d}} \int_{\R^d}
\one_{C-x}(y)\exp(-\tfrac12 |y|^2)\,dy \\ &\le  \frac1{\sqrt{(2\pi)^d}}
\int_{\R^d}
\one_{C}(y)\exp(-\tfrac12 |y|^2)\,dy  = \P\{X\in C\}.\eal
$$
 \epf

As an application of Anderson's inequality we have the following comparison
result for $E$-valued Gaussian random variables (see {\sc Neidhardt} \cite[Lemma
28]{Nei}).

\begin{theorem}[Covariance domination II]\label{thm:Anderson}
Let $X_1$ and $X_2$ be 
Gaussian random variables with values
in $E$. If for all $x\s\in E\s$ we have
$$\E\lb X_1,x\s\rb^2\le \E\lb X_2,x\s\rb^2,$$
then for all closed convex symmetric sets $C$ in $E$ we have
$$ \P_1\{X_1\not\in C\}\le \P_2\{X_2\not\in C\}.$$
\end{theorem}
\begin{proof}
We proceed in two steps.

{\em Step 1} - 
First we prove the theorem for $E=\R^d$.
Let $Q_1$ and $Q_2$ denote the covariance matrices of $X_1$ and $X_2$. The
assumptions of the theorem imply that the matrix $Q_2-Q_1$ is symmetric and
non-negative definite, and
therefore it is the covariance matrix of some Gaussian random variable $X_3$
with values in $\R^d$.
On a possibly larger probability space $(\tilde\O,\tilde\F,\tilde\P)$ let
$\tilde X_j$ be independent copies of $X_j$.
Then 
$\tilde X_1+\tilde X_3$ has covariance
matrix $Q_1 + (Q_2-Q_1) = Q_2$. Hence, by Fubini's theorem and 
Anderson's inequality,
$$
 \P\{X_2\in C\} 
= \tilde \P\{\tilde X_1+\tilde X_3\in C\} 
\le \P\{X_1\in C\}.
$$

{\em Step 2} - We will reduce the general case to
the finite-dimensional case by a procedure known as {\em cylindrical
approximation}.
Let $X_1$ and $X_2$ be Gaussian random variables with values in a 
real Banach space $E$. By strong measurability, $X_1$ and $X_2$ take their
values in a 
separable closed subspace of $E$ almost surely and therefore we may assume
that $E$ itself is separable. 

For each $u\in\complement C$ there exists an element
$x_u\s\in E\s$ such that  
$\lb u,x_u\s\rb > 1$ and $\lb x,x_u\s\rb \le 1$ for all $x\in C.$ 
Since $C$ is symmetric, we also have 
$-\lb x,x_u\s\rb\le 1$ for all $x\in C.$ 
Choose balls $B_u$ with centres $u$ such that $\lb v,x_u\s\rb > 1$ for all $v\in
B_u$. The family $\{B_u: \ u\in \complement C\}$ is an open cover of
$\complement C$ and by the Lindel\"of property of $E$ it has a countable
subcover $\{B_{u_n}: \ n\ge 1\}$. Let us write $B_n:= B_{u_n}$ and $x_n\s:=
x_{u_n}\s$.

Put $$C_N := \{x\in E: \ |\lb x,x_n\s\rb|\le 1, \ n=1,\dots, N\},\qquad N\ge
1.$$
Each $C_N$ is convex and symmetric, we have
$C_1\supseteq C_2 \supseteq \dots$ and, noting that $u\not\in C_N$ for all $u\in
B_N$,  
$\bigcap_{N\ge 1} C_N = C.$

Define $\R^N$-valued Gaussian variables by
$X_{j,N} := T_N X_j$ for $j=1,2$,
where $T:E\to \R^N$ is given by
$ T_Nx := (\lb x,x_1\s\rb,\dots, \lb x,x_N\s\rb)$.  
The covariances of $X_{j,N}$ are given by
$T_NQ_jT_N\s$, and for all
$\xi\in\R^N$ we have
$$ \lb  T_NQ_1T_N\s \xi,\xi\rb = \lb Q_1 T_N\s \xi,T_N\s\xi\rb
\le \lb Q_2 T_N\s \xi,T_N\s\xi\rb= \lb  T_NQ_2T_N\s \xi,\xi\rb.$$
Hence, by what we have already proved,
$$
\P\{X_2\in C_N\} = \P\{X_{2,N}\in [-1,1]^N\}
\le \P\{X_{1,N}\in [-1,1]^N\} = \P\{X_1\in C_N\}.
$$
Upon letting $N\to\infty$ we obtain
$\P \{X_2\in C\}\le\P\{X_1\in C\}.$
\end{proof}

\begin{corollary}\label{cor:convex-f}
Let $X_1$ and $X_2$ be Gaussian random variables with
values in $E$ and assume that
for all $x\s\in E\s$ we have
$$ \E\lb X_1,x\s\rb^2\le \E\lb X_2,x\s\rb^2.$$ 
Suppose $\phi: E\to [0,\infty)$ is lower semi-continuous, convex and symmetric.
Then,
$$ \E \,\phi(X_1) \le \E \,\phi(X_2).$$
\end{corollary}
\begin{proof}
By the assumptions of $\phi$, 
for each $r\ge 0$ the set $C_r := \{x\in E: \ \phi(x)\le r\}$ 
is closed, convex and symmetric.
Therefore, by Theorem \ref{thm:Anderson}, 
$$\P\{ \phi(X_1)\le r\} = \P\{X_1\in C_r\} \ge  \P\{X_2\in C_r\} = 
\P\{ \phi(X_2)\le r\}.$$
Hence,
$$ \E \,\phi(X_1) = \int_0^\infty \P\{ \phi(X_1) > r\}\,dr
\le  \int_0^\infty \P\{ \phi(X_2) > r\}\,dr = \E \,\phi(X_2).
$$
\end{proof}

In particular we obtain that $ \E \n X_1\n^p\le \E \n X_2\n^p$ for all $1\le
p<\infty$; this extends Lemma \ref{lem:g-compar}.

Our next aim is to deduce from Theorem \ref{thm:Anderson} a domination theorem
for Gaussian covariance operators (Theorem \ref{thm:domination} below). 
The proof is based on standard reproducing kernel Hilbert space arguments;
classical references are {\sc Aronszajn} \cite{Aro} and {\sc Schwartz}
\cite{Schw}. We have already employed reproducing kernel arguments implicitly
with the introduction of the space $H_X$ in the course of proving Theorem
\ref{thm:KL}. In the absence of Gaussian random variables $X$, a somewhat more
abstract approach is necessary. 

The starting point is the trivial observation that covariance operators
$Q\in\calL(E\s,E)$ of $E$-valued Gaussian random variables are {\em positive}
and {\em symmetric}, i.e.,
$\lb Qx\s,x\s\rb\ge 0$ for all $x\s\in E\s$ and 
$\lb Qx\s,y\s\rb = \lb Qy\s,x\s\rb$ for all $x\s,y\s\in E\s.$ 

Now let $Q\in\calL(E\s,E)$ be an arbitrary positive symmetric operator. On the 
range of $Q$, the formula
$$
[Qx\s, Qy\s]_{H_Q} := \lb Qx\s, y\s\rb
$$
defines an inner product $[\,\cdot\,,\cdot\,]_{H_Q}$. 
Indeed,
if $Qx\s=0$, then $[Qx\s, Qy\s]_{H_Q} = \lb Qx\s, y\s\rb = 0$, and
if $Qy\s=0$, then $[Qx\s, Qy\s]_{H_Q} = \lb Qx\s, y\s\rb = \lb Qy\s, x\s\rb = 0$
by the symmetry of $Q$. This shows that $[\,\cdot\,,\cdot\,]_{H_Q}$ is well
defined. Moreover, if 
$[Qx\s, Qx\s]_{H_Q} = \lb Qx\s,x\s\rb=0$, then by the Cauchy-Schwarz
inequality we have, for all
$y\s\in E\s$, 
$$|\lb Qx\s,y\s \rb|\le \lb Qx\s,x\s\rb^\frac12 \lb Qy\s,y\s\rb^\frac12=0.$$
Therefore, $Qx\s=0.$

Let $H_Q$ be the real Hilbert space 
obtained by completing the range of
$Q$ with respect to $[\,\cdot\,,\cdot\,]_{H_Q}$. 
From 
$$\n Qx\s\n_{H_Q}^2 = \lb Qx\s,x\s\rb = 
|\lb Qx\s,x\s\rb| \le \n Q\n_{\calL(E\s,E)} \n x\s\n^2
$$  
we see that $Q$ is bounded from $E\s$ into $H_Q$, 
with norm $\le \n Q\n_{\calL(E\s,E)}^{\frac12}$. 
From
$$|\lb Qx\s,y\s\rb| \le \n Qx\s\n_{H_Q}\n Qy\s\n_{H_Q} 
\le \n Qx\s\n_{H_Q}\n Q\n_{\calL(E\s,H_Q)} \n y\s\n$$
it then follows that 
$$\n Qx\s\n \le \n Q\n_{\calL(E\s,H_Q)} \n Qx\s\n_{H_Q}.$$
Thus, the identity mapping $Qx\s\mapsto Qx\s$ on the range of $Q$ 
has a unique extension 
to a bounded linear operator, denoted by $i_Q$, from $H_Q$ into $E$
and its norm satisfies 
$\n i_Q\n \le \n Q\n_{\calL(E\s,H_Q)}$.

The pair $(i_Q,H_Q)$ is the {\em reproducing kernel Hilbert
space (RKHS)} associated with $Q$. 

\begin{remark}
In the special case where $Q$ is the covariance operator of an $E$-valued
Gaussian random variable $X$, then $H_Q$ and the space $H_X$ introduced in the
proof of Theorem \ref{thm:KL} are canonically isometric by means of the mapping
$i_Q\s x\s\mapsto \lb X,x\s\rb$.
\end{remark}

The next proposition has its origins in the work of {\sc Gross} \cite{Gro62,
Gro67}; see also {\sc Baxendale} \cite{Bax}, {\sc Dudley, Feldman, Le Cam}
\cite{DFL}, {\sc Kallianpur} \cite{Kal71}, {\sc Kuelbs} \cite{Kue}, and {\sc
Sat\^o} \cite{Sat}.

\begin{proposition}\label{prop:RKHS1}
Let $(i_Q, H_Q)$ be the RKHS associated with the positive symmetric operator
$Q\in\calL(E\s,E)$. 
The mapping $i_Q: H_Q\to E$ is injective and we have the identity
$$
Q = i_Q\circ i_Q\s.
$$
As a consequence, $Q$ is the covariance operator of an $E$-valued Gaussian
random variable $X$ if and only if $i_Q\in \g(H_Q,E)$. In this situation we have
$$\E \n X\n^2 = \n i_Q\n_{\g(H_Q,E)}^2.$$
\end{proposition}
\begin{proof}
Given an element $x\s\in E\s$ we denote by $h_{x\s}$ the element in $H_Q$
represented by $Qx\s$. With this notation we have 
$i_Q(h_{x\s}) = Qx\s$ and 
$$  [h_{x\s}, h_{y\s}]_{H_Q} = \lb Qx\s,y\s\rb.$$
For all $y\s\in E\s$ we then have
$$
[h_{x\s}, h_{y\s}]_{H_Q} = \lb Qx\s,y\s\rb = \lb i_Q(h_{x\s}),y\s\rb
= [h_{x\s}, i_Q\s y\s]_{H_Q}.
$$
Since the elements $h_{x\s}$ span a dense subspace of $H_Q$ it follows that
$h_{y\s} = i_Q\s y\s.$ Therefore,
$$Qy\s = i_Q(h_{y\s}) = i_Q(i_Q\s y\s)$$ for all $y\s\in E\s$,
and the identity $Q = i_Q\circ i_Q\s$ follows. Finally if $i_Q g = 0$ for some
$g\in H_Q$, then for all $y\s\in E\s$ we have
$$[g, h_{y\s}]_{H_Q} =  [g, i_Q\s y\s]_{H_Q} =\lb i_Q g, y\s\rb = 0,$$
and therefore $g = 0$. This proves that $i_Q$ is injective.
\end{proof}

For an interesting addendum to the second part of this theorem we refer to {\sc
Mathieu} and {\sc Fernique} \cite{MatFer}. Using a deep regularity result for
Gaussian processes due to {\sc Talagrand}, they prove that if $Q\in\calL(E\s,E)$
is positive and symmetric, then $i_Q\in\g(H,E)$ if and only if there exists a
sequence $(h_n)_{n\ge 1}$ in $H$ such that the following two conditions are
satisfied:
\begin{enumerate}
\item[\rm(i)] $\limn \n h_n\n^2 \log n = 0$;
\item[\rm(ii)] $\n i_Q h\n \le \sup_{n\ge 1} |[h,h_n]|$ for all $h\in H$.
\end{enumerate}

\begin{proposition}\label{prop:RKHS2}
If $Q,R\in\calL(E\s,E)$ are positive symmetric operators such that
$$ \lb Rx\s,x\s\rb\le \lb Qx\s,x\s\rb, \quad x\s\in E\s,$$
then as subsets of $E$ we have $i_R(H_R)\subseteq i_Q(H_Q)$, and this inclusion
mapping induces a 
contractive embedding $H_R \embed H_Q$.
\end{proposition} 
\begin{proof}
By the Cauchy-Schwarz inequality, for each $x\s\in E\s$ the mapping
$i_Q\s y\s\mapsto \lb R\s x\s,y\s\rb$
extends to a bounded linear functional $\phi_{x\s}$ on $H_Q$ of norm $\n
\phi_{x\s}\n \le \n i_R\s x\s\n$. By the Riesz representation theorem there
exist a unique element $h_{x\s}\in H_Q$ such that
$ [i_Q\s y\s, h_{x\s}] = \lb Rx\s,y\s\rb$ for all $y\s\in E\s$.
Then
$$ \lb i_Q h_{x\s}, y\s\rb = \lb Rx\s,y\s\rb = \lb i_R i_R\s x\s,y\s\rb.$$
This shows that $i_Q h_{x\s} = i_R i_R\s x\s$. The contractive embedding $H_R
\embed H_Q$ we are looking for is therefore given by $i_R\s x\s\mapsto h_{x\s}$.
\end{proof}

\begin{theorem}[Covariance domination III]\label{thm:domination} Let
$Q\in\calL(E\s,E)$ 
be the covariance operator of an $E$-valued Gaussian random variable $X$.
Let ${\mathcal R}$ be the set of positive symmetric operators  $R\in
\calL(E\s,E)$ satisfying
$$ \lb Rx\s,x\s\rb\le \lb Q x\s,x\s\rb, \quad x\s\in E\s.$$ 
Then each $R\in{\mathcal R}$ is the  covariance operator of a 
$E$-valued Gaussian random variable $X_R$ 
and the family $\{X_R: \ R\in{\mathcal R}\}$ is uniformly tight.
Moreover, for all $R\in{\mathcal R}$ and all $1\le p<\infty$ we have
$$
 \E \n X_R\n^p \le \E \n X\n^p.
$$
\end{theorem}
\begin{proof}
By the second part of Proposition \ref{prop:RKHS1} we have $i_Q\in \g(H_Q,E)$.
By the right ideal property, for all $R\in \mathscr{R}$ we have $i_R = i_Q\circ
i_{R,Q}\in \g(H_R,E)$, where $i_{R,Q}:H_R\embed H_Q$ is the embedding of
Proposition \ref{prop:RKHS2}. Hence by the second part of Proposition
\ref{prop:RKHS1}
there exists an $E$-valued  Gaussian random variables $X_R$ with covariance
operator $R$.

Let $\e>0$ be arbitrary and fixed, and choose a compact set
$K\subseteq E$ such that $\P\{X\in K\}\ge 1-\e$.
By replacing $K$ by its convex symmetric hull, which is still compact,
we may assume that $K$ is convex and symmetric.
In view of 
$$\E \lb X_R,x\s\rb^2 = \lb Rx\s,x\s\rb \le \lb Qx\s,x\s\rb = \E \lb
X,x\s\rb^2,$$
from Theorem 
\ref{thm:Anderson} we obtain that
$\P\{X_R\in K\} \ge \P\{X\in K\}\ge 1-\e$. 
\end{proof}

\section{Compactness}\label{sec:compactness}

Recall that a sequence of $E$-valued random variables $(X_n)_{n\ge 1}$
is said to {\em converge in distribution} to an $E$-valued random variable $X$
if $\limn \E f(X_n) = \E f(X)$
for all $f\in C_{\rm b}(E)$ (see Section \ref{sec:preliminaries}). 
As it turns out, it is possible to allow certain unbounded functions $f$.

\begin{lemma}\label{lem:sq-w-conv} Let $(X_n)_{n\ge 1}$ be a sequence 
of $E$-valued random variables converging in distribution to a random variable
$X$.
Let $\phi:E\to [0,\infty)$ be a Borel function
with the property that
$$ \sup_{n\ge 1}\, \E \phi(X_n) < \infty.$$ 
If $f:E\to\R$ is a continuous function with the property that
$$
|f(x)| \le c(\n x\n)\phi(x), \qquad x\in E,$$
where $c(r)\downarrow 0$ as $r\to\infty$,
then
$$ \lim_{n\to\infty} \E f(X_n) =  \E f(X).$$
\end{lemma}
\begin{proof}
Put
$$
f_R(x): =\left\{
\begin{array}{rl}
R, & \hbox{if } f(x)> R, \\
f(x), & \hbox{if } -R\le f(x)\le  R, \\
-R, & \hbox{if } f_R(x)< - R, \\
\end{array}\right.
$$ Then $f_R\in C_{\rm b}(E)$ and
\beq\label{eq:fR}
\lim_{n\to\infty} \E f_R(X_n) = \E f_R(X).
\eeq
We also deduce that
\beq\label{eq:est-d(R)}
\bal 
  \lim_{R\to\infty} \Bigl(
\sup_{n\ge 1}\E |f(X_n) - f_R(X_n)|
\Bigr)
 & \le \lim_{R\to\infty} \Bigl(
\sup_{n\ge 1}\E \big(\one_{\{|f(X_n)|> R\}} |f(X_n)|\big)
\Bigr)
\\ & \le \lim_{R\to\infty} c(\delta(R))
\sup_{n\ge 1} \E \big(\one_{\{|f(X_n)|> R\}} 
\phi(X_n)\big),
\eal 
\eeq
where $$\delta(R):= \sup\{\delta\ge 0: \ |f(x)|\le R \ \hbox{ for all $\n
 x\n\le \d$}\}.$$
From $\lim_{R\to\infty} \delta(R) =\infty$
we see that the right-hand side of \eqref{eq:est-d(R)} tends to $0$ as
$R\to\infty$. 
Combined with \eqref{eq:fR}, this gives the desired result.
\end{proof}

The main result of this section gives a necessary and sufficient
condition for relative compactness in the space $\ga(H,E)$. In a rephrasing in
terms of sequential convergence in $\ga(H,E)$, this result is due to {\sc
Neidhardt} \cite{Nei}.

\begin{theorem}\label{thm:g-rad-conv}
Let $W$ be an $H$-isonormal process. For a subset $\cT$ of $\ga(H,E)$ the
following assertions are equivalent:
\begin{enumerate}
\item[\rm(1)] the set $\cT$ is relatively compact in $\ga(H,E)$;
\item[\rm(2)] the set $\{W(T):\  T\in\cT\}$ is relatively compact in
$L^2(\O;\E)$;
\item[\rm(3)] the set $\{W(T):\  T\in\cT\}$ is uniformly tight and for all
$x\s\in
E\s$ the set $\{T\s x\s:\   T\in\cT\}$ is relatively compact in $H$.
\end{enumerate}
\end{theorem}

\begin{proof}

(1)$\Leftrightarrow$(2): \ This is immediate from the fact that $W$ is
isometric.
 
(1)$\Rightarrow$(3): 
By the continuity of $T\mapsto T\s x\s$,  $\{T\s x\s:\   T\in\cT\}$ is
relatively compact in $H$.
It remains to prove that the set $\{W(T):\   T\in\cT\}$ is uniformly tight.
For this it suffices to prove that every sequence in this set has a subsequence
which is uniformly tight. Let $(T_n)_{n\ge 1}$ be a sequence in $\cT$ and set $
X_n := W(T_n)$. By passing to a subsequence we may assume that $(T_n)_{n\ge 1}$
is convergent in $\g(H,E)$. 

We shall prove that the sequence $(X_n)_{n\ge 1}$ is uniformly tight.
Fix $\e>0$ and choose $m_0\ge 1$ so large that $2^{2-2m_0} < \e$. 
For every $m\ge m_0$ we choose $N_m\ge 1$ so large that 
$$\n T_n-T_{N_m}\n_{\g(H,E)}\le 2^{-{2}m}\qquad\forall n\ge N_m.$$
Let $X_{n,m}:= W(T_n-T_{N_m})$.
By Chebyshev's inequality, for $n\ge N_m$ we have
$$ 
\P\{\n X_{n,m}\n \ge 2^{-m}\}) 
\le 2^{2m} \E \n X_{n,m}\n^2 =  2^{2m}\n T_n-T_{N_m}\n_{\g(H,E)}^2\le 2^{-2m}.
$$
For $m\ge m_0$ we also choose compact sets $K_m\subseteq E$ such that
$$ \P\{X_n\in K_m\} \ge 1-2^{-2m}, \qquad 1\le n\le N_m,$$
and let
$ V_m := \{x\in E: \ d(x,K_m) < 2^{-m}\}.$
For $n\ge N_m$ we have
$$
 \P\{X_n\not\in V_m\}
 \le \P\{\n X_n-X_{N_m}\n\ge 2^{-m}\} + \P\{X_{N_m}\not\in K_{m}\}
 \le 2^{-2m}+2^{-2m} = 2^{1-2m}.
$$
On the other hand, for $1\le n\le N_m$ we have
$$ \P\{X_n\not\in  V_m\}\le \P\{X_n\not\in K_m\}\le 2^{-2m} \le  2^{1-2m}.$$
It follows that the estimate $\P\{X_n\not\in V_m\}\le 2^{1-2m}$ holds for all
$n\ge 1$. 

Let
$$ K := \overline{\bigcap_{m\ge m_0} V_m}.$$
If finitely many open balls $B(x_i,2^{-m})$ cover $K_m$, then the open balls
$B(x_i,3\cdot 2^{-m})$ cover $\overline{V_m}$.
Hence $K$ is totally bounded and therefore compact.
For all $n\ge 1$,
$$ \P\{X_n\not\in K\} \le \sum_{m\ge m_0} \P\{X_n \not\in V_m\}
\le \sum_{m\ge m_0} 2^{1-2m} < 2^{2-2m_0} < \e.
$$
This proves that $(X_n)_{n\ge 1}$ is uniformly tight.

(3)$\Rightarrow$(1): Let $(T_n)_{n\ge 1}$ be a sequence in  $\cT$.
We must show that its contains a Cauchy subsequence. 

Choose a separable closed subspace $E_0$ of $E$ such that each $X_n = W(T_n)$
takes values in $E_0$ almost surely. 
Noting that the weak$\s$-topology of the closed unit ball in $E_0\s$ is
metrisable, we can choose a sequence $(x_j\s)_{j\ge 1}$ in $E\s$ whose
restrictions to $E_0$ are weak$\s$-dense in the closed unit ball of $E_0\s$.
After passing to a subsequence we may assume that for all $j\ge 1$ the sequence
$(T_n\s x_j\s)_{j\ge 1}$ converges in $H$
and that the sequence $(X_n)_{n\ge 1}$ converges in distribution. We claim that
$\lim_{n,m\to\infty} X_{n}-X_m = 0$ in distribution. To see this fix arbitrary
sequences $n_k\to \infty$ and $m_k\to\infty$.
After passing to a subsequence of the indices $k$
we may assume that $(X_{n_k} - X_{m_k})_{k\ge 1}$ 
converges in distribution to some $E_0$-valued random variable $Y$. 
Taking Fourier transforms
we see that for all $j\ge 1$,
$$
\bal
\E \exp(-i\lb Y, x_j\s\rb ) & = \lim_{k\to\infty}\E \exp(-i\lb
X_{n_k}-X_{m_k},x_j\s\rb)
\\ &  = \lim_{k\to\infty}\exp(-\tfrac12\n T_{n_k}\s x_j\s -T_{m_k}\s x_j\s\n) =
1.
\eal
$$
It follows that 
$ \exp(-i\lb Y, x\s\rb) = 1$ for all $x\s\in E\s$, and therefore $Y=0$ by the
uniqueness theorem for the Fourier transform.
This proves the claim.

Thus, for all $f\in C_{\rm b}(E)$ we obtain
$$ \lim_{m,n\to\infty} \E f( X_n-X_m) = \E f(0).$$
By Lemma \ref{lem:sq-w-conv} combined with Proposition \ref{prop:fernique}
and Theorem \ref{thm:X},  
$$\lim_{m,n\to\infty} \n T_n-T_m\n_{\g(H,E)}^2 = \lim_{m,n\to\infty} \E\n
X_n-X_m\n^2 = 0.
$$
\end{proof}

Here is a simple application:

\begin{theorem}\label{thm:dom}
Let $\cT$ be a subset of $\calL(H,E)$ which is dominated in covariance by some
fixed
element $S\in \ga(H,E)$, in the sense that for all $T\in \cT$ and $x\s\in E\s$,
$$ \n T\s x\s\n \leq \n S\s x\s\n.$$
Then the following assertions are equivalent:
\ben
\item[\rm(1)] the set $\cT$ is relatively compact in $\ga(H,E)$;
\item[\rm(2)] the set $\{T\s x\s: T\in \cT\}$ is relatively compact in $H$ for
all $x\s\in E\s$. \een
\end{theorem}
\begin{proof}
By Theorem \ref{thm:domination} the family $\{W(T):\  T\in \cT\}$
is uniformly tight and therefore the result follows from Theorem
\ref{thm:g-rad-conv}.
\end{proof}

\begin{corollary}[$\g$-Dominated convergence]\label{cor:dom-conv}
Suppose $\limn T_n\s x\s = T\s x\s$ in $H$ for all $x\s\in E\s$.
If there exists $S\in \g(H,E)$ such that $$ 0\leq \n T_n\s x\s\n_H \leq \n S\s
x\s\n_H$$
for all $n\ge 1$ and $x\s\in E\s$, then 
$\limn T_n =T$ in $\g(H,E)$.
\end{corollary}

\section{Trace duality}\label{sec:trace_duality}

In this section we investigate duality properties of the spaces $\g(H,E)$.
As we shall see we have a natural identification $(\g(H,E))\s = \g(H,E\s)$ if
$E$ is a $K$-convex Banach space. In order to define the notion of $K$-convexity
we start with some preliminaries.

For a Gaussian sequence $\g = (\g_n)_{n\ge 1}$ we define projections 
$\pi_N^\g$ in $L^2(\O;E)$ by
\beq\label{eq:proj1} 
\pi_N^\g X := \sum_{n=1}^N \g_n\,\E (\g_n X).
\eeq
Identifying $L^2(\O;E\s)$ isometrically with a norming subspace of
$(L^2(\O;E))\s$, for all $X\s\in L^2(\O;E\s)$ we have 
\beq\label{eq:proj2}(\pi_N^\g)\s X\s = \sum_{n=1}^N \g_n\,\E (\g_n X\s).
\eeq

\begin{lemma}\label{lem:change-g}
If $\g = (\g_n)_{n\ge 1}$ and $\g' = (\g_n')_{n\ge 1}$ 
are Gaussian sequences, then for all $N\ge 1$ we have
$$ \n \pi_N^\g\n = \n \pi_N^{\g'}\n.$$
\end{lemma}
\begin{proof}
Define the bounded operator
$\pi_N$ on $L^2(\O;E)$ by
\beq\label{eq:proj3}
\pi_N X :=  \sum_{n=1}^N \g_n'\E(\g_n X), \quad X\in L^2(\O;E).
\eeq
On the closed subspace $L^2(\O;E\s)$ of $(L^2(\O;E))\s$, the adjoint operator
$\pi_N\s$ is given by
\beq\label{eq:proj4}\pi_N\s X\s = \sum_{n=1}^N \g_n \E(\g_n' X\s), 
 \quad X\s\in L^2(\O;E\s).
\eeq
Now let $X\in L^2(\O;E)$ be given.
Given $\e>0$ choose $Y\s\in L^2(\O;E\s)$ 
of norm one such that $(1+\e)|\lb \pi_N X,Y\s\rb|\ge \n \pi_N X\n_{L^2(\O;E)}$.
Then, first comparing \eqref{eq:proj1} and \eqref{eq:proj3}, and then 
\eqref{eq:proj2} and \eqref{eq:proj4},
\begin{align*}
\n \pi_N^{\g} X\n_{L^2(\O;E)} 
 = \n \pi_N X\n_{L^2(\O;E)}
 & \le (1+\e)|\lb \pi_N X,Y\s\rb|
\\ & = (1+\e)|\lb X,\pi_N\s Y\s\rb|
\\ & \le (1+\e) \n X\n_{L^2(\O;E)}\n \pi_N\s Y\s\n_{L^2(\O;E\s)} 
\\ & =(1+\e) \n X\n_{L^2(\O;E)}\n (\pi_N^{\g'})\s Y\s\n_{L^2(\O;E\s)} 
\\ & \le(1+\e)\n \pi_N^{\g'}\n  \n X\n_{L^2(\O;E)}.
\end{align*}
Since $\e>0$ was arbitrary this shows that $\n \pi_N^\g\n \le \n \pi_N^{\g'}\n$.
By reversing the roles of $\g$ and $\g'$ we also obtain the converse inequality
$\n \pi_N^{\g'}\n \le\n \pi_N^{\g}\n$.
\end{proof}

This allows us to define $$K_N(E):= \n \pi_N^\g\n.$$
Clearly, the numbers $K_N(E)$ are increasing with $N$.

\begin{lemma}\label{lem:sup} For any closed norming subspace $F$ of $E\s$ we
have
\begin{align*}
\ & \E\Big\n \sum_{n=1}^N \g_nx_n\Big\n^2 
\\ & \quad\le K_N^2(E)\,
\sup\Big\{ \Big|\sum_{n=1}^N \lb x_n,x_n\s\rb\Big|^2: \, x_1\s,\dots,x_N\s\in F,
\ \E\Big\n \sum_{n=1}^N \g_nx_n\s\Big\n^2 \le 1 \Big\}.
\end{align*}
\end{lemma}
\begin{proof}
Put $X:= \sum_{n=1}^N \g_n x_n.$ Since $L^2(\O;F)$ is isometric to a norming
closed subspace of $(L^2(\O;E))\s$,
given $\e>0$ we may choose $X\s\in L^2(\O;F)$ 
of norm one such that $(1+\e)|\lb X,X\s\rb|\ge \n X\n_{L^2(\O;E)}$.
Noting that $\pi_N^\g X = X$ and putting $x_n\s := \E (\g_n X\s)$ we obtain
$$ 
\bal
 \n X\n_{L^2(\O;E)}
& \le (1+\e)|\lb X,X\s\rb| = (1+\e)|\lb \pi_N^\g X,X\s\rb|
\\ & = (1+\e)|\lb X,(\pi_N^\g)\s X\s\rb| = (1+\e)\Big|\sum_{n=1}^n \lb x_n,
x_n\s\rb\Big|.
\eal
$$
Since $\e>0$ was arbitrary, the proof is concluded by noting that $x_n\s\in F$
and 
$$ \E\Big\n \sum_{n=1}^N \g_n x_n\s\Big\n^2 = \E\Big\n \sum_{n=1}^N \g_n \E
(\g_n X\s)\Big\n^2 = \E\n (\pi_N^\g)\s X\s\n^2 \le \n \pi_N^\g\n^2 = K_N^2(E).  
$$
\end{proof}

\begin{definition}
A Banach space $E$ is called {\em $K$-convex} if 
$$K(E):= \sup_{N\ge 1} K_N(E)$$
is finite. 
\end{definition}

Closed subspaces of $K$-convex spaces are $K$-convex. The next result shows that
$K$-convexity is a self-dual property:

\begin{proposition}\label{prop:K-convex-dual} 
A Banach space $E$ is $K$-convex if and only if its dual 
$E\s$ is $K$-convex, in which case we have $K(E) = K(E\s)$.
\end{proposition}
\begin{proof}
The identity \eqref{eq:proj2} shows that
$(\pi_N^E)\s = \pi_N^{E\s}$. As an immediate consequence we see
that if $E$ is $K$-convex, then $E\s$ is $K$-convex
and $K(E) = K(E\s)$. 
If $E\s$ is $K$-convex, then $E^{**}$ is $K$-convex, and
therefore its closed subspace $E$ is $K$-convex. 
\end{proof}

The notion of $K$-convexity has been introduced by {\sc Maurey} and {\sc Pisier}
\cite{MauPis} and was studied thoroughly in {\sc Pisier} \cite{Pis82,Pis89}.
Usually this notion is defined using Rademacher variables rather than Gaussian
variables. In fact, both definitions are equivalent. In fact, one may use an
argument similar to the one employed in Lemma \ref{lem:change-g} to pass from
the Gaussian definition to the Rademacher definition, and a central limit
theorem argument allows one to pass from the Rademacher definition to the
Gaussian definition. For the details we refer to {\sc Figiel} and {\sc
Tomczak-Jaegermann} \cite{FigTom} and {\sc Tomczak-Jaegermann} \cite{Tom89}.

\begin{example}
Every Hilbert space $E$ is $K$-convex and $K(E)=1$.
\end{example}

\begin{example}
Let $(A,\calA,\mu)$ be a $\sigma$-finite measure space and let $1<p<\infty$.
Then $L^p(A)$ is $K$-convex, and more generally if $E$ is $K$-convex then
then $L^p(A;E)$ is $K$-convex and $$K(L^p(A;E))\le \left\{
\begin{array}{ll}
K_{p,2}^\g K(E), & \hbox{if $2\le p<\infty$}, \\
K_{q,2}^\g K(E), & \hbox{if $1< p\le 2$ and $\tfrac1p+\frac1q=1$}
\end{array}
\right.$$ 
Here $K_{p,2}^\g$ and $K_{q,2}^\g$ are the Gaussian Kahane-Khintchine
constants. 

First let $2\le p<\infty$. The projections defined by \eqref{eq:proj1}
in $E$ and $L^p(A;E)$ will be denoted by $\pi_N^\g$ and $\pi_n^{\g,L^p(A;E)}$,
respectively. For $X\in L^2(\O;L^p(A;E))$ we obtain, using
Jensen's inequality,
Fubini's theorem, the Kahane-Khintchine inequality, and the $K$-convexity of
$E$,
\begin{align*}
 \E\n \pi_N^{\g,L^p(A;E)} X\n_{L^p(A;E)}^2 & 
=  \E \Big(\int_{A}
\Big\n\sum_{n=1}^N \g_n 
\E(\g_n X(\xi))\Big\n^p d\mu(\xi)\Big)^\frac{2}{p}
\\ &  \le \Big(\E\int_{A}
\Big\n\sum_{n=1}^N \g_n 
\E (\g_n X(\xi))\Big\n^p d\mu(\xi)\Big)^\frac{2}{p}
\\ & \le  (K_{p,2}^\g)^2\Big(\int_{A}
\Big(\E\Big\n\sum_{n=1}^N \g_n 
\E (\g_n X(\xi))\Big\n^2\Big)^\frac{p}{2}\,d\mu(\xi)\Big)^\frac{2}{p}
\\ & \le (K_{p,2}^\g)^2\n \pi_N^\g\n^2 \Big(\int_{A}
\big(\E\n  X(\xi)\n^2\big)^\frac{p}{2}\,d\mu(\xi)\Big)^\frac{2}{p}
\\ & = (K_{p,2}^\g)^2\n \pi_N^\g\n^2 \Big\n \E\n 
X(\xi)\n^2\Big\n_{L^{\frac{p}{2}}(A)}
\\ & \le (K_{p,2}^\g)^2\n \pi_N^\g\n^2 \E \Big\n \n 
X(\xi)\n^2\Big\n_{L^{\frac{p}{2}}(A)}
\\ & = (K_{p,2}^\g)^2\n \pi_N^\g\n^2\E \Big(\int_A \n X(\xi)\n^p\,d\mu(\xi)
\Big)^{\frac2p}
\\ &  = (K_{p,2}^\g)^2\n\pi_N^\g\n^2 \E\n X\n_{L^p(A;E)}^2.
\end{align*}
This proves the result for $2\le p< \infty$.

Next let $1<p<2$. We can identify $(L^q(A;E\s))$ isometrically with a closed
subspace of $(L^p(A;E))\s$, $\frac1p+\frac1q=1$.
Since $E\s$ is $K$-convex, by what we just proved the space $L^q(A;E\s)$ is
$K$-convex. 
Hence $L^p(A;E)$, being isometrically contained in the dual of $L^q(A;E\s)$, 
is $K$-convex, and $K(L^p(A;E)) \le K(L^q(A;E\s)) \le K_{q,2}^\g K(E\s)
= K_{q,2}^\g K(E\s)$. 
\end{example}

\begin{example}\label{ex:c0}
The space $c_0$ fails to be $K$-convex. 
To see this, let $(\g_n)_{n\ge 1}$ be a Gaussian sequence
and let  $(u_n)_{n\ge 1}$ be the standard basis of $c_0$. 
Set $$ X_N := \sum_{n=1}^N \sgn(\g_n)u_n.$$ 
We have $\n X_N\n_{L^2(\O;c_0)} =1 $ and
$\E(\g_nX_N) = \E |\g_n|u_n  = \sqrt{{\pi/2}}\,u_n$, so $$ \n \pi_N^\g X_N
\E\n_{c_0}^2 =\frac{\pi}{2} \E\Big\n\sum_{n=1}^N \g_n u_n \Big\n_{c_0}^2.$$
Arguing as in Example \ref{ex:LP}, the right hand side can be bounded from below
by a term which grows asymptotically like $\log N$. It follows that 
$\n \pi_N^\g\n\ge C \log N$.
\end{example}

A deep theorem of {\sc Pisier} \cite{Pis82} states that a Banach space $E$ is
$K$-convex if and only if $E$ has non-trivial type (the notion of type is
discussed in the next section). 
The following simple proof that every Banach space with type $2$ is $K$-convex
was given by {\sc Blasco, Tarieladze, Vidal} \cite{BTV}; see also {\sc
Chobanyan} and {\sc Tarieladze} \cite{ChoTar} and {\sc Maurey} and {\sc Pisier}
\cite{MauPis}.

\begin{proposition}\label{prop:type2-Kconvex}
If $E$ has type $2$, then $E$ is $K$-convex and $K^\g(E)\le T_2^\g(E)$.
\end{proposition}
\begin{proof}
Let $X = \sum_{j=1}^k 1_{\O_j}x_j$ be simple, with the measurable sets $\O_j$
disjoint and of positive probability.
Let $y_j:= \sqrt{\P(\O_j)}x_j$, so $$\E\n X\n^2 = \sum_{j=1}^k \n y_j\n^2$$
and $$\E\lb X,x\s\rb^2 = \sum_{j=1}^k \lb y_j,x\s\rb^2, \qquad x\s\in E\s.$$
Let $z_n:= \E (\g_n X)$. Then, by the orthonormality of Gaussian sequences in
$L^2$,
$$ \sum_{n=1}^N \lb z_n, x\s\rb^2 = \sum_{n=1}^N \E (\g_n \lb X,x\s\rb)^2
\le \E\lb X,x\s\rb^2 = \sum_{j=1}^k \lb y_j,x\s\rb^2.
$$
Hence by covariance domination,
$$
\bal
  \E \n \pi_N^\g X \n^2 
  & = \E \Big\n \sum_{n=1}^N \g_n z_n\Big\n^2
 \le  \E \Big\n \sum_{k=1}^k \g_j y_j\Big\n^2 
\\ & \le (T_2^\g(E))^2 \sum_{j=1}^k \n y_j\n^2 = (T_2^\g(E))^2 \E \n X\n^2.
\eal
$$
It follows that $\n \pi_N^\g\n \le T_2^\g(E).$ Since $N\ge 1$ was arbitrary this
gives $K(E)\le T_2^\g(E)$.
\end{proof}

The next result is essentially due to {\sc Pisier} \cite{Pis89}; its present
formulation was stated by {\sc Kalton} and {\sc Weis} \cite{KalWei07}. It
describes a natural pairing between 
$\g(H,E)$ and $\g(H,E\s)$, the so-called {\em trace duality}.

\begin{theorem}[Trace duality]\label{thm:Kconvex-gamma-dual}
For all $T\in H\ot E$ and $S\in H\ot E\s$
we have
$$ |{\rm tr} (S\s T)| \le \n T\n_{\g(H,E)}\n S\n_{\g(H,E\s)}.$$
As a consequence, for all $S\in \g(H,E\s)$ the mapping $\phi_S: T \mapsto 
{\rm tr} (S\s T)$ defines an element $\phi_S\in (\g(H,E))\s$ of norm 
$$\n \phi_S\n\le \n S\n_{\g(H,E\s)}.$$ If $E$ is $K$-convex, 
the mapping $\phi: S\mapsto \phi_S$ is an isomorphism of $\g(H,E\s)$ onto
$(\g(H,E))\s$ and $$\n S\n_{\g(H,E\s)} \le K(E)\n \phi_S\n.$$
\end{theorem}
\begin{proof}
For the proof of the first assertion we
may assume that $T = \sum_{n=1}^N h_n\otimes x_n$ and $S = \sum_{n=1}^N
h_n\otimes x_n\s$ with $h_1,\dots,h_N$ orthonormal in $H$.
Then,
$$
\bal
 |{\rm tr} (S\s T)| 
& = \Big|{\rm tr} \sum_{m=1}^N\sum_{n=1}^N \lb x_m,x_n\s\rb h_m\otimes h_n\Big|
 = \Big|\sum_{n=1}^N \lb x_n,x_n\s\rb\Big| 
\\ &  =  \Big| \E \Big\lb \sum_{m=1}^N \g_m x_m,  \sum_{n=1}^N \g_n
x_n\s\Big\rb\Big|
\le \n T\n_{\g(H,E)}\n S\n_{\g(H,E\s)}.
\eal
$$
Lemma \ref{lem:sup}, applied to the Banach spaces $E\s$ and the norming subspace
$E\subseteq E^{**}$, shows that 
$$ \n S\n_{\g(H,E\s)} \le K(E)\ \sup\Big\{ |\textrm{tr}(S\s T)| : \ \n
T\n_{\g(H,E)}\le 1\Big\} = K(E)\n \phi_S\n.$$
This shows that $\phi$ is an isomorphic embedding of $\g(H,E\s)$ into
$(\g(H,E))\s$. 

It remains to prove that $\phi$ is surjective.
To this end let $\Lambda\in (\g(H,E))\s$ be given. 
We claim that the bounded operator $S: H\to E\s$ defined by
$\lb x, Sh\rb = \lb h\ot x,\Lambda\rb $
belongs to $\g(H,E\s)$ and that $S = \Lambda$ in $(\g(H,E))\s$.  
Fix any finite orthonormal system $(h_n)_{n=1}^N$ in $H$. 
By Lemma \ref{lem:sup}, applied to $E\s$ and the norming subspace $E\subseteq
E^{**}$,
$$ 
\bal
\E \Big\n \sum_{n=1}^N \g_n Sh_n\Big\n^2 
&  \le K^2(E)
\sup
  \Big|\sum_{n=1}^N \lb x_n, S h_n\rb\Big|^2
\\ & = K^2(E) \sup 
  \Big|\sum_{n=1}^N \lb h_n\ot x_n,\Lambda\rb\Big|^2 = K^2(E) \n \Lambda\n^2.  
\eal
$$
Example \ref{ex:c0} shows that a $K$-convex subspace cannot contain an
isomorphic copy of $c_0$, and therefore an appeal to Theorem \ref{thm:c0}
finishes the proof.
 \end{proof}

Our final result relates the notion of $K$-convexity to isonormal processes.

\begin{theorem}
Let $E$ be $K$-convex and let $W:H\to L^2(\O)$ be an isonormal process. The 
closure of the range of the induced mapping $W: \g(H,E)\to L^2(\O;E)$ is the
range of a projection $P^W$ in $L^2(\O;E)$ of norm  $\n P^W\n \le K(E)$.
\end{theorem}
\begin{proof}
Let $(h_i)_{i\in I}$ be a maximal orthonormal system in $H$. We claim that the
projection $P_W$ is given as the strong operator limit 
$\lim_{J} P_J^W$, where 
$$ P_J^W X : =  \sum_{j\in J} \g_j \E(\g_j X),$$
with $\g_j = W(h_j)$. Here the limit is taken along the net of all finite
subsets $J$ of $I$.

To see that the strong limit exists, 
recall that 
every $X\in L^2(\O;E)$ can be approximated by
simple functions of the form $X=\sum_{n=N}^k \one_{A_n}\ot x_n$.
By the uniform boundedness of the projections $P_J^W$ and linearity 
it suffices to show that
the limit $\lim_J P_J^W X_n$ exists for each $X_n := \one_{A_n}\ot x_n$. 
But in $L^2(\O)$, the limit $\lim_J P_J^W \one_{A_n}$ exists by standard facts
about orthogonal projections in Hilbert spaces.

From $\n P_J^W\n \le K(E)$ for all finite subsets $J\subseteq I$ 
we infer $\n P^W\n\le K(E)$.
\end{proof}

\section{Embedding theorems}\label{sec:embedding}

As we have seen in Example \ref{ex:SI}, if $W: L^2(\R_+;H)\to L^2(\O)$ is an
isonormal process, then the induced isometric mapping
$$ W: \g(L^2(\R_+;H),E)\to L^2(\O;E)$$
can be interpreted as a stochastic integral. Indeed,
the stochastic integral of the $H\ot E$-valued function
$f\ot (h\ot x)$ can be defined by
$$ \int_0^\infty f\ot (h\ot x)\,dW:= W((f\ot h)\ot x),$$
and this definition extends by linearity to functions $\phi \in L^2(\R_+)\ot
(H\ot E)$. 
The isometric property of the induced mapping $W$ then expresses that
$$ \E \Big\n \int_0^\infty \phi\,dW \Big\n^2 = \n
T\phi\n_{\g(L^2(\R_+;H),E)}^2,$$
where $T: L^2(\R_+)\ot (H\ot E) \to (L^2(\R_+)\ot H)\ot E$ is the linear mapping
$$T (f\ot (h\ot x)):= (f\ot h)\ot x.$$
Since $(L^2(\R_+)\ot H)\ot E$ is dense in $\g(L^2(\R_+;H),E)$, the stochastic
integral has a unique isometric extension to $\g(L^2(\R_+;H),E)$.
It is therefore of considerable interest to investigate the structure of the
space $\g(L^2(\R_+;H),E)$. In this section we shall prove various embedding
theorems which show that suitable Banach spaces of $\g(H,E)$-valued functions
embed in $\g(L^2(\R_+;H),E)$. 

The simplest example of such an embedding occurs when $E$ has type $2$. 

\begin{definition}
A Banach space $E$ is said to have {\em type $p\in [1,2]$}
if there exists
a constant $C_p\ge 0$ such that for all finite sequences $x_1,\dots,x_N$ in $E$
we have
$$
\Big(\E \Big\n \sum_{n=1}^N r_n x_n\Big\n^2\Big)^\frac12
\le C_p \Big(\sum_{n=1}^N \n x_n\n^p\Big)^\frac1p.
$$
The space $E$ is said to have {\em cotype
$q\in [2,\infty]$}  if there exists
a constant $C_q\ge 0$ such that for all finite sequences $x_1,\dots,x_N$ in $E$
we have
$$ \Big(\sum_{n=1}^N \n x_n\n^q\Big)^\frac1q 
 \le C_q\Big(\E \Big\n \sum_{n=1}^N r_n x_n\Big\n^2\Big)^\frac12.
$$
For $q=\infty$ we make the obvious adjustment in this definition.
\end{definition}

The least constants in the above definitions are denoted by $T_p(E)$ and
$C_q(E)$, respectively, and are called the {\em type} and {\em cotype constant}
of $E$. 

\begin{remark}\label{rem:Gaussian-type-cotype}
In the definitions of type and cotype, the Rademacher variables may be replaced
by Gaussian random variables; this only affects the numerical values of the type
and cotype constants. 
The Gaussian type and cotype constants of a Banach space $E$ are denoted by 
$T_p^\g(E)$ and $C_q^\g(E)$, respectively.  
\end{remark}

It is easy to check that the inequalities defining 
type and cotype cannot be satisfied
for any $p>2$ and $q<2$, respectively, even in one-dimensional spaces $E$. 
This explains the restrictions imposed on
these numbers.

\begin{example}
Every Banach space has type $1$ and cotype $\infty$.
\end{example}

\begin{example}
Every Hilbert space has type $2$ and cotype $2$. A deep result
of {\sc Kwapie\'n} \cite{Kwa72} states that, conversely, every Banach space
with type $2$ and cotype $2$ is isomorphic to a Hilbert space.
\end{example}

\begin{example}
Let $(A,\calA,\mu)$ be a 
$\sigma$-finite measure space and let $1\le r<\infty$.
If $E$ has type $p$ (cotype $q$),  
then  $L^{r}(A;E)$ has type $\min\{p,r\}$ (cotype $\max\{q,r\}$).
In particular,
$L^r(A)$ has type $\min\{2,r\}$ and cotype $\max\{2,r\}$.

Let us prove this for the case of type, the case of cotype being similar.
If $r<p$ we may replace $p$ by $r$ and thereby assume that 
$1\le p \le r <\infty$; we shall prove that $L^r(A;E)$ has type $p$, with
$$T_p(L^r(A;E)) \le 
K_{2,r}K_{r,2}T_p(E)
=\left\{
\begin{array}{rl}
K_{2,r}T_p(E), & \hbox{ if $1\le r < 2$};\\
       T_p(E), & \hbox{ if $r=2$}; \\
K_{r,2}T_p(E), & \hbox{ if $2< r < \infty$}.
\end{array}
\right.
$$
Here $K_{2,r}$ and $K_{r,2}$ are the Kahane-Khintchine constants.
 
Let $f_1,\dots, f_N\in L^{r}(A;E)$.
By
using the Fubini theorem, the Kahane-Khintchine inequality,
type $p$,  H\"older's inequality, and 
the triangle inequality in $L^{\frac{r}{p}}(A)$, 
$$\bal 
\Big( \E \Big\n \sum_{n=1}^N r_n
f_n\Big\n_{L^{r}(A;E)}^{r}\Big)^\frac1{r}
& = \Big(\int_A \E \Big\n \sum_{n=1}^N r_n
 f_n(\xi)\Big\n^{r}\,d\mu(\xi)\Big)^\frac1{r} 
\\ & \le K_{r\!,2}\Big(\int_A \Big(\E \Big\n \sum_{n=1}^N r_n
 f_n(\xi)\Big\n^{2}\Big)^\frac{r}{2} \,d\mu(\xi)\Big)^\frac1{r} 
\\ & \le K_{r\!,2}T_p(E)\Big(\int_A \Big( \sum_{n=1}^N 
 \n f_n(\xi)\n^p\Big)^\frac{r}{p} \,d\mu(\xi)\Big)^\frac1{r}
\\ & = K_{r\!,2} 
T_p(E)\Big\n \sum_{n=1}^N \n f_n\n^p\Big\n_{L^\frac{r}{p}(A)}^\frac{1}{p} 
\\ & \le K_{r\!,2}
T_p(E)\Big(\sum_{n=1}^N\Big\n  \n
f_n\n^p\Big\n_{L^\frac{r}{p}(A)}\Big)^\frac{1}{p}
\\ & = 
K_{r\!,2} T_p(E) \Big(\sum_{n=1}^N \n f_n\n_{L^{r}(A;E)}^p\Big)^\frac1p.
\eal
$$
An application of the Kahane-Khintchine inequality to change moments in the
left hand side finishes the proof of the first assertion.
\end{example}

If a Banach space has type $p$ for some $p\in [1,2]$,  then it has type
$p'$ for all $p'\in [1,p]$;
if a Banach space has cotype $q$ for some $q\in [2,\infty]$,  then it has cotype
$q'$ for all $q'\in
[q,\infty]$. A simple duality argument shows that if $E$ has type $p$, then the
dual space $E\s$ has cotype $p'$, $\frac1p+\frac1{p'}=1$. 
If $E$ is $K$-convex and has cotype $p$, then the
dual space $E\s$ has type $p'$, $\frac1p+\frac1{p'}=1$. 
The $K$-convexity assumption cannot be omitted: $\ell^1$ has cotype $2$ while
its dual 
 $\ell^\infty$ fails to have non-trivial type.

The next theorem goes back to {\sc Hoffmann-J\o rgensen} and {\sc Pisier}
\cite{HofPis} and {\sc Rosi\'nski} and {\sc Suchanecki} \cite{RosSuc}; in its
present formulation it can be found in
{\sc van Neerven} and {\sc Weis} \cite{NeeWei05b}.

\begin{theorem}\label{thm:type2}
Let $(A,\calA,\mu)$ be a $\sigma$-finite measure space.
\begin{enumerate}
\item[\rm(1)]
If $E$ has type $2$, then the
mapping $(f\otimes h)\otimes x \mapsto f\otimes (h\otimes x)$ has a unique
extension to a continuous embedding
$$L^2(A;\g(H,E))\embed \g(L^2(A;H), E)$$
of norm at most $T_2(E)$.
Conversely, if the identity mapping $f\ot x \mapsto f\ot x$ extends to a bounded
operator from 
$L^\infty(0,1;E)$ to $\g(L^2(0,1), E)$, then $E$ has type $2$.
\item[\rm(2)] 
If $E$ has cotype $2$, then the
mapping $f\otimes (h\otimes x) \mapsto (f\otimes h)\otimes x$ has a unique
extension to a continuous embedding
$$\g(L^2(A;H), E)\embed L^2(A;\g(H,E))$$
of norm at most $C_2(E)$.
Conversely, if the identity mapping $f\otimes x \mapsto f\otimes $ extends to a
bounded operator from $\g(L^2(0,1), E)$ to
$L^1(0,1;E)$, then $E$ has cotype $2$.
\een
\end{theorem}

\begin{proof} 
We shall prove (1); the proof of (2) is very similar.

Let $(f_m)_{m=1}^M$ and $(h_n)_{n=1}^N$ be orthonormal systems in $L^2(A)$ and
$H$, respectively, with $f_m = c_m \one_{A_m}$ for suitable disjoint sets
$A_m\in \calA$; here $c_m := 1/\sqrt{\mu(A_m)}$ is a normalising constant.
Let $(\g_{mn})_{m,n\ge 1}$ be a Gaussian sequence on $(\O,\F,\P)$ and let
$(r_m')_{m\ge 1}$ be a Rademacher sequence on a second probability space
$(\O',\F',\P')$.
For each $\om'\in\O'$ the Gaussian sequences $(\g_{mn})_{m,n\ge 1}$ and
$(r_m(\om')\g_{mn})_{m,n\ge 1}$ are identically distributed. Averaging over
$\O'$, using Fubini's theorem and the type $2$ property of $L^2(\O;E)$,  we
obtain
\begin{align*}
\ & \Big\n \sum_{m=1}^M \sum_{n=1}^N (f_m\ot h_n)\ot
x_{mn}\Big\n_{\g(L^2(A;H),E)}^2
\\ & \qquad\qquad =
\E \Big\n \sum_{m=1}^M \sum_{n=1}^N \g_{mn} x_{mn}\Big\n^2  
\\ &  \qquad\qquad =
\E \E'\Big\n \sum_{m=1}^M r_m' \sum_{n=1}^N \g_{mn} x_{mn}\Big\n^2 
\\ &  \qquad\qquad\le T_2^2(E) \sum_{m=1}^M \E \Big\n \sum_{n=1}^N \g_{mn}
x_{mn}\Big\n^2 
\\ &  \qquad\qquad =  T_2^2(E)\sum_{m=1}^M \E \Big\n \sum_{n=1}^N \g_{n}
x_{mn}\Big\n^2
\\ &  \qquad\qquad =  T_2^2(E)\sum_{m=1}^M c_m^2\mu(A_m)
\E \Big\n \sum_{n=1}^N \g_{n} x_{mn}\Big\n^2
\\ &  \qquad\qquad =  T_2^2(E) \Big\n \sum_{m=1}^M \sum_{n=1}^N f_m\ot (h_n\ot
x_{mn})\Big\n_{L^2(A;\g(H,E))}^2.
\end{align*}
It is easy to check that elements of the form $\sum_{m=1}^M \sum_{n=1}^N (f_m\ot
h_n)\ot x_{mn}$ and $\sum_{m=1}^M \sum_{n=1}^N(f_m\ot (h_n\ot x_{mn})$ are dense
in $\g(L^2(A;H),E)$ and $L^2(A;\g(H,E))$, respectively. This gives the first
assertion.

The proof of the converse relies on the preliminary observation that in the
definition of type $2$ {\em we may restrict ourselves to vectors of norm one}.
To prove this we follow {\sc James} \cite{Jam}. Keeping in mind Remark
\ref{rem:Gaussian-type-cotype}, suppose there is a constant $C$ such that for
all $N\ge 1$ and all $x_1,\dots,x_N\in E$ of norm one we have $$
\Big(\E \Big\n \sum_{n=1}^N \g_n x_n\Big\n^2\Big)^\frac12
\le C_p \Big(\sum_{n=1}^N \n x_n\n^p\Big)^\frac1p.
$$
Now let $x_1,\dots,x_N\in E$ have integer norms, say $\n x_n\n = M_n,$
and let $(\g_{mn})_{m,n\ge 1}$ be a doubly indexed Gaussian sequence. Since
$\sum_{m=1}^{M_n^2} \g_{mn}$ and $M_n \g_n$ are identically distributed,
we have
$$ \E \Big\n \sum_{n=1}^N \g_n x_n\Big\n^2
= \E \Big\n \sum_{n=1}^N \sum_{m=1}^{M_n^2} \g_{mn} \frac{x_n}{\n x_n\n}\Big\n^2
\le C_2 \sum_{n=1}^N \sum_{m=1}^{M_n^2} \big\n \frac{x_n}{\n x_n\n}\big\n^2 
= C_2 \sum_{n=1}^N \n x_n\n^2.
$$
Upon dividing by a large common integer, this inequality extends
 to $x_1,\dots,x_N\in E$ having rational norms, and the general case follows
from this by approximation.
 
Suppose now that $E$ fails type $2$, and let $N\ge 1$ be fixed. By the
observation (and Remark \ref{rem:Gaussian-type-cotype}), there exist 
$x_1,\dots, x_M\in E$ of norm one such that
$$ \E\Big\n \sum_{m=1}^M \g_m x_m\Big\n^2 \ge N^2 \sum_{m=1}^M \n x_m\n^2.$$ 
Let $I_1,\dots,I_M$ be disjoint intervals in $(0,1)$ of measure $|I_m| = 1/M^2$.
 
Then, using that the functions $\sqrt{M}\one_{I_m}$ are orthonormal in
$L^2(0,1)$, 
\begin{align*}
 \Big\n \sum_{m=1}^M \one_{I_m}\ot x_{m}\Big\n_{\g(L^2(0,1),E)}^2
& =  \frac1M \E\Big\n \sum_{m=1}^M \g_m {x_m}\Big\n^2 
 \ge \frac{CN^2}{M}
\sum_{m=1}^M {\n x_m\n^2} 
\\ & = N^2
 = N^2\Big\n \sum_{m=1}^M \one_{I_m}\ot x_{m}\Big\n_{L^\infty(0,1;E)}.
\end{align*}
This shows that the identity mapping on $L^2(0,1)\ot E$
does not extend to a bounded operator from $L^\infty(0,1;E)$ into
$\g(L^2(0,1),E)$.
\end{proof}

Note that if $\phi:= f\ot (h\ot x)$, then $T_\phi:= f\ot(h\ot x)$ is the
operator given by
\begin{equation}\label{eq:Tphi}  T_\phi g =  \int_A \phi g\,d\mu, \quad g\in
L^2(A;H).
\end{equation}

\begin{corollary}
If the identity mapping
$f\ot x \mapsto f\ot x$ extends to an isomorphism $$ L^2(\R_+;E) \simeq
\g(L^2(\R_+),E),$$ 
then $E$ is isomorphic to a Hilbert space
\end{corollary}
\begin{proof}
By Theorem \ref{thm:type2}, $E$ has type $2$ and cotype $2$ and $E$ is
isomorphic to a Hilbert space by {\sc Kwapie\'n}'s theorem cited earlier.
\end{proof}

We continue with an example of {\sc van Neerven} and {\sc Weis} \cite{NeeWei05b}
which shows that in certain spaces without cotype $2$
there exist bounded strongly measurable functions $\phi:(0,1)\to\calL(H;E)$
such that the operator $T_\phi$ defined by \eqref{eq:Tphi} belongs to
$\g(L^2(0,1;H),E)$, even though $\phi(t)\not\in \g(H,E)$ for all $t\in (0,1)$.
 
\begin{example}\label{ex1} Let $H=\ell^2$ and $E=\ell^p$ with $2<p<\infty$.
For $k=1,2,\dots$ choose sets $A_k\subseteq (0,1)$
of measure $\frac1k$ in such a way that for all $t\in (0,1)$ we have
\begin{equation}\label{number}
 \#\{k\ge 1: \ t\in A_k\} = \infty.
\end{equation}
Define the operators $\phi(t): \ell^2\to \ell^p$ as coordinate-wise
multiplication
with the sequence $(a_1(t),a_2(t), \dots)$, where
\begin{equation}\label{infinite}
 a_k(t) = \left\{
\begin{array}{ll}
1, & \hbox{if} \ t\in A_k, \\
0, & \hbox{otherwise}.
\end{array}
\right.
\end{equation}
Then $\n \phi(t)\n =1$ for all $t\in (0,1)$ and none of the operators
$\phi(t)$ is $\gamma$-radonifying. 
Indeed, this follows from Proposition \ref{prop:sq-fc-Lp} below, according to
which we have $\phi(t)\in\g(\ell^2,\ell^p)$ if and only if
$$\sum_{k\ge 1}\n \phi\s(t) e_k\s\n_{\ell^2}^p < \infty,$$
where $e_k\s$ denote the $k$-th unit vector of
$l^q$ $(\frac1p+\frac1q =1$). 
By \eqref{number} and \eqref{infinite}, 
the sum $\sum_{k=1}^\infty \n \phi\s(t) e_k\s\n_{\ell^2}^p $ diverges for all
$t\in [0,1]$.

The associated operator $T_\phi: L^2(0,1;\ell^2)\to \ell^p$  
is well-defined and bounded, and we have
$$
\n T_\phi u_k\s\n_{\ell^2}^2 
 = \int_0^1 a_k^2(t)\,dt = |A_k| = \frac1k.
$$
Consequently,
$$ \sumk \n T_\phi u_k\s\n_{\ell^2}^p
= \sumk \frac1{k^\frac{p}{2}} < \infty
$$ 
and $T_\phi$ is $\g$-radonifying.
\end{example}

Using the scale of Besov spaces, a version of Theorem \ref{thm:type2}(1) can be
given for Banach spaces $E$ having type $p\in [1,2]$.
In {\sc van Neerven, Veraar, Weis} \cite{NVW07c}, it is shown by elementary
methods that if $E$ has type $p$, then for all Hilbert spaces $H$ the mapping
$f\ot (h\ot x) \mapsto (f\ot h)\ot x$ 
extends to a continuous embedding
$$ B_{p,p}^{\frac1p-\frac12}(0,1;\g(H,E))\embed \g(L^2(0,1;H),E).$$  
Conversely, by a result of {\sc Kalton, van Neerven, Veraar, Weis} \cite{KNVW},
if the identity mapping $f\ot x \mapsto f\ot x$ 
extends to a continuous embedding
$$ B_{p,1}^{\frac1p-\frac12}(0,1;E)\embed \g(L^2(0,1),E),$$ 
then $E$ has type $p$.

The first assertion is a special case of the main result of  {\sc Kalton, van
Neerven, Veraar, Weis} \cite{KNVW}, where arbitrary smooth bounded domains
$D\subseteq \R^d$ are considered. In this setting, the exponent 
 $\frac1p-\frac12$ has to be be replaced by $\frac{d}{p}-\frac{d}{2}$. It is
deduced from a corresponding result for $D=\R^d$ which is proved using 
Littlewood-Paley decompositions. This approach is less elementary but it leads
to stronger results. It also yields dual a characterization of spaces with
cotype $q\in [2,\infty]$. 

\section{$p$-Absolutely summing operators} 

Let $1\le p<\infty$. A bounded operator $T:E\to F$ is called 
{\em $p$-absolutely summing} if
if there exists a constant $C\ge 0$
such that for all finite sequences $x_1,\dots, x_N$ in $ E$
we have 
$$ \sum_{n=1}^N \n Tx_n \n^p
\le C^p\sup_{\n x\s\n\le 1} \sum_{n=1}^N |\lb
x_n,x\s\rb|^p.
$$
The least admissible constant $C$ is called the {\em $p$-absolutely summing
norm} of $T$, notation $\n T\n_{\pi_p(E,F)}$. 
 
It follows in a straightforward way from the definition that 
the space $\pi_p(E,F)$ of all $p$-absolutely summing operators from $E$ to $F$
is a Banach space with respect to the norm $\n \cdot\n_{\pi_p(E,F)}$.
We have the following two-sided ideal property:
if $S: E'\to E$ is bounded, 
$T: E\to F$ is $p$-absolutely summing, and $U:F\to F'$ is bounded,
then $UTS: E'\to F'$ is $p$-absolutely summing and
$$ \n UTS\n_{\pi_p(E',F')}\le \n U\n\n T\n_{\pi_p(E,F)}\n
S\n.$$

We shall prove next that $p$-absolutely summing operators
are $\g$-radonifying. The proof is an application of the 
{\em Pietsch factorisation theorem} (see {\sc Diestel, Jarchow, Tonge}
\cite{DJT})
which states that if $T$ is $p$-absolutely summing from $E$ to another Banach
space $F$, then there exists a Radon probability measure
$\nu$ on $(B_{E\s},{\rm weak}\s)$ such that for all $x\in E$ we have
$$ \n Tx\n^p \le \n T\n_{\pi_p(E,F)}^p
\int_{B_{E\s}}|\lb x, x\s\rb|^pd\nu(x\s).$$
  
Recall that $K_{p,q}^\g$ denote the Gaussian Kahane-Khintchine constants.
  
\begin{proposition}[{\sc Linde and Pietsch} \cite{LinPie}]\label{thm:thmLP}
If $T\in \pi_p(H,E)$ for some $1\le p<\infty$,
then $T\in \g(H,E)$ and 
$$\n T\n_{\g(H,E)}\le \max\{ K_{2,p}^\g, K_{p,2}^\g\} \n T\n_{\pi_p(H,E)}$$
\end{proposition}
\begin{proof}
Let $h_1,\dots,h_N$ be an orthonormal system in $H$.
Then, by the Pietsch factorisation theorem and the Fubini theorem,
\begin{align*}
\ &  \Big(\E\Big\n \sum_{n=1}^N \g_n Th_n\Big\n^2\Big)^\frac12 
 \le K_{2,p}^\g \Big(\E\Big\n \sum_{n=1}^N \g_n Th_n\Big\n^p\Big)^\frac1p 
\\ & \qquad  \le  K_{2,p}^\g \n T\n_{\pi_p(H,E)}\Big(\E\int_{B_{H}}\Big|\Big[
\sum_{n=1}^N 
\g_n h_n,h\Big]_H\Big|^pd\nu(h)\Big)^\frac1p 
\\ & \qquad  \le  K_{2,p}^\g K_{p,2}^\g\n T\n_{\pi_p(H,E)}\Big(\int_{B_{H}}
\Big(\sum_{n=1}^N 
|[h_n, h]_H|^2\Big)^{\frac{p}{2}}\,d\nu(h)\Big)^\frac1p 
\\ & \qquad  \le  K_{2,p}^\g K_{p,2}^\g\n T\n_{\pi_p(H,E)}\sup_{\n h\n_H\le 1}
\Big(\sum_{n=1}^N |[ h_n,h]_H|^2\Big)^{\frac{1}{2}}
\\ & \qquad  = K_{2,p}^\g K_{p,2}^\g\n T\n_{\pi_p(H,E)}.
\end{align*}
Since the finite rank operators are dense in $\pi_p(H,E)$,
this estimate implies that $T$ is $\g$-radonifying with 
$\n T\n_{\g^\infty(H,E)}\le  K_{2,p}^\g K_{p,2}^\g \n T\n_{\pi_p(H,E)}.$
Finally observe that $K_{2,p}^\g K_{p,2}^\g = \max\{ K_{2,p}^\g, K_{p,2}^\g\}$
such at least one of these numbers equals $1$.
\end{proof}

We also have a `dual' version:

\begin{proposition}\label{prop:dual} If $T\in \g(H,E)$, then
$T^*\in\pi_2(E\s,H)$ and
$$ \n T^*\n_{\pi_2(E\s,H)} \le \n T\n_{ \g(H,E)}.$$
\end{proposition}
\begin{proof} \ 
Let $(h_j)_{j\ge 1}$ be an orthonormal basis for the separable closed subspace
$({\rm ker}(T))^\perp$ of $H$. 
For all
$x_1^*,\dots, x_N^*$ in $E^*$ we have
\[\begin{aligned}
\sum_{n=1}^N \|T^* x^*_n\|^2
& = \sum_{n=1}^N  \sumj \lb T h_j, x_n^*\rb^2
  = \E \sum_{n=1}^N \Big\lb \sumj \g_j T h_j, x_n^*\Big\rb^2
\\ &  \leq   \E \Big\| \sumj \g_j T h_j\Big\|^2 
\sup_{\n x\n\le 1}\sum_{n=1}^N \lb x, x_n^*\rb^2
\\ & \leq \|T\|_{\g(H,E)}^2 \sup_{\n x^{**}\n\le 1}\sum_{n=1}^N \lb
x_n\s,x^{**}\rb^2.
\end{aligned}\]
\epf

Our next aim is to prove that, roughly speaking, a converse of Proposition
\ref{prop:dual} holds if and only 
if $E$ has type $2$, and to formulate a similar characterisation of spaces with
cotype $2$. These results are due to {\sc Chobanyan} and {\sc Tarieladze} 
\cite{ChoTar}; see also {\sc Diestel}, {\sc Jarchow}, {\sc Tonge} \cite[Chapter
12, Corollaries 12.7 and 12.21]{DJT}. 
For further refinements we refer to {\sc K\"uhn} \cite{Kuhn81}.

\begin{theorem}
For a Banach space $E$ the following two assertions are equivalent:
\ben
\item[\rm(1)] $E$ has type $2$;
\item[\rm(2)] whenever $H$ is a Hilbert space and $T\in \calL(H,E)$ satisfies
$T\s\in \pi_2(H,E)$, then
$T\in \g(H,E)$.
\een
In this situation one has
$$
\|T\|_{\g(H,E)}\leq K(E) T_2^\g(E) \|T^*\|_{\pi_2(E^*, H)},
$$
where $T_2^\g(E)$ is the Gaussian type $2$ constant of $E$.
\end{theorem}

\begin{proof}
Suppose first that $E$ has type $2$ and let $T\in\calL(H,E)$ be as stated. 
The dual space $E\s$ is 
$K$-convex by Propositions \ref{prop:K-convex-dual} and
\ref{prop:type2-Kconvex}, 
and therefore by Theorem 
\ref{thm:Kconvex-gamma-dual} we have a natural isomorphism
$ (\g(H,E\s))\s \simeq \g(H,E^{**})$ given by trace duality. 
The idea of the proof is now to show that $T$ defines an 
element of $(\g(H,E\s))\s$ via trace duality. Once we know this
it is immediate that $T\in \g(H,E)$. 

Given $S\in \g(H,E\s)$, define 
\begin{align}\label{eq:trace}
\phi_T(S):= {\rm tr}(T\s S) = \sum_{n=1}^N [T\s S h_n, h_n].
\end{align}
Since $E\s$ has cotype $2$, the implication (1)$\Rightarrow$(2) of Theorem
\ref{thm:CT2} below shows that $S$ is $2$-absolutely summing
and  
$$\n S\n_{\pi_2(H,E\s)} \le C_2^\g(E\s)\n S\n_{\g(H,E\s)}
\le T_2^\g(E)\n S\n_{\g(H,E\s)}.
$$
It follows that $T\s S$, being the composition of two $2$-absolutely
summing operators, is nuclear and therefore the sum in \eqref{eq:trace}
is absolutely convergent and 
$$
\bal  |\phi_T(S)| \le \sum_{n=1}^N |[T\s S h_n, h_n]|
& \le \n T\s\n_{\pi_2(E\s,H)}  \n S\n_{\pi_2(H,E\s)} 
\\ & \le  T_2^\g(E)\n S\n_{\g(H,E\s)} \n T\s\n_{\pi_2(E\s,H)} .
\eal
$$
This shows that $\phi_T$ is a bounded linear functional on $\g(H,E\s)$\
of norm $\n \phi_T\n \le   T_2^\g(E) \n T\s\n_{\pi_2(E\s,H)}$.
This proves the implication (1)$\Rightarrow$(2) and the norm estimate.

The proof of  the implication (2)$\Rightarrow$(1)
is based on the observation that a bounded operator
$S:F\to \ell^2$, where $F$ is a Banach space, is $2$-absolutely summing if for
all bounded operators $U:\ell^2\to F$ the composition
$SU:\ell^2\to \ell^2$ is Hilbert-Schmidt. 
To prove this,  given a sequence $\seqx$ in $F$ which satisfies 
$\sumn \lb x_n,x\s\rb^2<\infty$ for all $x\s\in
F\s$ we need to show that $\sumn \n Sx_n\n_{\ell^2}^2<\infty$.
An easy closed graph argument then shows that $S \in\pi_2(F,\ell^2)$.

Let $(u_n)_{n\ge 1}$ de the standard unit basis of
$\ell^2$ and consider the operator $U:\ell^2\to F$ defined by $Uu_n:= x_n$.
The operator $U$ is bounded; this follows from
$$
\bal \n Uh\n^2 & = \sup_{\n x\s\n\le 1} \lb Uh,x\s\rb^2
\\ & =\sup_{\n x\s\n\le 1} \sumn [h,u_n]^2 \lb Uu_n,x\s\rb^2 
\le \n h\n_{\ell^2}^2 \sup_{\n x\s\n\le 1} \sumn \lb x_n, x\s\rb^2.
\eal$$
By assumption, $SU$ is Hilbert-Schmidt, so 
$ \sumn \n SU u_n\n_{\ell^2}^2 =\sumn \n Sx_n\n_{\ell^2}^2 <\infty$ as desired,
and we conclude that $S\in \pi_2(F,\ell^2)$. 

By the closed graph theorem, there is a constant $K\ge 0$ such that
$$ \n S\n_{\pi_2(F,\ell^2)}\le 
 K\sup_{\n U\n\le 1} \n SU\n_{\calL_2(\ell^2,\ell^2)}.$$

Now assume that for all $T\in \calL(\ell^2,E)$ with $T\s\in \pi_2(E\s,\ell^2)$
we have $\g(\ell^2,E)$.
By a Baire category argument we find a constant $C\ge 0$ such that 
$\n T\n_{\g(\ell^2,E)}\le C\n T\s\n_{\pi_2(E\s,\ell^2)}$.
Let 
$x_1,\dots, x_N$ in $E$ be arbitrary and given, and define $T_N u_n = x_n$ and
$T_N u=0$ if $u\perp
u_n$ for all $n=1,\dots,N$.  Then,
\begin{align*} 
\E \Big\n \sum_{n=1}^N \g_n x_n\Big\n^2
& = \E \Big\n \sum_{n=1}^N \g_n T_N u_n\Big\n^2 
 = \n T_N\n_{\g(\ell^2,E)}^2
\\ & \le C^2 \n T_N\s\n_{\pi_2(E\s,\ell^2)}^2
 \le C^2 K^2  \sup_{\n U\n\le 1} \n T_N\s U\n_{\calL_2(\ell^2, \ell^2)}^2
\\ & = C^2K^2 \sup_{\n U\n\le 1} \n U\s T_N\n_{\calL_2(\ell^2, \ell^2)}^2
 = C^2K^2 \sup_{\n U\n\le 1}\sum_{n=1}^N \n U\s T_N u_n\n_{\ell^2}^2
\\ & = C^2K^2 \sup_{\n U\n\le 1}\sum_{n=1}^N \n U\s x_n\n_{\ell^2}^2
 \le C^2K^2 \sum_{n=1}^N \n x_n\n^2.
 \end{align*}
This shows that $E$ has type $2$ with constant $T_2^\g(E)\le CK$.
\end{proof}

\begin{theorem}\label{thm:CT2}
For a Banach space $E$ the following two assertions are equivalent:
\ben
\item[\rm(1)] $E$ has cotype $2$;
\item[\rm(2)] whenever $H$ is a Hilbert space, $T\in \g(H,E)$ implies
$T\in \pi_2(H,E)$.
\een
In this situation one has
$$
\|T\|_{\pi_2(H,E)}\le C_2^\g(E) \|T\|_{\g(H,E)},
$$
where $C_2^\g(E)$ is the Gaussian cotype $2$ constant of $E$.
\end{theorem}
\begin{proof}
(1)$\Rightarrow$(2): \ We may assume that $H$ is separable. 
Let $(h_n)_{n\ge 1}$ be an orthonormal basis for $H$,
let $(\g_n)_{n\ge 1}$ be  a Gaussian
sequence on a probability space $(\O,\F,\P)$, and 
let $(r_n')_{n\ge 1}$ be a Rademacher
sequence on another probability space $(\O',\F',\P')$. 
Fix vectors $x_1,\dots, x_N\in H$ and
define $U:H\to H$ by $Uh_n = x_n$ for $n=1,\dots,N$ and $Uh_n=0$ for $n\ge
N+1$. Then,
$$
\begin{aligned} \sum_{n=1}^N \n Tx_n\n^2
& =  \E \sum_{n=1}^N \n \g_n Tx_n\n^2
\\ & \le (C_2^\g(E))^2 \E \E'\Big\n  \sum_{n=1}^N r_n' \g_n Tx_n\Big\n^2
 = (C_2^\g(E))^2 \E\Big\n  \sum_{n\ge 1} \g_{n} T Uh_n\Big\n^2
\\ & = (C_2^\g(E))^2 \n T U\n_{\g(H,E)}^2
 \le  (C_2^\g(E))^2 \n T\n_{\g(H,E)}^2 \n U\n^2.
\end{aligned}
$$
Moreover,
$$
\begin{aligned}
\n U\n
 & = \sup_{\n h\n, \n h'\n\le 1} [Uh,h']
 = \sup_{\n h\n,\n h'\n\le 1} \sum_{n=1}^N [h,h_n] [Uh_n,h']
\\ & \le \sup_{\n h\n\le 1} \Big( \sum_{n=1}^N [h,h_n]^2\Big)^\frac12
 \sup_{\n h'\n\le 1}  \Big( \sum_{n=1}^N [Uh_n,h']^2\Big)^\frac12
 \le \sup_{\n h'\n \le 1} \Big(\sum_{n=1}^N [x_n,h']^2\Big)^\frac12.
\end{aligned}
$$
Combining the estimates we arrive at
$$
\sum_{n=1}^N \n Tx_n\n^2
\le (C_2^\g(E))^2\n T\n_{\g(H,E)}^2
\sup_{\n h'\n\le 1} \sum_{n=1}^N [ x_n,h']^2.
$$

(2)$\Rightarrow$(1): \  If $T\in \g(\ell^2,E)$ implies $T\in\pi_2(\ell^2,E)$,
then a closed graph argument produces a constant $C\ge 0$ such that
$ \n T\n_{\pi_2(\ell^2,E)} \le C   \n T\n_{\g(\ell^2,E)}$ for all $T\in
\g(\ell^2,E)$.
Now let $x_1,\dots,x_N\in E$ be arbitrary and define $T\in\g(H,E)$ by
$Tu_n:=x_n$ for $n=1,\dots,N$ and $Tu_n:=0$ for $n\ge N+1$. Then
$$
\begin{aligned}
 \sum_{n=1}^N \n x_n\n^2 
& = \sum_{n=1}^N \n Tu_n\n^2
 \le \n T\n_{\pi_2(\ell^2,E)}^2
\\ & \le C^2  \n T\n_{\g(\ell^2,E)}^2 = C^2\E\Big\n \sum_{n=1}^N \g_n
Tu_n\Big\n^2
= C^2\E\Big\n \sum_{n=1}^N \g_n x_n\Big\n^2.
\end{aligned}
$$ 
Thus $E$ has cotype $2$ with constant $C_2^\g(E)\le C.$ 
\end{proof}

\section{Miscellanea}
\label{sec:conditions}

In this final section we collect miscellaneous results given conditions for -
and examples of - $\g$-radonification.

\medskip\noindent
{\bf Hilbert sequences.} \ We have introduced $\g$-radonifying operators in
terms of their action on 
finite orthonormal systems and obtained characterisations in terms of 
summability properties on orthonormal bases. 
In this section we show that if one is only interested in {\em sufficient}
conditions for
$\g$-radonification, the role of orthonormal systems may be replaced by that of
so-called Hilbert sequences. This
provides a more flexible tool to check that certain operators
are indeed $\g$-radonifying.

Let $H$ be a Hilbert space. A sequence $h=(h_n)_{n\ge 1}$ in $H$ is said
to be a {\em Hilbert sequence} if there exists a constant $C\ge 0$ such that for
all scalars $\a_1,\dots,\a_N$,
\[ \Big\n \sum_{n=1}^N \a_n h_n \Big\n_H
\leq C \Big(\sum_{n=1}^N |\a_n|^2\Big)^{\frac12}.
\]
The infimum of all admissible constants will be denoted by 
$C_h$. 

\begin{theorem}\label{thm:Hilbert-sequence-estimate}
Let $(h_n)_{n\geq
1}$ be a Hilbert sequence in $H$.
If $T\in \g(H,E)$, then $\sum_{n\ge 1} \g_n T h_n$ converges in
$L^2(\O;E)$ and
$$
\E\Big\n \sumn \g_n T h_n \Big\n^2 \leq C_h^2 
\n T\n _{\g(H,E)}^2.
$$
\end{theorem}
\begin{proof}
Let $(\wt h_n)_{n\ge 1}$ be an
orthonormal basis for the closed linear space $H_0$ 
of $(h_n)_{n\ge 1}$. Since $(h_n)_{n\ge 1}$ is a 
Hilbert sequence there
is a unique $S\in \calL(H_0)$ such that $S \wt h_n = h_n$ for all $n\ge 1$.
Moreover, $\n S\n \le C_h$.
Indeed, for $\wt h = \sum_{n=1}^N a_n \wt h_n$ we have
\[
  \n S\wt h\n_H^2
= \Big\n\sum_{n=1}^N a_n h_n  \Big\n_H^2\leq C_h^2 \sum_{n=1}^N
    |a_n|^2
= C_h^2 \n \wt h\n_H^2,
\]
and the claim follows from this.

By the right ideal property we have $T\circ S\in \g(H_0,E)$ and
\[
\bal
\E\Big\n \sum_{n\ge 1} \g_n T h_n \Big\n^2 
& = \E\Big\n \sum_{n\ge 1}\g_n TS \wt h_n \Big\n^2
 \leq \n T\circ S\n_{\g(H_0,E)}^2\leq C_h^2 
\n T\n_{\g(H_0,E)}^2.
\eal \]
\end{proof}

A sequence is a Hilbert sequence if it is {\em almost orthogonal}:

\begin{proposition}\label{prop:working-horse}
Let $(h_n)_{n\in \Z}$ be a sequence in $H$. If there exists a function
$\phi: \N\to \R_+$ such that for all $n\geq m \in \Z$ we have $\bigl| 
[h_n,
h_m] \bigr| \leq \phi(n-m)$ and $\sum_{j\ge 0} \phi(j) <\infty$, then
$(h_n)_{n\in \Z}$ is a Hilbert sequence. 
\end{proposition}
\begin{proof}
Let $(\a_n)_{n\in \Z}$ be scalars. 
Then
\begin{align*}
   \Big\n\sum_{n=-N}^N \a_n h_n \Big\n^2
 &=   \sum_{n=-N}^N |\a_n|^2 \n h_n\n^2
       + 2 \sum_{-N\le n<m\le N} \a_n {\a_m}
       [ h_n, h_m]\\
 &\leq \phi(0) \sum_{n\in \Z} |\a_n|^2
       + 2 \sum_{n<m} |\a_n| |\a_m| \phi(n-m) \\
 &=  \phi(0) \sum_{n\in \Z} |\a_n|^2
       + 2 \sum_{j\geq 1} \phi(j) \sum_{n\in \Z} |\a_n|\, |\a_{n+j}| \\
 &\leq \Big(\phi(0) + 2  \sum_{j\geq 1} \phi(j)\Big) \sum_{n\in \Z}
       |\a_n|^2,
\end{align*}
where the last estimate follows from the Cauchy-Schwarz inequality.
\end{proof}

For some applications see {\sc Haak, van Neerven} \cite{HaaNee} and {\sc Haak,
van Neerven, Veraar} \cite{HNV}.  We continue with some explicit examples of
Hilbert sequences. The first is due to {\sc Casazza, Christensen, Kalton}
\cite{CCK}.

\begin{example}\label{ex:CCK}
Let $\phi\in L^2(\R)$ and define the sequence $(h_n)_{n\in \mathbb{Z}}$ in
$L^2(\R)$ by $h_n(t) = e^{2\pi  n i t} \phi(t)$. Let $\mathbb{T}$ be the unit
circle in $\mathbb{C}$ and define $f:\mathbb{T}\to [0,\infty]$ as
\[
f(e^{2\pi i t}) := \sum_{k\in \mathbb{Z}} |\phi(t+ k)|^2.
\]  
From
$$\bal
    \Big\n\sum_{n\in \Z} a_n h_n\Big\n^2
&= \sum_{k\in \Z} \int_k^{k+1} \Big| \sum_{n\in \Z} a_n
      e^{2\pi   i nt} \phi(t)\Big|^2 \, dt
\\ &=  \sum_{k\in \Z} \int_0^{1} \Big| \sum_{n\in \Z} a_n e^{2\pi int}
      \phi(t+k)\Big|^2 \, dt
 = \int_0^{1} \Big| \sum_{n\in \Z} a_n e^{2\pi int}\Big|^2
    f(e^{2\pi i t}) \, dt
\eal
$$
we infer that $(h_n)_{n\in \mathbb{Z}}$ is a Hilbert sequence in $L^2(\R)$ if
and only if
there exists a finite constant $B$ such that $f(e^{2\pi i t} ) \leq B$ for
almost all $t\in [0,1]$.
In this situation we have $C_h^2= {\rm ess\,sup}(f)$.
\end{example}

\begin{example}\label{ex:JZ}
Let $(\lambda_n)_{n\geq 1}$ be a sequence in $\C_+$ which is {\em properly
  spaced} in the sense that $$\inf_{m\not=n} \Big|\frac{\lambda_m -
\lambda_n}{\lambda_m + \overline\lambda_n} \Big| > 0.$$ Then the functions
$$f_n(t):= \sqrt{\Re\lambda_n}
e^{-\lambda_n t}, \quad n\ge 1,$$ define a Hilbert sequence in
$L^2(\R_+)$; see {\sc Nikol$'$ski\u{\i}} and {\sc Pavlov} \cite{NikPav} or {\sc
Jacob} and {\sc Zwart} \cite[Theorem 1, proof of (3)$\Rightarrow$(5)]{JacZwa}.
From this one easily deduces that for
 any $a>0$ and $\rho\in [0,1)$ the
functions $$f_n(t) := e^{-at + 2\pi i(n+\rho)t}, \quad n\in\Z,$$
define a Hilbert sequence in $L^2(\R_+)$. The following direct proof of this
fact is taken from {\sc Haak, van Neerven, Veraar} \cite[Example 2.5]{HNV}.

For all $t\in [0,1)$,
\[
   F(e^{2\pi i t}) = \sum_{k\in \mathbb{Z}} |f(t+ k)|^2
 = \sum_{k\ge 0} e^{-2a(t+k)}
 = \frac{e^{2a(1-t)}}{e^{2a}-1}.
\]
Now Example \ref{ex:CCK} implies the result, with constant $C_h =
{1}/{\sqrt{1-e^{-2a}}}$.
\end{example}

More on this topic can be found in {\sc Young} \cite{You}.

\medskip\noindent
{\bf Conditions on the range space.}
For certain range spaces, a complete 
characterisation of $\g$-radon\-ify\-ing
operators can be given in non-probabilistic terms.
The simplest example occurs when the range space is a Hilbert space.

If $H$ and $E$ are Hilbert spaces, we denote by $\calL_2(H,E)$ the space of all
{\em Hilbert-Schmidt operators}
from $H$ to $E$, that is, the completion of the finite rank operators 
with respect to the norm 
$$ \Big\n \sum_{n=1}^N h_n\otimes x_n\Big\n_{\calL_2(H,E)}^2 :=  \sum_{n=1}^N \n
x_n\n^2,$$ where $h_1, \dots, h_N$ are taken orthonormal in $H$.
 
\begin{proposition}[Operators into Hilbert spaces]\label{prop:HS}
If $E$ is a Hilbert space, then $T\in \g(H,E)$ if and only if
$T\in \calL_2(H,E)$, and in this case we have
$$ \n T\n_{\g(H,E)} = \n T\n_{\calL_2(H,E)}.$$
\end{proposition}
 
\begin{proof}
This follows from the identity
$$ \E \Big\n\sum_{n=1}^{N} \g_n x_n\Big\n^2 = 
 \E\sum_{m,n=1}^{N} \g_m\g_n [x_m, x_n] = \sum_{n=1}^{N} 
\n x_n\n^2.$$
\end{proof}

The next two results are taken from {\sc van Neerven, Veraar, Weis}
\cite{NVW07a}.

\begin{theorem}[Operators into $L^p(A;E)$]\label{thm:gFub}
For all $1\le p<\infty$ the mapping
$h\ot (f \ot x)\mapsto f \ot (h\ot x)$
defines an isomorphism of Banach spaces
$$\g(H,L^p(A;E))\simeq L^p(A;\g(H,E)).$$
For $p=2$ this isomorphism is isometric.
\end{theorem}

\begin{proof} 
Let $f\in L^p(A)\ot (H\ot E)$, say $f=
\sum_{m=1}^M \phi_m\otimes T_m$. By a Gram-Schmidt argument may assume that the
operators $T_m \in H\ot E$ are of the form
$\sum_{n=1}^N h_n \otimes x_{mn}$ for some fixed
orthonormal systems $\{h_1,\dots,h_N\}$ in $H$. 
Denoting by $U$ the mapping $f \ot (h\ot x)\to h\ot (f \ot x)$ 
from the
Kahane-Khintchine inequalities and Fubini's theorem
we obtain, writing $fh_n = \sum_{m=1}^M \phi_m\otimes x_{mn}$,  
\begin{align*}
 \|Uf\|_{\g(H, L^p(A; E))}
& = \Big(\E\Bigl\|\sum_{n=1}^N\g_n f h_n \Bigr\|_{L^p(A;E)}^2\Big)^\frac12
\\ & \eqsim_{p} \Big(\E\Bigl\|\sum_{n=1}^N\g_n 
f h_n \Bigr\|_{L^p(A; E)}^p\Big)^\frac1p
\\ & = \Big(\int_A \E \Big\n \sum_{n=1}^N \g_n 
fh_n\Big\n^p\,d\mu \Big)^\frac1p
\\ &  \eqsim_{p}
\Big( \int_A \Big(\E\Bigl\|\sum_{n=1}^N\g_n  f
h_n \Bigr\|^2\Bigr)^\frac{p}{2}\,d\mu\Bigr)^\frac1p
\\ &  
=\Big( \int_A \n f\n_{\g(H,E)}^p \,d\mu\Bigr)^\frac1p
 = \|f \|_{L^p(A; \g(H,E))}.
\end{align*}
The result now follows by observing that the functions $f$ of the above
form are dense in $ L^p(A;\g(H,E))$ and that their images under $U$ are dense 
in $\g(H,L^p(A;E))$.
\end{proof}

The equivalence (1)$\Leftrightarrow$(3) of the next result shows that an
operator from a Hilbert space
into an $L^p$-space is $\g$-radonifying if and only if it satisfies a 
{\em square function estimate}. The equivalence (1)$\Leftrightarrow$(2)
was noted in {\sc Brze\'zniak} and {\sc van Neerven} \cite{BrzNee03}.

\begin{proposition}[Operators into $L^p(A)$]\label{prop:sq-fc-Lp}
Let $(A,\calA)$ be a  
$\sigma$-finite measure space and let
$1\le p<\infty$. Let $(h_i)_{i\in I}$ be a maximal orthonormal system in $H$. 
For an  operator $T\in\calL(H,L^p(A))$ the following assertions are equivalent:
 
\begin{enumerate}
\item[\rm(1)] $T\in \g(H,L^p(A))$;
\item[\rm(2)] there exists a function $f\in L^p(A;H)$ such that $Th = [f,h]$ for
all $h\in H$.
\item[\rm(3)] 
$\big(\sum_{i\in I} |T h_i|^2\big)^{\frac12}$ is summable in $L^p(A)$.
\end{enumerate}
In this case we have
$$\|T\|_{\g(H,L^p(A))} \eqsim_p \Big\|\Big(\sum_{i\in I} |T
h_i|^2\Big)^{\frac12}\Big\|.
$$
\end{proposition}

\begin{proof}
The equivalence (1)$\Leftrightarrow$(2) is a special case of Theorem
\ref{thm:gFub}. To prove the equivalence (1)$\Leftrightarrow$(3)
we apply the identity $$ \E\Big|\sum_{n=1}^N c_{n} \g_{n}\Big|^2 = \sum_{n=1}^N
|c_{n}|^2 $$
with $c_n = f_n(\xi)$, $\xi\in A$. Combined with the Khintchine
inequality, Fubini's theorem, and finally
the Kahane-Khintchine inequality in $L^p(A)$,
for all $f_1,\dots,f_N\in L^p(A)$ we obtain
\[
\bal
 \Big\|\Big(\sum_{n=1}^N |f_n|^2\Big)^{\frac12}\Big\|_{p} 
 &  = \Big\|\Big(\E\Big|\sum_{n=1}^N \g_{n}
f_n\Big|^2\Big)^{\frac12}\Big\|_{p}
\eqsim_p\Big\|\Big(\E\Big|\sum_{n=1}^N \g_{n}
f_n\Big|^p\Big)^{\frac1p}\Big\|_{p}
\\ &  = \Big(\E\Big\n\sum_{n=1}^N \g_{n} f_n\Big\|_{L^p(A)}^p\Big)^{\frac1p}
  \eqsim_p 
\Big(\E\Big\|\sum_{n=1}^N \g_{n} f_n\Big\|_{p}^2\Big)^{\frac12}.
\eal\]
The equivalence as well as the final two-sided estimate now follow by
taking $f_n := Th_{i_n}$
and invoking Theorem \ref{thm:ONS}.
\end{proof}

Here is a neat application, which is well-known when $p=2$.

\begin{corollary}\label{cor:Linfty}
Let $(A,\calA)$ be a finite measure space. 
For all $T\in\calL(H,L^\infty(A))$ and $1\le p<\infty$ we have  
$T\in \g(H,L^p(A))$ and
$$\n T\n_{\g(H,L^p(A))} \lesssim_p \n T\n_{\calL(H,L^\infty(A))}.$$
\end{corollary}
\begin{proof}
Let $(h_i)_{i\in I}$ be a maximal orthonormal system in $H$.
For any choice of finitely many indices $i_1,\dots,i_N\in I$ and $c \in
\ell_N^2$,
for $\mu$-almost all $\xi\in A$ we have
$$\bal
 \Big|\sum_{n=1}^N c_n (Th_{i_n})(\xi) \Big|
 & \le \Big\n\sum_{n=1}^N c_n Th_{i_n}\Big\n_{\infty} 
 \\ & \le \n T\n_{\calL(H,L^\infty(A))}  
\Big\n  \sum_{n=1}^N c_n h_{i_n}\Big\n
=  \n T\n_{\calL(H,L^\infty(A))} \n c\n.
\eal
$$
Taking the supremum over a countable dense set in the unit ball of 
$\R^d$ we obtain the following
estimate, valid for $\mu$-almost all $\xi\in A$:
$$ \Big(\sum_{n=1}^N |(Th_{i_n})(\xi)|^2 \Big)^\frac12 \le \n
T\n_{\calL(H,L^\infty(A))}.
$$
Now apply Proposition \ref{prop:sq-fc-Lp}.
\end{proof}

\medskip\noindent
{\bf New $\g$-radonifying operators from old.} \ 
The next proposition is a minor extension of a result of {\sc Kalton} and {\sc
Weis} \cite{KalWei07}. 

\begin{proposition}\label{prop:C1-derivative}
Let $(a,b)$ be an interval and $\phi: (a,b) \to \g(H,E)$ be continuously
differentiable with $$\int_a^b (s-a)^{\frac12}
\n{\phi'(s)}\n_{\g(H,E)}\, ds < \infty.$$ 
Define $T_\phi: L^2(a,b;H)\to E$ by
$$ T_\phi f:= \int_a^b \phi(t)f(t)\,dt.$$
Then $T_\phi \in \gamma (L^2(a,b;H),E)$ and
\[ \n T_\phi\n_{\gamma (L^2(a,b;H),E)} \leq
(b-a)^\frac12\n{\phi(b)}\n_{\g(H,E)}+\int_a^b(s-a)^\frac12
 \n \phi'(s)\n _{\g(H,E)}\, ds. 
\]
\end{proposition}

\begin{proof}For notational simplicity we shall identify $\g(H,E)$-valued
functions on $(a,b)$ with the induced operators in $\calL(L^2(a,b;H),E)$.

The integrability condition implies that $\phi'$ is integrable on every interval
$(a',b)$ with $a<a'<b$. 
Put $\psi(s,t) := \one_{(t,b)}(s)\phi'(s)$ for $s,t \in (a,b)$. Then, by the
observations just made, 
$$\phi(t) = \phi(b) - \int_a^b\psi(s,t)\, ds$$ for all $t\in (a,b)$. 
By Example \ref{ex:f-ot-T}, for all $s\in (a,b)$ the function $t\mapsto
\psi(s,t) = \one_{(t,b)}(s)\phi'(s) = \one_{(a,s)}(t)\phi'(s)$ belongs to
$\gamma (L^2(a,b;H),E)$ with norm
\[ \n \one_{(\cdot,b)}(s)\phi'(s) \n_{\gamma (L^2(a,b;H),E)} = \n
\one_{(a,s)}\n_2\n \phi'(s)\n_{\gamma(H,E)} = (s-a)^\frac12\n
\phi'(s)\n_{\gamma(H,E)}. \]
It follows that the $\gamma (L^2(a,b;H),E)$-valued function $s \mapsto \psi(s,
\cdot )$ is
Bochner integrable. Identifying the operator $\phi(b)\in \gamma(H,E)$ with the
constant function $\one_{(a,b)}\phi(b)\in \gamma(L^2(a,b;H),E)$, we find that
$\phi \in \gamma (L^2(a,b;H),E)$ and
\begin{align*}
\n{\phi}\n_{\gamma (L^2(a,b;H),E)} &  \leq 
(b-a)^\frac12\n{\phi(b)}\n_{\g(H,E)} + \int_a^b\n{\psi(s,\cdot )}\n_{\gamma
(L^2(a,b;H),E)}\, ds\\
& = (b-a)^\frac12\n{\phi(b)}\n_{\g(H,E)} +
\int_a^b(s-a)^\frac12\n{\phi'(s)}\n_{\g(H,E)}\, ds.
\end{align*}
\end{proof}

The next result is due to {\sc Chevet} \cite{Che77}; see also {\sc Carmona}
\cite{Car}. We state it without proof; a fuller discussion would require a
discussion of injective tensor norms (see {\sc Diestel} and {\sc Uhl}
\cite{DieUhl} for an introduction to this topic). 

\begin{proposition} For all $T_1\in \g(H_1,E_1)$ and $T_2\in \g(H_2,E_2)$ we
have
$$ T_1\ot T_2 \in \g(H_1\wh\ot H_2, E_1\wh\ot_\e E_2),$$ 
where $H\wh\ot H'$ denotes the Hilbert space completion of $H\ot H'$ and
$E_1\wh\ot_\e E_2$ denotes the injective tensor product of $E_1$ and $E_2$. 
\end{proposition}

In view of the identity $C[0,1]\wh\ot_\e E = C([0,1];E)$, the interest of this
example lies in the special case where one of the operators is the indefinite
integral from $L^2(0,1)$ to $C[0,1]$ (see Proposition \ref{prop:WM}). 

The final result of this subsection is a Gaussian version of the Fubini theorem.
For its statement we need to introduce another Banach space property. Let
$(\g_m')_{m\ge 1}$ and $(\g_n'')_{n\ge 1}$
be Gaussian sequences on probability spaces $(\O',\F',\P')$ and 
$(\O'',\F'',\P'')$, and let $(\g_{mn})_{m,n\ge 1}$
be a doubly indexed Gaussian sequence on a probability space $(\O,\F,\P)$.

It is easy to check that $(\g_m'\g_n'')_{m,n\ge 1}$ is not
a Gaussian sequence. The following definition singles out a class of Banach
spaces in which it is possible to compare double Gaussian sums with single
Gaussian sums.

\begin{definition}\label{def:alpha}
A Banach space $E$ is said to have {\em property $(\a)$} if there exists a
constant $0<C<\infty$ such that for all finite sequences $(x_{mn})_{1\le m\le
M,\, 1\le n\le N}$ in $E$ we have
$$ \frac1{C^2}\E \Big\n \sum_{m=1}^M \sum_{n=1}^N \g_{mn} x_{mn}\Big\n^2
\!\le \E' \E'' \Big\n \sum_{m=1}^M \sum_{n=1}^N \g_m' \g_n'' x_{mn}\Big\n^2
\!\le C^2 \E \Big\n \sum_{m=1}^M \sum_{n=1}^N\g_{mn} x_{mn}\Big\n^2.
$$
\end{definition}
In an equivalent formulation, this property was introduced by {\sc Pisier}
\cite{Pis}.
The least possible constant $C$ is
called the {\em property $(\a)$ constant} of $E$, notation $\a(E)$.

Let $1\le p<\infty$. From 
\begin{align*}
 \E' \E'' \Bigl\n \sum_{m=1}^M \sum_{n=1}^N \g_m' \g_n'' y_{mn} \Bigr\n
 & \le\Big( \E' \E'' \Bigl\n \sum_{m=1}^M \sum_{n=1}^N \g_m'
\g_n'' y_{mn} \Bigr\n^p\Big)^\frac1p
\\ & \le K_{p,1}^\g \Big(\E' \Big(\E'' \Bigl\n \sum_{m=1}^M \sum_{n=1}^N \g_m'
\g_n'' y_{mn}
  \Bigr\n\Big)^p\Big)^\frac1p 
\\ & = K_{p,1}^\g\Bigg\n \E'' \Bigl\n \sum_{m=1}^M \sum_{n=1}^N  \g_m' \g_n''
y_{mn}
  \Bigr\n\, \Bigg\n_{L^p(\O')}
\\ & \le K_{p,1}^\g\E''\Bigl\n \sum_{m=1}^M \sum_{n=1}^N \g_m' \g_n'' y_{mn}
  \Bigr\n_{L^p(\O';E)}
\\ & \le (K_{p,1}^\g)^2\E''\Bigl\n \sum_{m=1}^M \sum_{n=1}^N \g_m' \g_n'' y_{mn}
  \Bigr\n_{L^1(\O';E)}
\\ & = (K_{p,1}^\g)^2 \E' \E'' \Bigl\n \sum_{m=1}^M \sum_{n=1}^N \g_m' \g_n''
y_{mn} \Bigr\n
\end{align*}
and another application of the Kahane-Khintchine inequalities (in order to prove
similar estimates for 
the sums $\sum_{m=1}^M\sum_{n=1}^N  \g_{mn} x_{mn}$), we see that the moments of
order $2$ in the Definition \ref{def:alpha}
may be replaced by moments of any order $p$. The resulting constants
will be denoted by $\a_p(E)$. Thus, $\a(E)=\a_2(E)$.

\begin{example}\label{ex:Hilbert-alpha}
Every Hilbert space $H$ has property $(\a)$, with constant $\a(H)=1$. This is
clear by writing out
the square norms as inner products.
\end{example}

\begin{example}\label{ex:Lp-alpha} Let $(A,\calA,\mu)$ be a $\sigma$ measure
space 
and let $1\le p<\infty$. The space
$L^p(A)$ has property $(\a)$, and more generally if $E$ has property $(\a)$
then $L^p(A;E)$ has property $(\a)$, with constant
$$ \a_p(L^p(A;E)) = \a_p(E).$$
Indeed,
for $f_{mn}\in L^p(A;E)$, $m=1,\dots,M$, $n=1,\dots,N$, we have 
\begin{align*}
 \E \Bigl\n \sum_{m=1}^M \sum_{n=1}^N \g_{mn} f_{mn}
\Bigr\n_{L^p(A;E)}^p
 & =
\int_A \E \Big\n \sum_{m=1}^M \sum_{n=1}^N \g_{mn} 
 f_{mn}(\xi)\Big\n^p\,d\mu(\xi) 
\\ &   \le \a_p^p(E) \int_A \E'\E''\Big\n
\sum_{m=1}^M \sum_{n=1}^N \g_m'\g_n''
 f_{mn}(\xi)\Big\n^p d\mu(\xi) 
\\ & = \a_p^p(E) \E' \E'' \Bigl\n \sum_{m=1}^M \sum_{n=1}^N \g_m' \g_n'' f_{mn}
\Bigr\n^p.
\end{align*}
The other bound is proved in the same way.
This gives $ \a_p(L^p(A;E)) \le \a_p(E)$; the opposite inequality
is trivial.
\end{example} 

The next result is due to {\sc Kalton} and {\sc Weis} \cite{KalWei07}. For
further results and refinements see {\sc van Neerven} and {\sc Weis}
\cite{NeeWei07}.

\begin{proposition}[$\g$-Fubini theorem]
Let $E$ have property $(\a)$.
Then the mapping
$h\otimes (h'\otimes x)\mapsto (h\ot h')\ot x$
extends uniquely to an isomorphism of Banach spaces
$$ \g(H,\g(H',E))\simeq \g(H\widehat\ot H',E).$$
\end{proposition}
\begin{proof}
For elements in the algebraic tensor products, the equivalence of norms is
merely a restatement of the definition of property $(\a)$. The general result
follows from it by approximation.
\end{proof}

\medskip\noindent
{\bf Entropy numbers.}
Following {\sc Pietsch} \cite[Chapter 12]{Pie-OI}, the {\em entropy numbers}
$e_n(T)$ of 
a bounded operator $T\in \calL(E,F)$ are defined as the infimum of all $\e>0$
such that there are $x_1,\dots,x_{2{n-1}}\in T(B_E)$ such that
$$ T(B_E)\subseteq \bigcup_{j=1}^{2^{n-1}} (x_j + \e B_F).$$
Here $B_E$ and $B_F$ denote the closed unit balls of $E$ and $F$.
Note that $T$ is compact if and only if $\limn e_n(T) = 0.$
Thus the entropy numbers $e_n(T)$ 
measure the degree of compactness of an operator $T$.

The following result is due to {\sc K\"uhn} \cite{Kuhn82}. Parts (1) and (2) of
can be viewed as a reformulation, in operator theoretical language, of  a
classical result due to Dudley \cite{Dud} and the Gaussian minoration principle
due to {\sc Sudakov} \cite{Sud}, respectively.

\begin{theorem}\label{thm:entropy} Let $T\in \calL(H,E)$ be a bounded operator.
\begin{enumerate}
\item[\rm(1)]
If $\sum_{n=1}^\infty n^{-\frac12} e_n(T^*)<\infty$, then $T\in \g(H,E)$;
\item[\rm(2)]
If $T\in \g(H,E)$, then $\sup_{n\ge 1} n^{\frac12} e_n(T^*)<\infty$.
\end{enumerate}
\end{theorem}

If fact one has the following quantitative version of part (2): there exists an
absolute constant $C$ such that for all Hilbert spaces $H$, Banach spaces $E$,
and operators $T\in \g(H,E)$ one has
$$\sup_{n\ge 1} \ n^{\frac12} e_n(T^*)\le C\n T\n_{\g(H,E)}.$$
In combination with a result of {\sc Tomczak-Jaegermann} \cite{TomJae87} to the
effect that for any compact operator $T\in \calL(\ell^2,E)$ one has
$$ \frac1{32} e_n(T\s)\le e_n(T)\le 32 e_n(T\s),$$
this yields (recalling that $\g$-radonifying operators are supported on a
separable closed subspace, see \eqref{eq:sep-supp}) the inequality
$$\sup_{n\ge 1} \ n^{\frac12} e_n(T)\le C\n T\n_{\g(H,E)}$$  
for some absolute constant $C$. See 
{\sc Cobos} and {\sc K\"uhn} \cite{CobKuhn} and  
{\sc K\"uhn} and {\sc Schonbeck} \cite{KuhnSch}, where these results are applied
to obtain estimates for the entropy numbers
of certain diagonal operators between Banach sequence spaces.

{\sc K\"uhn} \cite{Kuhn82} also showed that 
Theorem \ref{thm:entropy} can be improved for Banach spaces with (co)type $2$:

\begin{theorem} Let $T\in \calL(H,E)$ be a bounded operator.
\begin{enumerate}
\item[\rm(1)]
If $E$ has type $2$ and $(\sum_{n=1}^\infty (e_n(T^*))^2)^\frac12 <\infty$, then
$T\in \g(H,E)$;
\item[\rm(2)]
If $E$ has cotype $2$ and $T\in \g(H,E)$, then $(\sum_{n=1}^\infty 
(e_n(T^*))^2)^\frac12 <\infty$.
\end{enumerate}
\end{theorem}

It appears to be an open problem whether these properties characterise 
spaces with type $2$ (cotype $2$) and whether they can be extended to spaces of
type $p$ (cotype $q$).

\medskip\noindent
{\bf The indefinite integral.} \ 
The final example is a reformulation of {\sc Wiener}'s classical result on the
existence of the existence of Brownian motions. The proof presented here is due
to {\sc Ciesielski} \cite{Cie}.

\begin{proposition}[Indefinite integration]\label{prop:WM}
The operator $I: L^2(0,1)\to C[0,1]$ defined by
$$ (If)(t):= \int_0^t f(s)\,ds, \qquad f\in L^2(0,1), \ t\in [0,1],$$ 
is $\g$-radonifying. 
\end{proposition}
 
The proof is based on the following simple lemma (which is related to the
estimates in Example \ref{ex:LP}).

\begin{lemma}\label{lem:log} For any Gaussian sequence $(\g_n)_{n\ge 1}$,
$$ \limsup_{N\to\infty} \sum_{n=1}^N \frac{|\g_n|}{\sqrt{2
\log(n+1)}} \le 1.$$
\end{lemma}
\begin{proof}
For all $t\ge 1$, 
$$
\P\{|\g_n| > t\} 
 = \frac{2}{\sqrt{2\pi}}\int_t^\infty {e^{-\frac12 u^2}du}
 \leq \frac{2}{t\sqrt{2\pi}}\int_t^\infty{ue^{-\frac12 u^2}du} =
\frac{2}{t\sqrt{2\pi}}e^{-\frac12t^2}.
$$
Fix $\alpha > 1$ arbitrarily. For all
$n \geq 1$ we have $2\alpha \log (n+1) \geq 1$ and therefore
$$
\P\big\{|\g_n| \geq \sqrt{2\alpha \log (n+1)}\big\} 
\leq  \sqrt{2/\pi}\,(n+1)^{-\alpha}.
$$
The Borel-Cantelli lemma now implies that almost surely
$|\g_n| \geq \sqrt{2\alpha \log (n+1)}$ for at most finitely many $n\ge 1$.
\end{proof}

Let $(\chi_n)_{n\ge 1}$ be the $L^2$-normalised Haar functions
on $(0,1)$, which are defined by $h_1 \equiv 1$ and  
$\phi_n := \chi^{jk}$ for $n\ge 1$, where $n=2^j+k$ with 
$j = 0,1,\dots$ and $ k = 0,\dots, 2^j-1$, and
$$ \chi^{jk} = 2^{j/2} \one_{\big(\tfrac{k}{2^{j}},\tfrac{k+1/2}{2^{j}}\big)}
- 2^{j/2} \one_{\big(\tfrac{k+1/2}{2^{j}},\tfrac{k+1}{2^{j}}\big)}.$$
Note that the functions $\chi^{jk}$ are supported on the interval
$[\tfrac{k-1}{2^{j}}, \tfrac{k}{2^{j}}]$.

\bpf[Proof of Proposition \ref{prop:WM}]
It suffices to prove that the sum 
$$\sum_{j\ge 0} \sum_{k=1}^{2^j} {\g_{jk}
I\chi^{jk}(t)}, \qquad t\in [0,1],
$$
converges uniformly on $[0,1]$ almost surely. 

Fixing $j\ge 0$, for all $t\in [0,1]$ we have
$I\chi^{jk}(t) = 0$ for all but at most one $k\in \{1,\dots,2^j\}$,
and for this $k$ we have $$0\le I\chi^{jk}(t) \le 2^{-j/2-1}, \qquad t\in
[0,1].$$
Using this, 
for all $j_0\ge 0$ we obtain the following estimate, uniformly in $t\in [0,1]$:
$$
\bal
\sum_{j\ge j_0}\sum_{k=1}^{2^j} |\g_{jk}(\om)| I\chi^{jk}(t) 
& \leq C(\om)  \sum_{j\ge j_0}\sum_{k=1}^{2^j}\! \sqrt{\log (2^j+k)}
I\chi^{jk}(t)
\\ &  \leq C(\om)\sum_{j=j_0}^{\infty}\sum_{k=1}^{2^j}\!
\sqrt{j+1} I\chi^{jk}(t) 
 = C(\om)\sum_{j=j_0}^{\infty}2^{-j/2-1}\sqrt{j+1} ,
\eal
$$
where $$C(\om):= \sup_{n\ge 1}\, \frac{|\g_n(\om)|}{\sqrt{\log (n+1)}}$$
is finite almost surely by Lemma \ref{lem:log}. 
This proves the result.
\epf

It is straightforward to show that $Q:= I\circ I\s$ is given by
$$ (Q \mu)(t) = \int_0^1 s\wedge t \,d\mu(s), \qquad
\mu\in M[0,1].$$
Here $M[0,1] = (C[0,1])\s$ is the space of all bounded Borel measures
$\mu$ on $[0,1]$.
The unique Gaussian measure on $C[0,1]$
with covariance operator $Q$ is called the {\em Wiener
measure}.

Refining the proof of Proposition \ref{prop:WM}, one can prove that the
indefinite integral is  $\g$-radon\-ifying from $L^2(0,1)$ into the H\"older
space $C^\a[0,1]$ for $0\le \a<\frac12$; this reflects the fact that the paths
of a Brownian motion are $C^\a$-continuous for all $0\le \a<\frac12$.
Alternatively, this can be deduced from the Sobolev embedding theorem combined
with fact that the indefinite integral is $\g$-radonifying from $L^2(0,1)$ into
the Sobolev space $H^{\a,p}(0,1)$ for all $2<p<\infty$ and $\a\in
(\frac1p,\frac12)$; see {\sc Brze\'zniak} \cite{Brz96}.

Concerning the critical exponent $\a=\frac12$, it is known that the paths of a
Brownian motion $B$ belong to the Besov space $B_{p,\infty}^\frac12(0,1)$ for
all $1\le p<\infty$
and there is a strictly positive constant $C>0$ such that
\begin{equation}\label{eq:lower-Besov}
\P\Big\{\n B\n_{B_{p,\infty}^\frac12(0,1)} \ge C\Big\} = 1.
\end{equation}
see {\sc Ciesielski} \cite{Cie91, Cie93}, {\sc Ciesielski, Kerkyacharian,
Roynette} \cite{CKR}, {\sc Roy\-nette} \cite{Roy}, {\sc Hyt\"onen} and {\sc
Veraar} \cite{HytVer08} for a discussion of this result and further refinements.
As a consequence of this inequality one obtains the somewhat surprising fact
that the indefinite integral fails to be $\g$-radonifying from $L^2(0,1)$ into $
B_{p,\infty}^\frac12(0,1)$; the point is that \eqref{eq:lower-Besov} prevents
$B$ from being a strongly measurable (i.e. Radon) Gaussian random variable. 
A similar phenomenon in $\ell^\infty$ had been discovered previously by {\sc
Fremlin} and {\sc Talagrand} \cite{FreTal}.

\medskip\noindent
{\bf An application to stochastic Cauchy problems} \
In this section we shall briefly sketch how the theory of $\g$-radonifying
operators enters naturally in the study of stochastic abstract Cauchy problems
driven by an isonormal process. For unexplained terminology we refer to {\sc
Engel} and {\sc Nagel}  \cite{EngNag} and {\sc Pazy} \cite{Paz} (for the theory
of semigroups of operators) and {\sc van
Neerven} and {\sc Weis} \cite{NeeWei05a} (for a discussion of stochastic Cauchy
problems).   

Suppose $A$ is the infinitesimal generator of a strongly continuous semigroup
$S=(S(t))_{t\ge 0}$ of
bounded linear operators  on a Banach space $E$, let $W_H$ be
an $L^2(\R_+;H)$-isonormal process, and let $B\in\calL(H,E)$ be a bounded linear
operator.
Building on previous work of {\sc Da Prato} and {\sc Zabczyk} \cite{DaPZab} and
{\sc van Neerven} and {\sc Brze\'zniak} \cite{BrzNee00}, it has been shown in
{\sc van Neerven} and {\sc Weis} \cite{NeeWei05a} that the linear stochastic
Cauchy problem
\begin{align*}
 dU(t) &= AU(t)\,dt + B\,dW_H(t), \quad t\ge 0, \\
 U(0)  &= u_0,
\end{align*}
admits a unique weak solution $U$ if and only if for some (and then for all)
$T>0$ the bounded operator $R_T: L^2(0,T;H)\to E$ given by
$$ R_T f := \int_0^T S(t)Bf(t)\,dt$$
is $\g$-radonifying from $L^2(0,T;H)$ to $E$. 
Here we give two sufficient condition for this to happen.

\begin{proposition} Each of the following two conditions imply that $R_T$ is
$\g$-radonifying: 
\begin{enumerate}
 \item[\rm(1)] $E$ has type $2$ and $B\in\g(H,E)$;
 \item[\rm(2)] $S$ is analytic and $B\in\g(H,E)$.
\end{enumerate}
\end{proposition}
\begin{proof}
(1): \  By the strong continuity and Corollary \ref{cor:g-cont} the $\g(H,E)$-valued 
function $t\mapsto S(t)B$ is continuous on $[0,T]$. In particular it belongs to
$L^2(0,T;\g(H,E))$ and therefore, by Theorem \ref{thm:type2}, the induced operator $R_T$ 
belongs to $\g(L^2(0,T;H),E))$. 

(2): By the analyticity of $S$ the $\g(H,E)$-valued function $t\mapsto S(t)B$ is continuously differentiable
 on $(0,T)$ and
$$ \int_0^T t^\frac12 \n S'(t)B\n_{\g(H,E)}\,dt\le C_T \int_0^T t^{-\frac12} \n B\n_{\g(H,E)}\,dt
\le 2C_T T^\frac12 \n B\n_{\g(H,E)}.  
$$
where we used the analyticity of $S$ to estimate $\n S'(t)\n = \n AS(t)\n \le C_T t^{-1}$
for $t\in (0,T)$. Now Proposition \ref{prop:C1-derivative} implies that $R_T$ 
belongs to $\g(L^2(0,T;H),E))$. 
\end{proof}

\medskip
\noindent
{\em Acknowledgment} -- It is a pleasure to thank David Applebaum and Thomas K\"uhn for useful comments on a previous draft of this paper, and the anonymous referee for providing many detailed corrections. 
 
\bibliographystyle{ams-pln}
\bibliography{canberra}

\end{document}